\documentclass[11pt,oneside,dvipsnames]{amsart}
\usepackage{stackrel}
\usepackage[french, english]{babel}
\usepackage{a4wide}
\usepackage[utf8,latin1]{inputenc}
\usepackage{times}
\usepackage{adjustbox}
\usepackage{float}
\usepackage[cyr]{aeguill}
\usepackage{amsmath,amscd}
\usepackage{amssymb}
\usepackage{mathrsfs}
\usepackage{amsthm}
\usepackage{geometry}
\usepackage{graphicx}
\usepackage{epstopdf}
\usepackage{amsfonts}
\usepackage{amsmath}
\usepackage{amsmath,mathtools}
\usepackage{amsmath,graphicx}
\usepackage{amscd}
\usepackage{dsfont}
\usepackage{latexsym}
\usepackage{verbatim}
\usepackage{comment}
\usepackage{tikz-cd}
\usetikzlibrary{matrix,arrows}
\usepackage{tikz}
\usetikzlibrary{tikzmark}
\usepackage{tensor}
\usepackage{leftidx}
\usepackage[retainorgcmds]{IEEEtrantools}
\usepackage[title]{appendix}
\usepackage[mathscr]{euscript}
\usepackage[colorinlistoftodos]{todonotes}
\usepackage{bbm}
\usepackage{url}
\usepackage{upgreek}
\usepackage{mathtools, nccmath}
\usepackage[utf8]{inputenc}
\usepackage{multirow}
\usepackage{rotating}
\usepackage{marginnote}
\usepackage{hyperref} 
\usepackage{nameref}
\usepackage{ytableau}
\usepackage{youngtab}

\usepackage[all]{xy}
\usepackage[T1]{fontenc}

\usepackage{hyperref}
\usepackage{array}
\usepackage{tabularx}
\usepackage{yfonts}
\usepackage{setspace}  

\makeatletter
\renewcommand\subsubsection{\@startsection{subsubsection}{3}{\z@}%
                                     {-3.25ex\@plus -1ex \@minus -.2ex}%
                                     {-0.5em}% <---changed
                                     {\normalfont\normalsize\bfseries}}
\makeatother

\theoremstyle{plain}
\newtheorem{Thm}{Theorem}
\newtheorem{Prop}[Thm]{Proposition}
\newtheorem{Lem}[Thm]{Lemma}
\newtheorem{Cor}[Thm]{Corollary}
\newtheorem{Con}[Thm]{Conjecture}

\theoremstyle{definition}

\newtheorem{Nota}[Thm]{Notation}
\newtheorem{Rem}[Thm]{Remark}

\newtheorem{Defn}[Thm]{Definition}

\newtheoremstyle{named}{}{}{}{}{\bfseries}{.}{.5em}{\thmnote{#3}#1}
\theoremstyle{named}

\DeclareMathOperator{\ad}{ad}

\DeclareMathOperator{\Aut}{Aut}

\DeclareMathOperator{\ch}{char}
\DeclareMathOperator{\Ch}{Ch}
\DeclareMathOperator{\sCh}{\mathbf{Ch}}

\DeclareMathOperator{\diag}{diag}

\DeclareMathOperator{\Exp}{Exp}
\DeclareMathOperator{\Gal}{Gal}

\DeclareMathOperator{\GL}{GL}
\DeclareMathOperator{\Gr}{Gr}

\DeclareMathOperator{\Hom}{Hom}
\DeclareMathOperator{\ii}{\mathfrak{i}}
\DeclareMathOperator{\IC}{\textbf{IC}}
\DeclareMathOperator{\Id}{Id}

\DeclareMathOperator{\Irr}{Irr}

\DeclareMathOperator{\isom}{\!\!\smash{\begin{array}{c}\sim\\[-1em]
\rightarrow\end{array}}\!\!}
\DeclareMathOperator{\lisom}{\!\!\smash{\begin{array}{c}\sim\\[-1em]
\longrightarrow\end{array}}\!\!}

\DeclareMathOperator{\lisomLR}{\!\!\smash{\begin{array}{c}\sim\\[-1em]
\longleftrightarrow\end{array}}\!\!}

\DeclareMathOperator{\kk}{\mathds{k}}
\DeclareMathOperator{\ladic}{\bar{\mathbb{Q}}_{\ell}}

\DeclareMathOperator{\Log}{Log}

\DeclareMathOperator{\Ort}{O}
\DeclareMathOperator{\PGL}{PGL}

\DeclareMathOperator{\reg}{reg}
\DeclareMathOperator{\rk}{rk}

\DeclareMathOperator{\Rep}{Rep}
\DeclareMathOperator{\sRep}{\mathbf{Rep}}

\DeclareMathOperator{\sgn}{sgn}

\DeclareMathOperator{\SL}{SL}
\DeclareMathOperator{\SO}{SO}
\DeclareMathOperator{\Sp}{Sp}
\DeclareMathOperator{\Spec}{Spec}

\DeclareMathOperator{\Stab}{Stab}
\DeclareMathOperator{\Sym}{Sym}
\DeclareMathOperator{\SymF}{\mathbf{Sym}}

\DeclareMathOperator{\Uni}{U}
\DeclareMathOperator{\X}{\mathbf{x}}
\DeclareMathOperator{\Xo}{\mathbf{x}^{(0)}}
\DeclareMathOperator{\Xl}{\mathbf{x}^{(1)}}
\DeclareMathOperator{\Xk}{\mathbf{x}^{(k)}}
\DeclareMathOperator{\Y}{\mathbf{y}}

\DeclareMathOperator{\Z}{\mathbf{z}}
\DeclareMathOperator{\zz}{\mathfrak{z}}

\DeclareMathOperator{\lb}{\!<\!}
\DeclareMathOperator{\rb}{\!>\!}
\DeclareMathOperator{\ds}{\!/\mkern-2mu/\mkern-2mu}

\newcommand*\circled[1]{
  \tikz[baseline=(char.base)]\node[shape=circle,draw,inner sep=0.2pt,font=\tiny,minimum size=8pt] (char) {#1};}

\newcommand{\hooklongrightarrow}{\lhook\joinrel\longrightarrow}

\makeatletter
\newcommand*{\rom}[1]{\expandafter\@slowromancap\romannumeral #1@}
\makeatother

\newcommand{\rn}[2]{%% "rn": "remember node"
    \tikz[remember picture,baseline=(#1.base)]\node [inner sep=0] (#1) {$#2$};%
}

\renewcommand{\descriptionlabel}[1]{\hspace\labelsep\upshape\bfseries #1.}
\makeatletter
\let\orgdescriptionlabel\descriptionlabel
\renewcommand*{\descriptionlabel}[1]{%
  \let\orglabel\label
  \let\label\@gobble
  \phantomsection
  \edef\@currentlabel{#1}%
  \let\label\orglabel
  \orgdescriptionlabel{#1}%
}
\makeatother

\usepackage{subfig}	 % \chaptionof

\title{E-Polynomials of Generic $\mathbf{\GL_n\rtimes\lb\sigma\rb}~$-Character Varieties: Branched Case}
\author{Cheng Shu}
\usepackage{pxfonts}
\usepackage{indentfirst}
\usepackage{setspace}  
\usepackage{hyperref}
\colorlet{ivory}{Apricot!30!}
\colorlet{space}{black!85!}
\definecolor{bgc}{RGB}{29, 44, 46}
\definecolor{txt}{RGB}{223, 222, 189}
\definecolor{cmd}{RGB}{206, 151, 88}
\begin{document}
\let\bs\boldsymbol
\maketitle
\begin{abstract}
For any branched double covering of compact Riemann surfaces, we consider the associated character varieties that are unitary in the global sense, which we call $\GL_n\rtimes\lb\sigma\rb~$-character varieties. We restrict the monodromies around the branch points to generic semi-simple conjugacy classes contained in $\GL_n\sigma$, and compute the E-polynomials of these character varieties using the character table of $\GL_n(q)\rtimes\lb\sigma\rb$. The result is expressed as the inner product of certain symmetric functions associated to the wreath product $(\mathbb{Z}/2\mathbb{Z})^N\rtimes\mathfrak{S}_N$. We are then led to a conjectural formula for the mixed Hodge polynomial, which involves (modified) Macdonald polynomials and wreath Macdonald polynomials. 
\end{abstract}

\tableofcontents
\addtocontents{toc}{\protect\setcounter{tocdepth}{-1}}
%%% Usually after Introduction %%%
\setcounter{tocdepth}{1}
\numberwithin{Thm}{section}
\numberwithin{equation}{subsection}
\addtocontents{toc}{\protect\setcounter{tocdepth}{1}}
%%%%%%%%%%%%%%%%%%%
\section{Introduction}
\subsection{Mixed Hodge Polynomial}\hfill

The E-polynomial of a complex algebraic variety encodes a part of the dimensional information of its mixed Hodge structure. Mirror partners are expected to have the same E-polynomial, with the caveat that one must suitably modify its definition to include the contribution of singularities. Therefore, it is an important invariant in the study of mirror symmetry. See \cite{HT}. 

According to Deligne (\cite{D1}, \cite{D2}), to any possibly singular and not necessarily projective complex algebraic variety $X$, we can associate a mixed Hodge structure on each of its rational cohomology groups $H^j(X,\mathbb{Q})$ in a functorial manner. The mixed Hodge structure consists of the following data:
\begin{itemize}
\item[(1)] a weight filtration $W_{\bullet}$ on $H^j(X,\mathbb{Q})$:
$$
\{0\}=W_{-1}\subset W_0\subset\cdots\subset W_{2j}=H^j(X,\mathbb{Q});
$$
\item[(2)] a Hodge filtration $F^{\bullet}$ on $H^j(X,\mathbb{C})$:
$$
H^j(X,\mathbb{C})=F^0\supset F^1\supset\cdots\supset F^m=\{0\}.
$$
\end{itemize}
On each complexified graded space $(W_i/W_{i-1})\otimes_{\mathbb{Q}}\mathbb{C}$, the Hodge filtration induces a pure Hodge structure of weight $i$. Similarly, one can also define a mixed Hodge structure on each compactly supported cohomology group $H^j_c(X,\mathbb{Q})$.

The generating series of the dimensions of the graded spaces of the mixed Hodge structure on the compactly supported cohomology groups is called the \textit{mixed Hodge polynomial}. More concretely, for any non-negative integers $p$, $q$ and $j$, put
$$
h^{p,q,j}_c(X)=\dim_{\mathbb{C}}\Gr_p^{F^{\bullet}}\Gr_{p+q}^{W_{\bullet}}H^j_c(X,\mathbb{C}),
$$
where $\Gr^{W_{\bullet}}_i:=W_i/W_{i-1}$ and $\Gr^{F^{\bullet}}_j\Gr^{W_{\bullet}}_i:=F_j\Gr^{W_{\bullet}}_i/F_{j+1}\Gr^{W_{\bullet}}_i$ for any $i$ and $j$. Then the mixed Hodge polynomial of $X$ is defined by
$$
H_c(X;x,y,t):=\sum_{p,q,j}h_c^{p.q,j}(X)x^py^qt^j,
$$
and the E-polynomial of $X$ is defined by
$$
E(X;x,y):=H_c(X;x,y,-1).
$$

According to a theorem of Katz \cite{HRV}, in many cases the E-polynomial of $X$ can be computed by counting points over finite fields. This means the following. Suppose that there is a ring $R\subset \mathbb{C}$ and a scheme $\mathfrak{X}$ over $R$ such that the base change of $\mathfrak{X}$ to $\mathbb{C}$ is isomorphic to $X$. Then for any homomorphism $\phi:R\rightarrow\mathbb{F}_q$ to a finite field, we consider the base change of $\mathfrak{X}$ via $\phi$, which is a variety $X_{\phi}$ over $\mathbb{F}_q$. Suppose that there is a polynomial $P_X(t)\in\mathbb{Z}[t]$ such that $|X_{\phi}(\mathbb{F}_q)|=P_X(q)$, which is independent of $\phi$. Then we have
$$
E(X;x,y)=P_X(xy).
$$
It happens that character varieties satisfy these assumptions, and the first application of this theorem of Katz is \cite{HRV}.

\subsection{The Conjecture of Hausel-Letellier-Rodriguez-Villegas}\hfill

Let $\Sigma$ be a Riemann surface of genus $g$ with $k>0$ punctures. Let $n\in\mathbb{Z}_{>0}$ and let $\mathcal{C}=(C_j)_{1\le j\le k}$ be a tuple of semi-simple conjugacy classes of $\GL_n$. In \cite{HLR}, the authors studied the character variety 
$$
\mathcal{M}_{\mathcal{C}}:=\left\{(A_i,B_i)_i(X_j)_j\in \GL_n^{2g}\times \prod_{j=1}^k C_j\mid\prod_{i=1}^g[A_i,B_i]\prod_{j=1}^kX_j=1\right\}\ds\GL_n,
$$where the bracket denotes the commutator and the quotient is with respect to conjugation. Their work led to a conjectural formula for the mixed Hodge polynomial $H_c(\mathcal{M}_{\mathcal{C}};x,y,t)$. We recall their conjectural formula below, which involves Macdonald polynomials. In this article, we will call Macdonald polynomials what are usually called modified Macdonald polynomials in the literature, since only this version of Macdonald polynomial appears.

Denote by $\mathcal{P}$ the set of all partitions, including the empty one. For any $\lambda\in\mathcal{P}$, denote by $H_{\lambda}(\Z;z,w)$ the Macdonald polynomial. It lives in the ring $\Sym[\Z]$ of symmetric functions over the field $\mathbb{Q}(z,w)$ of rational functions in $z$ and $w$. This ring is equipped with an inner product $\langle-,-\rangle$, called the Hall inner product. It can be deformed into another inner product $\langle-,-\rangle_{z,w}$.  Denote by $$N_{\lambda}(z,w)=\langle H_{\lambda}(\Z;z,w),H_{\lambda}(\Z;z,w)\rangle_{z,w}$$the self-pairing of a Macdonald polynomial for the deformed inner product. We will also need a deformation of it: $N_{\lambda}(u,z,w)$, with $N_{\lambda}(1,z,w)=N_{\lambda}(z,w)$. It is defined by an explicit formula (See \S \ref{subs-Mac}.)

The core of the conjecture of Hausel-Letellier-Rodriguez-Villegas is a mysterious generating series, the shape of which suggests a topological field theory behind. For any $g\ge 0$ and $k>0$, define 
\begin{equation}
\Omega_{HLRV,g,k}(z,w):=\sum_{\lambda\in\mathcal{P}}\frac{N_{\lambda}(zw,z^2,w^2)^g}{N_{\lambda}(z^2,w^2)}\prod_{j=1}^{k}\tilde{H}_{\lambda}(\mathbf{z}_j;z^2,w^2),
\end{equation}
where $\{\mathbf{z}_1,\ldots,\mathbf{z}_k\}$ are independent variables. We may omit $g$ and $k$ from the notation if their values are clear from the context. For each $1\le j\le k$, let $\mu_j$ be the partition of $n$ that encodes the multiplicities of the eigenvalues of $C_j$. Write $\bs\mu=(\mu_1,\ldots,\mu_k)$. Define a rational function $\mathbb{H}_{\bs\mu}(z,w)$ by taking the Hall inner product:
\begin{equation}
\mathbb{H}_{\bs\mu}(z,w):=(z^2-1)(1-w^2)\left\langle\Log\Omega_{HLRV}(z,w),\prod_{j=1}^kh_{\mu_j}\right\rangle,
\end{equation}
where $h_{\mu_j}$ is the \textit{complete symmetric function} and $\Log$ is the plethystic $\Log$ operator as defined in \cite{HRV}.

Let $d_{\bs\mu}$ be the dimension of $\mathcal{M}_{\mathcal{C}}$, which only depends on $\bs\mu$, $g$ and $n$. 
\begin{Con}(\cite{HLR})\label{Conj-HLRV}
Suppose that $\mathcal{C}$ is generic in the sense of \cite[Definition 2.1.1]{HLR}. Then the following statements are true.
\begin{itemize}
\item[(i)] The rational function $\mathbb{H}_{\mu}(z,w)$ is a polynomial. It has degree $d_{\mu}$ in each variable, and $\mathbb{H}_{\mu}(-z,w)$ has nonnegative integer coefficients.
\item[(ii)] The mixed Hodge polynomial $H_c(\mathcal{M}_{\mathcal{C}};x,y,t)$ is a polynomial in $q:=xy$ and $t$.
\item[(iii)] The mixed Hodge polynomial is given by $$
H_c(\mathcal{M}_{\mathcal{C}};q,t)=(t\sqrt{q})^{d_{\bs\mu}}\mathbb{H}_{\bs\mu}(-t\sqrt{q},\frac{1}{\sqrt{q}}).$$ In particular, it only depends on $\bs\mu$, and not on the generic eigenvalues of $\mathcal{C}$.
\end{itemize}
\end{Con}

\subsection{Around the mixed Hodge polynomial}\hfill

The case where $k=1$ and $C_1$ is central was studied in \cite{HRV}. Since the character variety $\mathcal{M}_{\mathcal{C}}$ can be defined over a subring of $\mathbb{C}$ that is of finite type over $\mathbb{Z}$, it admits base changes to finite fields. As mentioned above, a theorem of Katz reduces the computation of the E-polynomial to counting points over finite fields. The computation is implemented using the character table of $\GL_n(q)$. The general case of arbitrary $k>0$ was studied in \cite{HLR}. The resulting E-polynomials led to the conjecture that we have just recalled.

In \cite{HRV}, a symmetry in the conjectural mixed Hodge polynomial was observed, which is called the curious Poincar\'e duality. Hausel and Rodriguez-Villegas further conjectured that this symmetry can be upgraded to a symmetry in the mixed Hodge structure, which they call the curious Hard Lefschetz.

Then a remarkable observation was made by de Cataldo, Hausel and Migliorini, that the curious Hard Lefschetz for character varieties resembles the relative Hard Lefschetz for Hitchin fibrations. This led to the P=W conjecture \cite{dCHM}, claiming that the perverse filtration on the cohomology of Dolbeault moduli space coincides with the weight filtration on the cohomology of character variety under the non abelian Hodge correspondence. This conjecture has recently been settled by Maulik-Shen \cite{MS} and Hausel-Mellit-Minets-Schiffmann \cite{HMMS} indenpendently. 

The assertion in part (i) of Conjecture \ref{Conj-HLRV} concerning the polynomial property of $\mathbb{H}_{\bs\mu}(z,w)$ has been proved by Mellit in \cite{Me1}. Part (ii) of Conjecture \ref{Conj-HLRV} has been proved by Mellit in \cite{Me2}. The specialisation of part (iii) to Poincar\'e polynomials has been proved by Mellit and Schiffmann by counting Higgs bundles over finite fields. See \cite{Me3}, \cite{Me4} and \cite{Sch}.

The goal of this article is to study the mixed Hodge polynomial of a new family of character varieties.

\subsection{$\GL_n\lb\sigma\rb~$-Character Varieties}\label{subsec-Char-Var}\hfill

We study character varieties that are unitary in the global sense. This is what we call $\GL_n\rtimes\lb\sigma\rb~$-character varieties.  We will avoid calling them unitary character varieties since in the community of character variety, they often refer to character varieties with structure group $\Uni_n(\mathbb{R})$.

Let $\sigma$ be an order 2 exterior automorphism of $\GL_n$. We will denote by $\GL_n\lb\sigma\rb$ the semi-direct product $\GL_n\rtimes\lb\sigma\rb$. Let $p':\tilde{\Sigma}'\rightarrow\Sigma'$ be a double covering of compact Riemann surfaces that is branched at $2k$-points, with $k>0$. Let $p:\tilde{\Sigma}\rightarrow\Sigma$ be the restriction of $p'$ to its unbranched part, so that the punctures of $\Sigma$ are exactly the branch points in $\Sigma'$. Denote by $g$ the genus of $\Sigma$. Let $\mathcal{C}=(C_j)_{1\le j\le 2k}$ be a tuple of semi-simple conjugacy classes contained in $\GL_n\sigma$. Our $\GL_n\lb\sigma\rb~$-character variety is defined by
$$
\Ch_{\mathcal{C}}(\Sigma):=\left\{(A_i,B_i)_i(X_j)_j\in \GL_n^{2g}\times \prod_{j=1}^{2k} C_j\mid\prod_{i=1}^g[A_i,B_i]\prod_{j=1}^{2k}X_j=1\right\}\ds\GL_n.
$$ This is the moduli space of the homomorphisms $\rho$ that fit into the following commutative diagram:
$$
\begin{tikzcd}[row sep=2.5em, column sep=0.2em]
\pi_1(\Sigma) \arrow[dr, swap, "q_1"] \arrow[rr, "\rho"] && \GL_n\lb\sigma\rb \arrow[dl, "q_2"]\\
& \Gal(\tilde{\Sigma}/\Sigma)
\end{tikzcd}
$$
where $\Gal(\tilde{\Sigma}/\Sigma)\cong\mathbb{Z}/2\mathbb{Z}$ is the group of covering transformations, $q_1$ is the quotient by $\pi_1(\tilde{\Sigma})$, and $q_2$ is the quotient by the identity component. Of course, the monodromy of $\rho$ at the punctures must lie in the given conjugacy classes. 

The character variety $\Ch_{\mathcal{C}}(\Sigma)$, via the non-abelian Hodge theory\footnote{Non abelian Hodge correspondence should hold in the generality of Boalch-Yamakawa's twisted wild character varieties \cite{BY}. However, it seems that this has not been written down in the literature.},  corresponds to the moduli space of (parabolic) unitary Higgs bundles, that is, torsors under the unitary group scheme over $\Sigma$ equipped with a Higgs field. Torsors under the unitary group scheme can alternatively be described as vector bundles $\mathcal{E}$ on the the covering space $\tilde{\Sigma}$ equipped with an isomorphism $\mathcal{E}\isom\tau^{\ast}\mathcal{E}^{\vee}$, where $\mathcal{E}^{\vee}$ is the dual vector bundle and $\tau$ is the non-trivial covering transformation. Recent interests in this kind of moduli spaces grew out its connection to representation theory, e.g. the work of Laumon and Ng\^o \cite{LN} on the fundamental lemma, where they work with \'etale coverings of curves. We insist on branched coverings. As we will see, the most interesting phenomenon, i.e. insertion of wreath Macdonald polynomial, arises from the branch points.

This character variety also appears as a special case of the twisted character varieties studied by Boalch and Yamakawa \cite{BY}.

\subsection{Wreath Macdonald Polynomials}\hfill

In \cite{Hai}, Haiman conjectured the existence of a family of symmetric functions that are called \textit{wreath Macdonald polynomials}. Their symmetry is the wreath product $(\mathbb{Z}/r\mathbb{Z})^n\rtimes\mathfrak{S}_n$, with $r>0$, generalising the symmetric group $\mathfrak{S}_n$ for Macdonald polynomials. The existence and Schur-positivity of these symmetric functions were proved by Bezrukavnikov and Finkelberg in \cite{BF}. Let $\X=(\Xo,\Xl)$ be two independent sets of infinitely many variables and write $$\SymF[\Xo,\Xl]=\SymF[\Xo]\otimes\SymF[\Xl],$$ with coefficients in the ring $\mathbb{Q}(z,w)$ of rational functions. In this article, $r=2$, and the wreath Macdonald polynomials live in $\SymF[\Xo,\Xl]$. Below we explain some essence of wreath Macdonald polynomials and introduce some notations.

A 2-core is a partition of the form $(d,d-1,\ldots,1,0)$ for some $d\in\mathbb{Z}_{\ge0}$. A 2-partition is simply an element of $\mathcal{P}^2=\mathcal{P}\times\mathcal{P}$. Any partition determines a 2-core, and a 2-partition which is called its 2-quotient, and these data determine the original partition. For any $\bs\alpha=(\alpha^{(0)},\alpha^{(1)})\in\mathcal{P}^2$, its size is defined to be $|\alpha^{(0)}|+|\alpha^{(1)}|$. If $\lambda_c$ and $\bs\lambda_q$ are the 2-core and 2-quotient of $\lambda$ respectively, then we have $|\lambda|=|\lambda_c|+2|\bs\lambda_q|$. Fixing a 2-core of size $m$ and an integer $n>m$, we have a bijection $$\{\text{Partitions of size $n$ with the given 2-core}\}\lisomLR\{\text{2-Partitions of size $(n-m)/2$}\}$$ 
This bijection induces from the set of partitions an order on the set of 2-partitions. Wreath Macdonald polynomials are uniquely determined by two triangularity conditions and one normalisation condition. The order used in the triangularity conditions is the one induced from partitions. Different 2-cores may induce different orders on $\mathcal{P}^2$, and so define different wreath Macdonald polynomials. 

In this article, only the 2-cores $(0)$ and $(1)$ appear. For any $\bs\alpha\in\mathcal{P}^2$, we will denote by $\{\bs\alpha\}_0$ (resp. $\{\bs\alpha\}_1$) the partition which has $\bs\alpha$ as its 2-quotient, and has $(0)$ (resp. $(1)$) as its 2-core. With a fixed 2-core, we will denote by $\tilde{H}_{\bs\alpha}(\X;z,w)$ the wreath Macdonald polynomial associated to $\bs\alpha\in\mathcal{P}^2$. We will denote by $$\tilde{N}_{\bs\alpha}(z,w)=\langle \tilde{H}_{\bs\alpha}(\X;z,w),\tilde{H}_{\bs\alpha}(\X;z,w)\rangle_{z,w},$$the self-pairing of wreath Macdonald polynomials for the deformed inner product. We will also need a deformed version $\tilde{N}_{\bs\alpha}(u,z,w)$, with $\tilde{N}_{\bs\alpha}(1,z,w)=\tilde{N}_{\bs\alpha}(z,w)$. (See \S \ref{w-q,t-inner}.)

\subsection{From character varieties to wreath Macdonald polynomials}\hfill

Let $q$ be a prime power. Macdonald polynomials appear in the conjectural formula for the mixed Hodge polynomials of $\GL_n$-character varieties. The first step towards the mixed Hodge polynomial is to compute the $E$-polynomial by counting points over finite fields. In counting points, the representation theory of $\GL_n(q)$ provides the necessary combinatorial context in which the connection to the Macdonald polynomials can be observed. Each unipotent character of $\GL_n(q)$ is a linear combination of Deligne-Lusztig characters, and the coefficients are given by an irreducible character of the symmetric group $\mathfrak{S}_n$, which is the Weyl group of $\GL_n$. These Weyl group characters parametrise unipotent characters of $\GL_n(q)$. The computation of a Deligne-Lusztig character is reduced to the Green functions, and the Green functions are the bridge to symmetric functions. In our problem, we use the irreducible characters of $\GL_n(q)\lb\sigma\rb$, which will be described in terms of Deligne-Lusztig characters of non-connected groups. The bridge to symmetric functions is again provided by Green functions, using Shoji's results. In \cite{S01}, Shoji gives a description of Green functions, or rather, Kostka polynomials, associated to symplectic groups and orthogonal groups, in terms of symmetric functions in $\SymF[\Xo,\Xl]$.

Regard $\GL_n$ as an algebraic group over $\bar{\mathbb{F}}_q$, equipped with the Frobenius endomorphism $F$ that sends each entry of a matrix to its $q$-th power. Then the finite group $\GL_n(q)$ consists of the fixed points of $F$. Let $T_1$ be the maximal torus consisting of diagonal matrices and let $W\cong\mathfrak{S}_n$ be the Weyl group of $\GL_n$ defined by $T_1$. The automorphism $\sigma$ induces an automorphism of $W$. (See \S \ref{Section-GLnsigma} for the precise definition of $\sigma$.) Denote by $W^{\sigma}$ the fixed points of $\sigma$. In fact, we have $$W^{\sigma}\cong\mathfrak{W}_N:=(\mathbb{Z}/2\mathbb{Z})^N\rtimes\mathfrak{S}_N$$ with $N=[n/2]$. For each $w\in W^{\sigma}$, we have the generalised Deligne-Lusztig character $R_{T_w\sigma}\mathbf{1}$, as defined in \cite{DM94}, which is a function on $\GL_n(q)\sigma$ invariant under conjugation by $\GL_n(q)$. 

The clue that eventually leads to wreath Macdonald polynomials is the following theorem of Digne and Michel, which was a conjecture of Malle in \cite{Mal}.
\begin{Thm}(\cite[Th\'eor\`eme 5.2]{DM94})\label{Intro-DM}
Let $\lambda$ be a partition of $n$, and denote by $m$ the size of the 2-core of $\lambda$. Suppose that $m\le 1$. Then, the extension\footnote{That is, a character of $\GL_n(q)\lb\sigma\rb$ that restricts to the given character.} to $\GL_n(q)\sigma$ of the unipotent character of $\GL_n(q)$ corresponding to $\lambda$ is given by $$\pm|W^{\sigma}|^{-1}\sum_{w\in W^{\sigma}}\varphi(w)R_{T_w\sigma}\mathbf{1},$$where $\varphi$ is the irreducible character of $W^{\sigma}$ corresponding to the 2-quotient of $\lambda$.
\end{Thm}
\noindent
This theorem was then extended to \textit{quadratic-unipotent characters} by Waldspurger in \cite{W1}. In fact, Waldspurger's theorem imposes no restriction on the 2-cores. It says if the 2-core is larger than $(1)$, then the above expression would be a linear combination of inductions of cuspidal functions on Levi subgroups that are not tori.\footnote{Here we really mean "Levi" and "tori" in the twisted component $\GL_n\sigma$.} Then a theorem in \cite{DM15} shows that they must vanish on semi-simple conjugacy classes. Therefore, only the partitions with 2-core $(0)$ or $(1)$ have non-trivial contributions to the point counting of character varieties. 

The full character table of $\GL_n(q)\lb\sigma\rb$ has been completely determined in \cite{Shu2}, using the main theorem of \cite{W1}. Unlike $\GL_n(q)$, the irreducible characters of $\GL_n(q)\lb\sigma\rb$ may not be a linear combination of Deligne-Lusztig characters. Because of this, the determination of the irreducible characters is much more difficult than $\GL_n(q)$ and uses character sheaves on non connected groups developed by Lusztig.  The results of \cite{Shu2} will eventually allow us to compute the E-polynomial of $\GL_n\lb\sigma\rb~$-character varieties.

\subsection{The generating series}\label{subsec-series}\hfill

The main discovery of this article is the following infinite series that are built out of Macdonald polynomials, wreath Macdonald polynomials, and their self-pairings. They are expected to dominate the mixed Hodge polynomials of $\GL_n\lb\sigma\rb~$-character varieties.

For $e\in\{0,1\}$, define
\begingroup
\allowdisplaybreaks
\begin{align*}
\Omega_{e}(z,w):=&\sum_{\bs\alpha\in\mathcal{P}^2}\frac{N_{\{\bs\alpha\}_e}(zw,z^2,w^2)^{g+k-1}}{\tilde{N}_{\bs\alpha}(z^2,w^2)\tilde{N}_{\bs\alpha}(zw,z^2,w^2)^{k-1}}\prod_{j=1}^{2k}\tilde{H}_{\bs\alpha}(\mathbf{x}_j;z^2,w^2),\\
%%%
\Omega_{\ast}(z,w):
=&\sum_{\alpha\in\mathcal{P}}\frac{N_{\alpha}(zw,z^2,w^2)^{2g+k-1}}{N_{\alpha}(z^2,w^2)}\prod_{j=1}^{2k}H_{\alpha}(\mathbf{x}^{(0)}_j+\mathbf{x}^{(1)}_j;z^2,w^2).
\end{align*}
\endgroup
Beware that in $\Omega_e$, the wreath Macdonald polynomials and the (deformed) self-pairings are subordinate to the 2-core $(e)$. In $\Omega_{\ast}$ the symmetric functions involved are Macdonald polynomials in the variable $\Xo+\Xl:=\Xo\sqcup\Xl$.

\begin{Rem}
The series $\Omega_0$ and $\Omega_1$ have their origin in the quadratic-unipotent characters of $\GL_n(q)$.
\end{Rem}

\begin{Rem}
Note that $$\Omega_{\ast}(z,w)=\Omega_{HLRV,\tilde{g},2k}(z,w),$$where $\tilde{g}=2g+k-1$ is the genus of $\tilde{\Sigma}$. 
\end{Rem}

\begin{Rem}
If there are further $k'$ unbranched punctures, then each summand of the above series should be multiplied by $$\prod^{k'}_{j=1}H_{\{\bs\alpha\}_0}(\mathbf{z}_j;z^2,w^2),\quad\prod^{k'}_{j=1}H_{\{\bs\alpha\}_1}(\mathbf{z}_j;z^2,w^2),\quad\prod^{k'}_{j=1}H_{\alpha}(\mathbf{z}_j;z^2,w^2)^2,$$respectively, where each $\mathbf{z}_j$ is an independent family of variables. In this article, we will not introduce unbranched punctures.
\end{Rem}

\subsection{The Conjectural Mixed Hodge Polynomial}\label{subsec-Conjecture}\hfill

The infinite series that we have defined only depend on $g$ and $k$, and not on the conjugacy classes. Now we introduce the symmetric functions associated to the conjugacy classes.

At the end of \S \ref{Section-GLnsigma}, we will define the set $\tilde{\mathfrak{T}}_s$ of certain combinatorial data, called \textit{types}, encoding the multiplicities of the "eigenvalues" of semi-simple conjugacy classes in $\GL_n\sigma$. Given a tuple of semi-simple conjugacy classes $\mathcal{C}=(C_j)_{1\le j\le 2k}$ contained in $\GL_n\sigma$, let $\bs\beta_j$ be the type of $C_j$, and let $h_{\bs\beta_j^{\ast}}(\mathbf{x}_j)$ be the symmetric function associated to the dual $\bs\beta_j^{\ast}$ (See \S \ref{SymF-Types}), which is essentially the complete symmetric function. Write $\mathbf{B}=(\bs\beta_j)_{1\le j\le 2k}$. Let $\Ch_{\mathcal{C}}$ be the $\GL_n\lb\sigma\rb~$-character variety associated to $\mathcal{C}$. Denote by $d$ the dimension of $\Ch_{\mathcal{C}}$, which only depend on $\mathbf{B}$, $g$ and $n$.

The summand in $\Omega_{\ast}$ corresponding to the empty partition is equal to 1, therefore we can apply the formal expansion $$\frac{1}{1+x}=\sum_{m\ge 0}(-1)^mx^m$$ to $1+x=\Omega_{\ast}$.
Define the rational functions in $z$ and $w$:
\begingroup
\allowdisplaybreaks
\begin{align}
\mathbb{H}_{\mathbf{B}}(z,w):=&~\left\langle\frac{\Omega_{1}(z,w)\Omega_{0}(z,w)}{\Omega_{\ast}(z,w)},\prod^{2k}_{j=1}h_{\bs\beta^{\ast}_j}(\X_j)\right\rangle,\text{ if $n$ is odd,}\\
\mathbb{H}_{\mathbf{B}}(z,w):=&~\left\langle\frac{\Omega_{0}(z,w)^2}{\Omega_{\ast}(z,w)},\prod^{2k}_{j=1}h_{\bs\beta^{\ast}_j}(\X_j)\right\rangle,\text{ if $n$ is even.}
\end{align}
\endgroup
\begin{Rem}
A peculiar feature of these expressions is the absence of plethystic exponential or logarithm, which is by far ubiquitous in counting Higgs bundles and quiver representations.
\end{Rem}

\begin{Con}\label{The-Conj}
Let $\mathcal{C}=(C_j)_{1\le j\le 2k}$ be a strongly generic tuple of semi-simple conjugacy classes contained in $\GL_n\sigma$ which satisfies \ref{CCL}.\footnote{These notions will be defined in \S \ref{Section-Char-Var}.} Then 
\begin{itemize}
\item[(i)] The rational function $\mathbb{H}_{\mathbf{B}}(z,w)$ is a polynomial. It has degree $d$ in each variable, and each monomial in it has even degree. Moreover, $\mathbb{H}_{\mathbf{B}}(-z,w)$ has non-negative integer coefficients.
\item[(ii)] The mixed Hodge polynomial $H_c(\Ch_{\mathcal{C}};x,y,t)$ is a polynomial in $q:=xy$ and $t$.
\item[(iii)] The mixed Hodge polynomial is given by the following formula:
\begin{equation*}
H_c(\Ch_{\mathcal{C}};q,t)=(t\sqrt{q})^d\mathbb{H}_{\mathbf{B}}(-t\sqrt{q},\frac{1}{\sqrt{q}}).
\end{equation*}
In particular, it only depends on $\mathbf{B}$, and not on the generic eigenvalues of $\mathcal{C}$.
\end{itemize}
\end{Con}
Part (i) is a sufficient condition for the formula in part (iii) to be a polynomial in $q$ and $t$ with non negative integer coefficients. Part (ii) says that the mixed Hodge structure is Hodge-Tate.
\begin{Rem}
In this article, we only need wreath Macdonald polynomials associated to the wreath product $W^{\sigma}$, which is the Weyl group of the centraliser of $\sigma$. Since wreath products $(\mathbb{Z}/r\mathbb{Z})^n\rtimes\mathfrak{S}_n$ with $r>2$ are not the Weyl groups of any algebraic groups, their corresponding wreath Macdonald polynomials play no role in any character varieties. It is interesting to see if there is some imaginary mixed Hodge polynomials dominated by more general wreath Macdonald polynomials. Then the purely combinatorial part (i) of the conjecture may have a generalisation. 
\end{Rem}
\begin{Rem}
We may also consider a cyclic Galois covering $\tilde{\Sigma}'\rightarrow\Sigma'$ of degree $r>2$, and take for $\sigma$ an order $r$ exterior automorphism of $\GL_n$. The ramification index at a branch point determines which connected component of $\GL_n\lb\sigma\rb$ the local monodromy around that point lies in. The representation theory on a connected component $\GL_n\sigma^i$ is controlled by $\mathfrak{S}_n$ or $\mathfrak{W}_N$, according to the parity of $i$, or rather, whether $\sigma^i$ is inner or exterior. The combinatorial consequence of such a dichotomy of Weyl groups is the different symmetric functions inserted at this point - the usual Macdonald polynomial or the wreath Macdonald polynomial. We expect that essentially no new phenomenon arises when $r>2$. We do not study cyclic coverings of higher degrees because the character table of  $\GL_n(q)\lb\sigma\rb$ is only written down in the case $r=2$.
\end{Rem}

Part (iii) of the above conjecture, combined with some fundamental symmetries in Macdonald polynomials and wreath Macdonald polynomials, implies the following.
\begin{Con}(Curious Poincar\'e Duality)\label{Cur-Poin}
We have
\begin{equation*}
H_c(\Ch_{\mathcal{C}};\frac{1}{qt^2},t)=(qt)^{-d}H_c(\Ch_{\mathcal{C}};q,t).
\end{equation*}
\end{Con}
This is the starting point of the Curious Hard Lefschetz conjecture and the P=W conjecture.

\subsection{Main Theorem and Evidences of the Conjecture}\label{Intro-Main-Thm}\hfill

Our main theorem gives a formula for the E-polynomial, which is a main evidence of Conjecture \ref{The-Conj}.
\begin{Thm}\label{Main-Thm-intro}
The $t=-1$ specialisation of part (iii) of Conjecture \ref{The-Conj} holds:
$$
H_c(\Ch_{\mathcal{C}};q,-1)=(\sqrt{q})^d\mathbb{H}_{\mathbf{B}}(\sqrt{q},\frac{1}{\sqrt{q}}).
$$
\end{Thm}

Then the following theorem follows from this formula.
\begin{Thm}\label{E-Cur-Poin}
The $t=-1$ specialisation of Conjecture \ref{Cur-Poin} holds:
$$E(\Ch_{\mathcal{C}};q)=q^dE(\Ch_{\mathcal{C}};q^{-1}).$$
\end{Thm}
Other than the E-polynomial, we also have the following evidences. When $n=1$, the character variety is simply the torus $(\mathbb{C}^{\ast})^{2(g+k-1)}$, and we can confirm our conjecture in this case. When $n=2$, we do not know how to compute the mixed Hodge polynomial. However, we find a simple relation between $\GL_2\lb\sigma\rb~$-character varieties and $\GL_2$-character varieties, for any $g$ and $k$. We then prove that under this relation, the conjecture of Hausel-Letellier-Rodriguez-Villegas and our conjecture predict the same mixed Hodge polynomial.

When $n>2$, we are unable to say anything about the mixed Hodge polynomial, but we can focus on part (i) of Conjecture \ref{The-Conj}, which is purely combinatorial and non-trivial. To be precise, we have computed\footnote{The author uses MATLAB.} $\mathbb{H}_{\mathbf{B}}(z,w)$ in the following situations:
\begin{itemize}
\item $n=3$
\begin{itemize}
\item[] $g=1$, $k=1$, $\mathcal{C}=(C_1,C_2)$
\begin{quote}
$C_1$ regular, $C_2$ arbitrary
\end{quote}
\item[] $g=0$, $k=2$, $\mathcal{C}=(C_1,C_2,C_3,C_4)$
\begin{quote}
$C_1$ regular, $C_2$ arbitrary, $C_3=C_4\sim\sigma$, i.e. conjugacy class of $\sigma$
\end{quote}
\end{itemize}

\item $n=4$
\begin{itemize}
\item[] $g=1$, $k=1$, $\mathcal{C}=(C_1,C_2)$
\begin{quote}
$C_1$ regular, $C_2$ arbitrary;\\
$C_1\sim\diag(a,1,1,a^{-1})\sigma)$, $C_2\sim\diag(b,b,b^{-1},b^{-1})\sigma$
\end{quote}
\item[] $g=0$, $k=2$, $\mathcal{C}=(C_1,C_2,C_3,C_4)$
\begin{quote}
$C_1$ regular, $C_2$ arbitrary, $C_3=C_4\sim\sigma$;\\
$C_1\sim\diag(a,1,1,a^{-1})\sigma)$, $C_2\sim\diag(b,b,b^{-1},b^{-1})\sigma$, $C_3=C_4\sim\sigma$
\end{quote}
\end{itemize}

\item $n=5$
\begin{itemize}
\item[] $g=1$, $k=1$, $\mathcal{C}=(C_1,C_2)$
\begin{quote}
$C_1$ regular, $C_2$ arbitrary;\\
$C_1\sim\diag(a,1,1,1,a^{-1})\sigma)$, $C_2\sim\diag(b,b,1,b^{-1},b^{-1})\sigma$
\end{quote}
\item[] $g=0$, $k=2$, $\mathcal{C}=(C_1,C_2,C_3,C_4)$
\begin{quote}
$C_1$ regular, $C_2$ arbitrary, $C_3=C_4\sim\sigma$;\\
$C_1\sim\diag(a,1,1,1,a^{-1})\sigma)$, $C_2\sim\diag(b,b,1,b^{-1},b^{-1})\sigma$, $C_3=C_4\sim\sigma$
\end{quote}
\end{itemize}
\end{itemize}
and part (i) of Conjecture \ref{The-Conj} is always true.

\subsection*{Organisation of the Article}
In Sections \ref{Section-Pre}, \ref{Section-GLnsigma} and \ref{Section-Char-Var}, we recall and prove some results concerning character varieties and the representation theory of $\GL_n(q)\lb\sigma\rb$ that will be used later on. These sections contain three main ingredients in computing the E-polynomial:
\begin{enumerate}
\item Formula (\ref{Frob-Form-eq}) that reduces the point-counting problem to the evaluation of irreducible characters $\tilde{\chi}$.
\item Formula (\ref{thm-decIrr-eq}) that expresses $\tilde{\chi}$ as a linear combination of Deligne-Lusztig characters $R^{G\sigma}_{T_{w_1,\mathbf{w}}\sigma}\tilde{\theta}_{w_1,\mathbf{w}}$.
\item Formula (\ref{eq-char-formula}) that reduces the computation of $R^{G\sigma}_{T_{w_1,\mathbf{w}}\sigma}\tilde{\theta}_{w_1,\mathbf{w}}$ to the Green function $\smash{Q^{C_G(s\sigma)^{\circ}}_{C_{h^{-1}T_wh}(s\sigma)^{\circ}}(u)}$ and the linear character $\tilde{\theta}_{w_1,\mathbf{w}}$.
\end{enumerate}
Section \ref{Section-Lem} contains some technical lemmas. We prove that with strongly generic conjugacy classes, only a small subset of the characters of $\GL_n(q)\lb\sigma\rb$ can have non trivial contributions to the E-polynomial. This section also contains a computation of linear characters $\tilde{\theta}_{w_1,\mathbf{w}}$ using M\"obius inversion function. In Section \ref{Section-SymF} we recall the essence of symmetric functions associated to wreath products, notably the works of Shoji, and prepare some lemmas. In particular, we relate the Green function in formula (\ref{eq-char-formula}) to symmetric functions. In Section \ref{Section-E-poly} we compute the E-polynomial, combining results from all previous sections. We also make some sample computations by directly using the irreducible characters of $\GL_2(q)\lb\sigma\rb$ and $\GL_3(q)\lb\sigma\rb$. From these computations we extract a conjecture on the numbers of connected components of $\GL_n\lb\sigma\rb~$-character varieties. In Section \ref{Section-Proof}, some basic aspects of (wreath) Macdonald polynomials are recalled. Then we prove that the specialisation of the conjectural formula indeed agrees with our computation of the E-polynomial, concluding the proof of Theorem \ref{Main-Thm-intro}. Appendix \ref{App-WMac1-2} gives the explicit expressions of wreath Macdonald polynomials of degree 1 and degree 2 in terms of Schur functions. These formulae will be used in Appendix \ref{App-12} and Appendix \ref{App-45}. The reader will also observe some key features of wreath Macdonald polynomials. In Appendix \ref{App-12}, we verify our conjecture in the cases $n=1$ and $n=2$, admitting the conjecture of Hausel-Letellier-Rodriguez-Villegas. To prove that our claims about $\mathbb{H}_{\mathbf{B}}(z,w)$ in \S \ref{Intro-Main-Thm} are not bluffing, we give two sample computations in Appendix \ref{App-45}.

\subsection*{Acknowledgement.}
A good part of this article was prepared at Universit\'e de Paris as part of my thesis. The possibility of this work was suggested by Pierre Deligne to Fernando Rodriguez-Villegas in a private letter, which was then passed to my advisor Emmanuel Letellier. I thank them for sharing this idea. I thank Florent Schaffhauser and my advisor for some discussions. I thank Jean Michel and Philip Boalch for answering some questions. I thank Fran\c cois Digne for carefully reading and giving many comments on an earlier version of the article at the stage of thesis. I have also benefited from some results collected in the preprint \cite{C} of Camb\`o.  I thank Mark Shimozono and Daniel Orr for bringing to my attention wreath Macdonald polynomials and answering some questions. I thank the referee for pointing out many typos and giving some suggestions. Part of this article was prepared when I was unemployed at home. I thank my family for their support.

\numberwithin{Thm}{subsection}
\numberwithin{equation}{subsubsection}
%%%%%%%%%%%%%%%%%%%%%%%
\section{Preliminaries}\label{Section-Pre}
\subsection{General notations and terminology}\label{general-notations}\hfill

The 2-element group will be denoted by $\bs\mu_2$. If $G$ is a finite group, we will denote by $\Irr(G)$ the set of irreducible complex characters of $G$. If $\kk$ is a given algebraically closed field of characteristic different from 2, we will denote by $\mathfrak{i}$ a chosen square root of $-1$. In most part of the article, $n$ will denote the rank of our ambient group $\GL_n$, and $N$ will denote $[n/2]$.

Given an abstract group $G$ and an automorphism $\sigma$ of $G$, we will denote by $G^{\sigma}$ the subgroup of the fixed points of $\sigma$. If $G$ is abelian, then $[G,\sigma]:=\{g\sigma(g^{-1})\mid g\in G\}$ is also a subgroup of $G$. A $\sigma$-conjugacy class of $G$ is an orbit in $G$ under the action $g:x\mapsto gx\sigma(g)^{-1}$, for any $g$, $x\in G$. If $X\subset G$ is a subset, then we say that $X$ is $\sigma$-stable if $\sigma(X)=X$. 

If $X\subset G$ is a subset and $H\subset G$ is a subgroup, we will denote by $C_H(X)$ the centraliser of $X$ in $H$. If $H$, $X$ and $Y$ are subsets of $G$, we will denote by $N_H(X)$ the normaliser of $X$ in $H$ and by $N_H(X,Y)=N_H(X)\cap N_H(Y)$ the subset of $H$ that simultaneously normalises $X$ and $Y$. The centre of a group $G$ will be denoted by $Z_G$.

If $G_0\subset G$ is a subgroup and $G_1\subset G$ is a subset that is normalised by $G_0$, then by a $G_0$-conjugacy class in $G_1$, we mean a subset of the form $\{gxg^{-1}\mid g\in G_0\}$ for some $x\in G_1$.  Similarly, we can talk about a $G_0$-conjugate of an element or a subset of $G$. When $G$ is an algebraic group, the above notions make sense in the obvious way. If $G$ is an algebraic group, we will denote by $G^{\circ}$ its identity component. The subgroup $G_0$ as above is typically $G^{\circ}$ or a finite subgroup of $G$.

\subsection{Partitions}
\subsubsection{}\label{partitions}
For any $n\in\mathbb{Z}_{>0}$, we denote by $\mathfrak{S}_n$ the group of permutations of the set $$\mathbb{I}(n):=\{1,2,\ldots,n\}.$$ Each permutation $\tau\in\mathfrak{S}_n$ can be decomposed into a product of cycles $c_{I_1}\cdots c_{I_l}$, where the disjoint subsets $I_r\subset \mathbb{I}(n)$ form a partition of $\mathbb{I}(n)$ and $c_{I_r}$ is a circular permutation on $I_r$. For each $1\le r\le l$, put $\tau_r=|c_{I_r}|$, the size of $I_r$, then the conjugacy class of $\tau$ is determined by the partition $(\tau_1,\ldots,\tau_l)$. 

Let $n\in\mathbb{Z}_{\ge 0}$. A partition of $n$ will often be written as a non increasing sequence of non negative integers $\lambda=(\lambda_1\ge\lambda_2\ge\cdots)$ such that $n=\sum_k\lambda_k$, and $0$ only has the empty partition $(0,0,\ldots)$. The length of $\lambda$ will be denoted by $l(\lambda)$ and $n$ is called the size of $\lambda$. For each partition $\lambda$ we define the integer $n(\lambda):=\sum_i(i-1)\lambda_i$. The dual of a partition $\lambda$ will be denoted by $\lambda^{\ast}$. The set of partitions of $n$ will be denoted by $\mathcal{P}(n)$. Write $\mathcal{P}=\sqcup_{n\in\mathbb{Z}_{\ge 0}}\mathcal{P}(n)$. The irreducible characters and the conjugacy classes of $\mathfrak{S}_n$ are both parametrised by $\mathcal{P}(n)$. For any $\lambda\in\mathcal{P}(n)$, we will denote by $\chi^{\lambda}$ the corresponding irreducible character, and by $\chi^{\lambda}_{\mu}$ the value of this character on the conjugacy class corresponding to $\mu\in\mathcal{P}(n)$.

Write $\mathcal{P}^2=\mathcal{P}\times\mathcal{P}$. The elements of $\mathcal{P}^2$ will be called \textit{2-partitions}. For any $\bs\lambda=(\lambda^{(0)},\lambda^{(1)})\in\mathcal{P}^2$, its size is defined by $|\bs\lambda|:=|\lambda^{(0)}|+|\lambda^{(1)}|$, and its length is defined by $l(\bs\lambda):=l(\lambda^{(0)})+l(\lambda^{(1)})$. The dual of a 2-partition $\bs\lambda=(\lambda^{(0)},\lambda^{(1)})$ is defined by $\bs\lambda^{\ast}:=(\lambda^{(1)\ast},\lambda^{(0)\ast})$. The set of 2-partitions of size $n$ will be denoted by $\mathcal{P}^2(n)$.

\subsubsection{}\label{quotcore}
Given a partition $\lambda=(\lambda_1,\lambda_2,\ldots)$ of size $n$, we take $r\ge l(\lambda)$, and we put $$\delta_r=(r-1,r-2,\ldots,1,0).$$ Let $(2y_1>\cdots>2y_{l_0})$ and $(2y_1'+1>\cdots>2y'_{l_1}+1)$ be the even parts and the odd parts of $\lambda+\delta_r$, where the sum is taken term by term. Denote by $\lambda'$ the partition that has as its parts the numbers $2s+t$, $0\le s\le l_t-1$, $t=0,1$. We have $l(\lambda')=l(\lambda)$. The \textit{2-core} of $\lambda$ is the partition defined by $(\lambda'_k-l(\lambda)+k)_{1\le k\le l(\lambda)}$. It is independent of $r$ and necessarily of the form $(d,d-1,\ldots,2,1,0)$, for some $d\in\mathbb{Z}_{\ge0}$. Denote by $\smash{\lambda^{(0)}}$ the partition defined by $\smash{\lambda^{(0)}_k}=y_k-l_0+k$ and denote by $\smash{\lambda^{(1)}}$ the partition defined by $\smash{\lambda^{(1)}_k}=y'_k-l_1+k$. Then $(\lambda^{(0)},\lambda^{(1)})_r$ is a 2-partition that depends on the parity of $r$, which we call the \textit{2-quotient} of $\lambda$. Changing the parity of $r$ will permute $\lambda^{(0)}$ and $\lambda^{(1)}$. We make the convention that $r\equiv 1\mod 2$ if the 2-core is $(0)$ and $r\equiv 0\mod 2$ if the 2-core is $(1)$.
\begin{Rem}\label{conv-2-quot}
In \cite{Hai}, the 2-quotient is defined in terms of the residues of the contents of the boxes. The two definitions of 2-quotient agree when the 2-core is $(0)$, but not when the 2-core is $(1)$. In fact, Haiman's definition is equivalent to setting $r\equiv 1\mod 2$ in all cases. For example, let $\lambda=(3)$, then our definition gives the 2-quotient $((1),\varnothing)$, while Haiman's definition gives $(\varnothing,(1))$. Our convention is coherent with Theorem \ref{Intro-DM}, where the trivial character of $W^{\sigma}$ corresponds to the trivial character of $W$.
\end{Rem}

 The above constructions give a bijection 
\begin{equation}\label{eq-quotcore}
\left\{\begin{array}{c}\text{Partitions of $n$}\\\text{with 2-core $(d,d-1,\ldots,1,0)$}\end{array}\right\}\longleftrightarrow\left\{\text{2-partitions of $\frac{1}{2}(n-\frac{d(d+1)}{2})$}\right\}.
\end{equation}
We call this bijection the quotient-core decomposition of partitions.

\subsection{Wreath Products}
\subsubsection{}
Given a finite group $\Gamma$ and a positive integer $m$, the symmetric group $\mathfrak{S}_m$ acts on the direct product $\Gamma^m$ of $m$ copies of $\Gamma$ by permuting the factors. This defines a semi-direct product $\Gamma^m\rtimes\mathfrak{S}_m$, called a wreath product. We will denote by $\mathfrak{W}_m$ the wreath product defined by $\Gamma=\mathbb{Z}/2\mathbb{Z}$. It is the Weyl group of $\Sp_{2m}$ and $\SO_{2m+1}$. More general wreath products $(\mathbb{Z}/r\mathbb{Z})^m\rtimes\mathfrak{S}_m$ for $r>2$ can not be realised as the Weyl groups of any algebraic groups. We will always identify $\mathbb{Z}/2\mathbb{Z}$ with $\bs\mu_2$ and write its elements in the multiplicative form.

\subsubsection{}
Denote by $w_0$ the permutation $(1,-1)(2,-2)\cdots(m,-m)$ of the set $$\bar{\mathbb{I}}(m):=\{1,\ldots,m,-m,\ldots,-1\}.$$We can identify $\mathfrak{W}_m$ as the subgroup of the group of permutations on $\bar{\mathbb{I}}(m)$ consisting of elements commuting with $w_0$. Then each element $\tau\in\mathfrak{S}_m\subset\mathfrak{W}_m$ is identified with the permutation $$i\mapsto \tau(i),\quad-i\mapsto\tau(-i)=-\tau(i),$$and for any $i\in\{1,\ldots,m\}$, the element $(e_1,\ldots,e_m)\in(\mathbb{Z}/2\mathbb{Z})^m\subset\mathfrak{W}_m$, with $e_i=-1$ and $e_j=1$ if $j\ne i$, is identified with the permutation $(i,-i)$.

\subsubsection{}\label{signed-partitions}
Let $w=((e_1,\ldots,e_m),\tau)\in(\mathbb{Z}/2\mathbb{Z})^m\rtimes\mathfrak{S}_m$, then $\tau$ can be written as $\tau=c_{I_1}\cdots c_{I_l}$ as in \S \ref{partitions}. For each $1\le r\le l$, put $\bar{e}_r=\prod_{k\in I_r}e_k$. For any $1\le r\le l$, we will call $((e_r)_{r\in I_r},c_{I_r})$ a \textit{positive} (resp. \textit{negative}) \textit{cycle} in $w$ if $\bar{e}_r=1$ (resp. $\bar{e}_r=-1$), so that $w$ is a product of signed cycles. The size of a signed cycle in $w$ is defined as the size of the corresponding cycle in $\tau$. Define the permutations
\begin{equation}
\tau^{(0)}=\prod_{\bar{\epsilon}_r=1}c_{I_r},\quad
\tau^{(1)}=\prod_{\bar{\epsilon}_r=-1}c_{I_r},
\end{equation}
so that $\tau=\tau^{(0)}\tau^{(1)}$. Also denote by $\tau^{(0)}$ and $\tau^{(1)}$ the associated partitions by abuse of notation. We then have a 2-partition $(\tau^{(0)},\tau^{(1)})$, which is sometimes called a signed partition of $m$. This 2-partition determines the conjugacy class of $w$. The conjugacy classes and irreducible characters of $\mathfrak{W}_m$ are both parametrised by $\mathcal{P}^2(m)$. For any $\bs\lambda\in\mathcal{P}^2(m)$, we will denote by $\chi^{\bs\lambda}$ the corresponding irreducible character, and by $\chi^{\bs\lambda}_{\bs\mu}$ the value of this character on the conjugacy class corresponding to $\bs\mu\in\mathcal{P}^2(m)$.

Regarded as a permutation on the set $\bar{\mathbb{I}}(m)$, an element $w$ such that $\tau=\tau^{(0)}=(1,2,\ldots,m)$ is typically of the form:
\begin{equation}\label{w_+}
\begin{tikzpicture}[node distance=1, on grid, cir/.style={circle, minimum size=0.6cm, inner sep=0pt}]
\node[cir] (1) {1};
\node[cir, right=of 1] (2) {2};
\node[cir, right=of 2] (dt1) {$\cdots$};
\node[cir, right=of dt1] (m) {m};
\node[cir, right=of m] (m') {-m};
\node[cir, right=of m'] (dt2) {$\cdots$};
\node[cir, right=of dt2] (2') {-2};
\node[cir, right=of 2'] (1') {-1};
\draw [->] (1) to [out=290,in=250, looseness=1.3] (2);
\draw [->] (2) to [out=290,in=250, looseness=1.3] (dt1);
\draw [->] (dt1) to [out=290,in=250, looseness=1.3] (m);
\draw [->] (m) to [out=110,in=70, looseness=0.8] (1);
\draw [->] (1') to [out=250,in=290, looseness=1.3] (2');
\draw [->] (2') to [out=250,in=290, looseness=1.3] (dt2);
\draw [->] (dt2) to [out=250,in=290, looseness=1.3] (m');
\draw [->] (m') to [out=70,in=110, looseness=0.8] (1');
\end{tikzpicture}
\end{equation}
and an element $w$ such that $\tau=\tau^{(1)}=(1,2,\ldots,m)$ is typically of the form:
\begin{equation}\label{w_-}
\begin{tikzpicture}[node distance=1, on grid, cir/.style={circle, minimum size=0.6cm, inner sep=0pt}]
\node[cir] (1) {1};
\node[cir, right=of 1] (2) {2};
\node[cir, right=of 2] (dt1) {$\cdots$};
\node[cir, right=of dt1] (m) {m};
\node[cir, right=of m] (m') {-m};
\node[cir, right=of m'] (dt2) {$\cdots$};
\node[cir, right=of dt2] (2') {-2};
\node[cir, right=of 2'] (1') {-1};
\draw [->] (1) to [out=290,in=250, looseness=1.3] (2);
\draw [->] (2) to [out=290,in=250, looseness=1.3] (dt1);
\draw [->] (dt1) to [out=290,in=250, looseness=1.3] (m);
\draw [->] (m) to [out=290,in=250, looseness=0.8] (1');
\draw [->] (1') to [out=110,in=70, looseness=1.3] (2');
\draw [->] (2') to [out=110,in=70, looseness=1.3] (dt2);
\draw [->] (dt2) to [out=110,in=70, looseness=1.3] (m');
\draw [->] (m') to [out=110,in=70, looseness=0.8] (1);
\end{tikzpicture}
\end{equation}

\subsection{Springer correspondence and symbols}
\subsubsection{}
Let $G$ be a connected reductive group and let $W$ be the Weyl group of $G$ defined by a maximal torus. Denote by $\mathscr{N}$ the set consisting of pairs $(C,\phi)$, where $C\subset G$ is a unipotent conjugacy class and $\phi$ is an irreducible character of $C_G(u)/C_G(u)^{\circ}$ for some $u\in C$. The Springer correspondence is an injection (\cite[\S 4]{S87}) 
\begin{equation}\label{Spr-Corr}
\Irr(W)\hooklongrightarrow \mathscr{N}.
\end{equation}
There is an equivalence relation on $\mathscr{N}$ that identifies $(C_1,\phi_1)$ and $(C_2,\phi_2)$ whenever $C_1=C_2$, in which case we say $(C_1,\phi_1)$ and $(C_2,\phi_2)$ are \textit{similar}. Such an equivalence class is called a similarity class. 

\textit{Symbols} are certain combinatorial data that are in bijection with the image of (\ref{Spr-Corr}). If $G=\Sp_{2m}$ or $\SO_{2m+1}$, then $\Irr(W)$ is in bijection with $\mathcal{P}^2(m)$. Therefore, the symbols are also in bijection with $\mathcal{P}^2(m)$. Similarity classes in $\mathcal{P}^2(m)$ or in the set of symbols are induced from $\mathscr{N}$. If $\bs\alpha$ is a 2-partition, we will denote by $\bs\Lambda(\bs\alpha)$ the corresponding symbol. In this article, we do not need the notion of symbols in an essential way. We use them only to match the notations in \cite{S01}.

\subsubsection{}
For any symbol $\bs{\Lambda}$, let $a(\bs{\Lambda})$ be the integer defined by \cite[(1.2.2)]{S01}. This is the analogue of $n(\lambda)$ for partitions. Its value is constant on the similarity classes of symbols. If we denote by $u$ an element of the unipotent conjugacy class corresponding to $\bs{\Lambda}$, then 
\begin{equation}\label{a=dimBu}
a(\bs{\Lambda})=\dim\mathcal{B}_u,
\end{equation}
where $\mathcal{B}_u$ is the Springer fiber over $u$. For any 2-partition $\bs\alpha$, we define $a(\bs{\alpha}):=a(\bs{\Lambda}(\bs{\alpha}))$. We will fix once and for all a total order $\prec$ on the set of symbols in such a way that, if $C_1$ (resp. $C_2$) is the unipotent class corresponding to $\bs\Lambda_1$ (resp. $\bs\Lambda_2$), then $\bs\Lambda_2\prec\bs\Lambda_1$ if $C_2\subset\bar{C}_1$, and each similarity class forms an interval. We will write $\bs\Lambda_1\sim\bs\Lambda_2$ if $\bs\Lambda_1$ and $\bs\Lambda_2$ are similar. Note that $a(\bs\Lambda_1)\le a(\bs\Lambda_2)$ whenever $\bs\Lambda_2\prec\bs\Lambda_1$. The set of 2-partitions acquires a total order via the bijection with symbols, so that $a(\bs{\alpha})\le a(\bs{\beta})$ whenever $\bs{\beta}\prec\bs{\alpha}$. In particular, the element $\bs{\alpha}_0:=(\varnothing,(1^n))$ corresponding to the the identity of the finite classical group, is minimal, and it is alone in its similarity class.

\subsection{Non-Connected Algebraic Groups}
\subsubsection{}
Let $G$ be a not necessarily connected linear algebraic group over an algebraically closed field $\kk$. We say that $G$ is reductive if $G^{\circ}$ is reductive. Let $B\subset G^{\circ}$ be a Borel subgroup and let $T\subset B$ be a maximal torus. Then $\bar{T}:=N_G(T,B)$ has non empty intersection with every connected component of $G$, since the conjugation by any element of $G$ sends $T$ and $B$ to some other maximal torus and Borel subgroup, which are conjugate to $T$ and $B$ by an element of  $G^{\circ}$. An element of $G$ is called \textit{quasi-semi-simple} if it lies in $\bar{T}$ for some $T\subset B$. It is known that semi-simple elements are always quasi-semi-simple (\cite[Theorem 7.5]{St}). We will always assume that $\ch\kk\nmid |G/G^{\circ}|$. Under this assumption, all unipotent elements of $G$ are contained in $G^{\circ}$ and all quasi-semi-simple elements are semi-simple (\cite[Remarque 2.7]{DM94}).

\subsubsection{}
Let $G^1$ be a connected component of $G$ and let $\sigma\in G^1$ be a semi-simple element. Let $T\subset G^{\circ}$ be a maximal torus that is normalised by $\sigma$. Denote by $W=W_{G^{\circ}}(T)$ the Weyl group of $G^{\circ}$ defined by $T$. Then $\sigma$ induces an action on $W$. Denote by $W^{\sigma}$ the subgroup of $\sigma$-fixed points. The subtori $(T^{\sigma})^{\circ}$ and $[T,\sigma]$ are preserved by the action of $W^{\sigma}$ on $T$.

\begin{Prop}(\cite{DM18}[Proposition 1.16])\label{DM1.16}
Every semi-simple $G^{\circ}$-conjugacy class in $G^1$ has a representative in $(T^{\sigma})^{\circ}\sigma$. Two elements $t\sigma$ and $t'\sigma$ with $t$, $t'\in (T^{\sigma})^{\circ}$, represent the same class if and only if $t$ and $t'$, when passing to the quotient $(T^{\sigma})^{\circ}/(T^{\sigma})^{\circ}\cap[T,\sigma]$, belong to the same $W^{\sigma}$-orbit.
\end{Prop}

A semi-simple element $\sigma$ is \textit{quasi-central} in $G^1$ if there is no element $g\in G^{\circ}$ such that $\dim C_{G^{\circ}}(\sigma)^{\circ}<\dim C_{G^{\circ}}(g\sigma)^{\circ}.$

\begin{Thm}(\cite[Th\'eor\`eme 1.15]{DM94})\label{1.15DM94}
A semi-simple element $\sigma\in G^1$ is quasi-central if and only if for every $\sigma$-stable maximal torus $T$ contained in a $\sigma$-stable Borel subgroup of $G^{\circ}$, every $\sigma$-stable element of $N_{G^{\circ}}(T)/T$ has a representative in $C_{G^{\circ}}(\sigma)^{\circ}$. 
\end{Thm}

\subsubsection{}\label{Para-of-G}
A closed subgroup of $G$ is a parabolic subgroup if $G/P$ is proper. According to \cite[Lemma 6.2.4]{Spr}, a subgroup $P\subset G$ is parabolic if and only if $P^{\circ}$ is parabolic in $G^{\circ}$. For example, if $P$ is a parabolic subgroup of $G^{\circ}$, then $P$ and $N_G(P)$ are both parabolic subgroups of $G$. It is obvious that the intersection of $N_G(P)$ with any connected component of $G$ is isomorphic to $P$ if it is non-empty. We want to determine which connected components of $G$ have non-empty intersection with $N_G(P)$. 

There is a well-defined action of $G/G^{\circ}$ on the set of $G^{\circ}$-conjugacy classes of parabolic subgroups of $G^{\circ}$, induced by the conjugation action of $G$. Let $G^1\subset G$ be a connected component and $P\subset G^{\circ}$ a parabolic subgroup. Then $N_G(P)$ meets $G^1$ if and only if the $G^{\circ}$-conjugacy class of $P$ is $G^1$-stable. Let $L\subset P$ be a Levi factor. If $N_G(P)$ meets a connected component $G^1$ of $G$, then $N_G(L,P)$ also meets $G^1$, since any two Levi factors of $P$ are conjugate under $P$. 

Fix a semi-simple element $\sigma\in G^1$, then by \cite[Proposition 1.6]{DM94} and the above discussions, $N_G(P)$ meets $G^1$ if and only if there exists a $G^{\circ}$-conjugate of $P$ that is $\sigma$-stable, and $N_G(L,P)$ meets $G^1$ if and only if there exists a $G^{\circ}$-conjugate of the pair $(L,P)$ that is $\sigma$-stable. The following propositions describe the set of $\sigma$-stable parabolic subgroups and the set of $\sigma$-stable Levi factors of $\sigma$-stable parabolic subgroups.
\begin{Prop}(\cite[Proposition 1.11 (ii)]{DM94})\label{1.11DM94}
Let $s\in G$ be semi-simple. Let $L$ be an $s$-stable Levi factor of an $s$-stable parabolic subgroup $P\subset G^{\circ}$. Then $(P^s)^{\circ}\subset (G^s)^{\circ}$ is a parabolic subgroup and $(L^s)^{\circ}\subset (P^s)^{\circ}$ is a Levi factor.
\end{Prop}

\begin{Prop}(\cite[Corollaire 1.25]{DM94})\label{1.25DM94}
Let $\sigma\in G$ be quasi-central.
\begin{itemize}
\item[(1)] The map $P\mapsto(P^{\sigma})^{\circ}$  defines a bijection between the $\sigma$-stable parabolic subgroups of $G^{\circ}$ and the parabolic subgroups of $(G^{\sigma})^{\circ}$.
\item[(2)] Then map $L\mapsto(L^{\sigma})^{\circ}$  defines a bijection between the $\sigma$-stable Levi factors of $\sigma$-stable parabolic subgroups of $G^{\circ}$ and the Levi subgroups of $(G^{\sigma})^{\circ}$. The inverse is given by $L'\mapsto C_{G^{\circ}}(Z^{\circ}_{L'})$ for $L'\subset (G^{\sigma})^{\circ}$.
\end{itemize}
\end{Prop}

\subsubsection{}
The following proposition will be used via its corollaries.
\begin{Prop}\label{Nsigma-in-N}
Let $s\in G$ be semi-simple and let $L$ be an $s$-stable Levi factor of an $s$-stable parabolic subgroup of $G^{\circ}$. Suppose that $L=C_{G^{\circ}}(Z^{\circ}_{(L^s)^{\circ}})$. Then $N_{(G^{s})^{\circ}}((L^{s})^{\circ})$ is contained in  $N_{G^{\circ}}(L)$. Moreover, each connected component of $N_{G^{\circ}}(L)$ contains exactly one connected component of $N_{(G^{s})^{\circ}}((L^{s})^{\circ})$.
\end{Prop}
\begin{proof}
The inclusion $N_{(G^{s})^{\circ}}((L^{s})^{\circ})\subset N_{G^{\circ}}(L)$ follows from the assumption on $L$ and $s$. To prove the second part, it suffices to prove that $N_{(G^{s})^{\circ}}((L^{s})^{\circ})\cap L=(L^{s})^{\circ}$. This is achieved by considering the conjugation action of $\xi\in N_{(G^{s})^{\circ}}((L^{s})^{\circ})\cap L$ on $Z_{(L^{s})^{\circ}}^{\circ}$.

We have $(Z_L^{s})^{\circ}=Z_{(L^{s})^{\circ}}^{\circ}$. The inclusion $(Z_L^{s})^{\circ}\subset Z_{(L^{s})^{\circ}}^{\circ}$ is obvious. Inclusion in the other direction follows again from the assumption on $L$ and $s$. Since $\xi$ lies in $L$, the conjugation action of $\xi$ on $(Z_L^{s})^{\circ}$ must be trivial. 

It is well-known that if $M$ is a Levi subgroup of a connected reductive group $H$, then $M=C_H(Z^{\circ}_M)$. By Proposition \ref{1.11DM94}, $(L^{s})^{\circ}$ is a Levi subgroup of $(G^{s})^{\circ}$. We conclude that $\xi$ must lie in $(L^{\sigma})^{\circ}$.
\end{proof}
With the notion of \textit{isolated elements}, the following two corollaries could be integrated into one uniform result. But we prefer not to introduce more notions.
\begin{Cor}\label{Nsigma-in-N-1}
Let $\sigma\in G$ be quasi-central and let $L$ be a $\sigma$-stable Levi factor of a $\sigma$-stable parabolic subgroup of $G^{\circ}$. Write $$W_{(G^{\sigma})^{\circ}}((L^{\sigma})^{\circ}):=N_{(G^{\sigma})^{\circ}}((L^{\sigma})^{\circ})/(L^{\sigma})^{\circ},\quad W_{G^{\circ}}(L):=N_{G^{\circ}}(L)/L.$$Then there is a natural injective map $$W_{(G^{\sigma})^{\circ}}((L^{\sigma})^{\circ})\hookrightarrow W_{G^{\circ}}(L)^{\sigma}.$$
\end{Cor}
\begin{proof}
By \cite[Proposition 1.23]{DM94}, $L$ and $\sigma$ satisfy the assumption on $L$ and $s$ in Proposition \ref{Nsigma-in-N}. The image is automatically contained in the $\sigma$-fixed part.
\end{proof}
\begin{Cor}\label{Nsigma-in-N-2}
Let $\sigma\in G$ be quasi-central. Let $T\subset G^{\circ}$ be a $\sigma$-stable maximal torus and let $s\in T$. Write $$W_{(G^{s\sigma})^{\circ}}((T^{s\sigma})^{\circ}):=N_{(G^{s\sigma})^{\circ}}((T^{s\sigma})^{\circ})/(T^{s\sigma})^{\circ},\quad W_{G^{\circ}}(T):=N_{G^{\circ}}(T)/T.$$Then there is a natural injective map $$W_{(G^{s\sigma})^{\circ}}((T^{s\sigma})^{\circ})\hookrightarrow W_{G^{\circ}}(T)^{\sigma}.$$
\end{Cor}
\begin{proof}
We have $T^{s\sigma}=T^{\sigma}$. Thus by Proposition \ref{1.25DM94}, $T$ and $s\sigma$ satisfy the assumption on $L$ and $s$ in Proposition \ref{Nsigma-in-N}. It is easy to see that the image is contained in the $\sigma$-fixed part.
\end{proof}
\begin{Rem}\label{Wsigma=WGsigma}
The special case where $L=T$ is a maximal torus implies that the condition in Theorem \ref{1.15DM94} is equivalent to $W_{(G^{\sigma})^{\circ}}((T^{\sigma})^{\circ})\cong W_{G^{\circ}}(T)^{\sigma}$. 
\end{Rem}

\subsection{Algebraic Groups Defined over a Finite Field}\hfill

Let $q$ be a power of a prime number. In this section $\kk$ denotes an algebraic closure of $\mathbb{F}_q$. 
\subsubsection{}
An algebraic group $G$ over $\kk$ is defined over $\mathbb{F}_q$ if there is an algebraic group $G_0$ over $\mathbb{F}_q$ such that $G_0\otimes_{\mathbb{F}_q}\kk\cong G$. Given $G_0$, the geometric Frobenius endomorphism on $G$ is denoted by $F$. The set of fixed points of $F$, denoted by $G^F$, is a finite group, which we will sometimes write as $G(q)$. If $X\subset G$ is a subvariety, then we say that $X$ is $F$-stable if $F(X)=X$.

\subsubsection{}\label{ConnG-L_I-F}
Let $G$ be a connected reductive group defined over $\mathbb{F}_q$. Let $T\subset G$ be an $F$-stable maximal torus. Then $F$-acts naturally on $W_G(T)$ and so we can talk about the $F$-conjugacy classes in $W_G(T)$ as defined in \S \ref{general-notations}. The $G^F$-conjugacy classes of $F$-stable maximal tori are in bijection with the $F$-conjugacy classes of $W_G(T)$. This bijection is described as follows. Given $w\in W_G(T)$, let $\dot{w}\in G$ be a representative of $w$. By Lang-Steinberg theorem, there exists $g\in G$ such that $g^{-1}F(g)=\dot{w}$. Then $gTg^{-1}$ is again an $F$-stable maximal torus, and its $G^F$-conjugacy class only depends on the $F$-conjugacy class of $w$. Conversely, if $gTg^{-1}$ is an $F$-stable maximal torus for some $g\in G$, then $g^{-1}F(g)$ normalises $T$ and so defines an element of $W_G(T)$. The two constructions are inverse to each other.

Fix an $F$-stable maximal torus $T$ and a Borel subgroup $B$ containing $T$. Let $\Delta=\Delta(T,B)$ be the set of simple roots defined by $T$ and $B$. The parabolic subgroups containing $B$, called \textit{standard parabolic subgroups}, are parametrised by the set of the subsets of $\Delta$. For any $I\subset\Delta$, denote by $P_I$ the corresponding standard parabolic subgroup. Its unique Levi factor containing $T$ is called a \textit{standard Levi subgroup}, denoted by $L_I$. Every Levi subgroup of $G$ is conjugate to a standard Levi subgroup. Fix $I\subset\Delta$. The $G^F$-conjugacy classes of the $G$-conjugates of $L_I$ are in bijection with the $F$-conjugacy classes of $N_G(L_I)/L_I$. The construction of this correspondence resembles the case of maximal tori.

\subsubsection{}\label{Prop1.40}
Suppose that $G$ is a reductive group defined over $\mathbb{F}_q$ and that $G/G^{\circ}$ is a cyclic group. Let $\sigma\in G$ be an $F$-stable quasi-central element such that $G/G^{\circ}$ is generated by the component of $\sigma$.
\begin{Prop}(\cite[Proposition 1.40]{DM94})\label{1.40}
The $G^F$-conjugacy classes of the $F$-stable subgroups $\bar{L}=N_G(L,P)$ defined by some $F$-stable Levi factor $L$ of some parabolic subgroup $P\subset G^{\circ}$,  satisfying $\bar{L}\cap G^{\circ}\sigma\ne\emptyset$, are in bijection with the $((G^{\sigma})^{\circ})^F$-conjugacy classes of the $F$-stable Levi subgroups of $(G^{\sigma})^{\circ}$ in the following manner. Each $\bar{L}$ has a $G^F$-conjugate $\bar{L}_1$ containing $\sigma$, and the bijection associates the $((G^{\sigma})^{\circ})^F$-class of $((\bar{L}_1)^{\sigma})^{\circ}$  to the $G^F$-class of $\bar{L}$.
\end{Prop}
We give a more concrete description of this correspondence. Fix an $F$-stable and $\sigma$-stable maximal torus $T$ and a $\sigma$-stable Borel subgroup $B\subset G^{\circ}$ that contains $T$. Let $I$ be a subset of the set of simple roots defined by $(T,B)$. Let $L_I$ be the standard Levi subgroup with respect to $(T,B)$. Suppose that $I$ is $\sigma$-stable and so $(L_I^{\sigma})^{\circ}$ is a standard Levi subgroup of $(G^{\sigma})^{\circ}$ with respect to $((T^{\sigma})^{\circ},(B^{\sigma})^{\circ})$. Then the $((G^{\sigma})^{\circ})^F$-conjugacy classes of the $(G^{\sigma})^{\circ}$-conjugates of $(L_I^{\sigma})^{\circ}$ are in bijection with the $F$-conjugacy classes in $W_{(G^{\sigma})^{\circ}}((L_I^{\sigma})^{\circ})$. For any $w\in W_{(G^{\sigma})^{\circ}}((L_I^{\sigma})^{\circ})$, the procedure of \S \ref{ConnG-L_I-F} gives an $F$-stable Levi factor $L'_{I,w}$ of some parabolic subgroup $P'_{I,w}\subset (G^{\sigma})^{\circ}$, which is $(G^{\sigma})^{\circ}$-conjugate to $(L_I^{\sigma})^{\circ}$ and whose $((G^{\sigma})^{\circ})^F$-conjugacy class corresponds to the $F$-conjugacy class of $w$. By Proposition \ref{1.25DM94}, the pair $(L'_{I,w},P'_{I,w})$ determines an $F$-stable and $\sigma$-stable Levi factor $L_{I,w}$ of a $\sigma$-stable parabolic subgroup $P_{I,w}$. In fact, if $\dot{w}$, $g\in(G^{\sigma})^{\circ}$ are as in \S \ref{ConnG-L_I-F}, then $L_{I,w}=gL_Ig^{-1}$. Put $$\bar{L}_{I,w}:=N_G(L_{I,w},P_{I,w})=L_{I,w}\sqcup L_{I,w}\sigma.$$ It is $F$-stable and meets the connected component $G^{\circ}\sigma$. By Proposition \ref{1.40}, such groups for varying $I$ and $w$ exploit all $G^F$-conjugacy classes of $F$-stable subgroups of the form $\bar{L}$ that meet $G^{\circ}\sigma$. This implies in particular that the $G^F$-conjugacy classes of the $F$-stable subgroups $N_G(T',B')$, defined by an $F$-stable maximal torus $T'$ contained in a Borel subgroup $B'\subset G^{\circ}$, are parametrised by the $F$-conjugacy classes of $W_{G^{\circ}}(T)^{\sigma}$ (See Remark \ref{Wsigma=WGsigma}). 

Define another Frobenius $F_w:=\ad\dot{w}\circ F$ on $G$. Then there is a commutative diagram
\begin{equation}\label{F_w}
\begin{CD}
L_I @>\ad g>> L_{I,w}\\
@VF_wVV @VVFV\\
L_I @>>\ad g> L_{I,w}
\end{CD}
\end{equation}
We will simply say that $(L_I,F_w)$ is isomorphic to $(L_{I,w},F)$ via $\ad g$.

\subsection{Frobenius Formula}
\subsubsection{}
Let $H$ be a finite group and let $H_0\lhd H$ be a normal subgroup of index 2. Let $\sigma$ be an element of $H\setminus H_0$. Then $\sigma$ defines an involution on $\Irr(H_0)$ by composing a character with the conjugation by $\sigma$. This involution does not depend on the choice of $\sigma$ in $H\setminus H_0$. Denote by $\Irr(H_0)^{\sigma}$ the subset of $\sigma$-fixed elements. By Clifford theory, each element $\chi\in\Irr(H_0)^{\sigma}$ admits an extension to $H$, i.e. some $\chi_H\in\Irr(H)$ such that $(\chi_H)|_{H_0}=\chi$. Moreover, there are precisely two such extensions, and they differ by $-1$ on $H\setminus H_0$. For every $\chi\in\Irr(H_0)^{\sigma}$, we choose a preferred extension, denoted by $\tilde{\chi}$. If the restriction of $\chi_H\in\Irr(H)$ to $H_0$ is not irreducible (or rather, $\chi_H$ is not an extension of any $\chi\in\Irr(H_0)$), then $\chi_H$ must vanish on $H\setminus H_0$.
\begin{Prop}\label{Frob-Form}
Let $\mathcal{C}=(C_j)_{1\le j\le 2k}$ be an arbitrary $2k$-tuple of $H_0$-conjugacy classes contained in $H\setminus H_0$. Then we have the following counting formula:
\begin{equation}\label{Frob-Form-eq}
\begin{split}
&\#\left\{((A_i,B_i)_{1\le i\le g},(X_j)_{1\le j\le 2k})\in H_0^{2g}\times \prod_{j=1}^{2k}C_j\Big\arrowvert\prod_{i=1}^{g}[A_i,B_i]\prod_{j=1}^{2k}X_j=1\right\}\\
=&|H_0|\sum_{\chi\in\Irr(H_0)^{\sigma}}\Big(\frac{|H_0|}{\chi(1)}\Big)^{2g-2}\prod_{i=1}^{2k}
\frac{|C_i|\tilde{\chi}(C_i)}{\chi(1)}.
\end{split}
\end{equation}
\end{Prop}
Note that the value of the counting formula is independent of the choices of $\tilde{\chi}$ due to the product of $2k$ terms. The proof is given below.

\subsubsection{}
Let us prepare some notations of finite groups. 

Denote by $\mathcal{C}(H)$ the vector space of complex valued class functions on $H$. Put $\hat{H}=\Irr(H)$. Denote by $\mathcal{C}(\hat{H})$ the vector space of linear combinations $\sum_{\chi\in\hat{H}}a_{\chi}\chi$ of irreducible characters, which is the same as $\mathcal{C}(H)$ but will be equipped with a different product operation.

The convolution product $\ast$ in $\mathcal{C}(H)$ is defined  by
\begin{equation}
(f_1\ast f_2)(x)=\sum_{yz=x}f_1(y)f_2(z),\quad f_1,f_2\in\mathcal{C}(H).
\end{equation}

The dot product $\cdot$ in $\mathcal{C}(\hat{H})$ is defined by
\begin{equation}
F_1\cdot F_2(\chi)=F_1(\chi)F_2(\chi),\quad F_1,F_2\in\mathcal{C}(\hat{H}),
\end{equation}
i.e. the coefficient of $\chi$ in the product is the product of the coefficients of $\chi$. 

The Fourier transform $\mathcal{F}:\mathcal{C}(H)\rightarrow\mathcal{C}(\hat{H})$ is defined by
\begin{equation}
\mathcal{F}(f)(\chi)=\sum_{h\in H}\frac{f(h)\chi(h)}{\chi(1)},\quad f\in\mathcal{C}(H).
\end{equation}

We also have the transform $\mathbb{F}:\mathcal{C}(\hat{H})\rightarrow\mathcal{C}(H)$ defined by
\begin{equation}
\mathbb{F}(F)(h)=\sum_{\chi\in\hat{H}}F(\chi)\chi(1)\overline{\chi(h)}.
\end{equation}

We have
\begin{equation}
\mathbb{F}\circ\mathcal{F}=|H|\cdot \Id_{\mathcal{C}(H)},\quad \mathcal{F}\circ\mathbb{F}=|H|\cdot\Id_{\mathcal{C}(\hat{H})}.
\end{equation}

They are compatible with the product operations:
\begin{equation}\label{prod/trans}
\mathcal{F}(f_1)\cdot\mathcal{F}(f_2)=\mathcal{F}(f_1\ast f_2),\quad \mathbb{F}(F_1)\ast\mathbb{F}(F_2)=|H|\cdot\mathbb{F}(F_1\cdot F_2).
\end{equation}

From the equality $\mathbb{F}\circ\mathcal{F}=|H|\cdot \Id_{\mathcal{C}(H)}$ and the definition of the transforms, we deduce that 
\begin{equation}\label{f(1)}
f(1)=\frac{1}{|H|}\sum_{\chi\in\hat{H}}\chi(1)^2\mathcal{F}(f)(\chi).
\end{equation}
for any class function $f$.

\subsubsection{}
We define the function $\mathfrak{n}^g:H\rightarrow\mathbb{C}$ by
\begin{equation}
\mathfrak{n}^g(x)=\#\left\{(A_1,B_1,\cdots,A_g,B_g)\in H_0^{2g}\Big\arrowvert\prod_{i=1}^g[A_i,B_i]=x\right\}.
\end{equation}
We find that $\mathfrak{n}^g\in\mathcal{C}(H)$(identically zero on $H\setminus H_0$) and that $\mathfrak{n}^g=\mathfrak{n}^1\ast\cdots\ast \mathfrak{n}^1$. Denote by $1_{C_i}$ the characteristic function of the class $C_i$. Then
\begin{equation}
\begin{split}
&\#\left\{((A_i,B_i)_{1\le i\le g},(X_j)_{1\le j\le 2k})\in H_0^{2g}\times \prod_{j=1}^{2k}C_j\Big\arrowvert\prod_{i=1}^{g}[A_i,B_i]\prod_{j=1}^{2k}X_j=1\right\}\\
=&(\mathfrak{n}^g\ast 1_{C_1}\ast\cdots\ast 1_{C_{2k}})(1).
\end{split}
\end{equation}
By (\ref{prod/trans}) and (\ref{f(1)}), we have
\begin{equation}
\begin{split}
&(\mathfrak{n}^g\ast 1_{C_1}\ast\cdots\ast 1_{C_{2k}})(1)\\
=&\frac{1}{|H|}\sum_{\chi\in\Irr(H)}\chi(1)^2(\mathcal{F}(\mathfrak{n}^1)(\chi))^g\prod_{i=1}^{2k}\frac{|C_i|\chi(C_i)}{\chi(1)}.
\end{split}
\end{equation}

It is known (\cite[Lemma 3.1.3]{HLR}) that
\begin{equation}
\mathcal{F}(\mathfrak{n}^1)(\chi)=
\big(\frac{|H_0|}{\chi(1)}\big)^2.
\end{equation}

Therefore,
\begin{equation}
\begin{split}
&\frac{1}{|H|}\sum_{\chi_H\in\Irr(H)}\chi_H(1)^2(\mathcal{F}(\mathfrak{n}^1)(\chi_H))^g\prod_{i=1}^{2k}\frac{|C_i|\chi_H(C_i)}{\chi_H(1)}\\
=&|H_0|\sum_{\chi\in\Irr(H_0)^{\sigma}}\left(\frac{|H_0|}{\chi(1)}\right)^{2g-2}\prod_{i=1}^{2k}\frac{|C_i|\tilde{\chi}(C_i)}{\chi(1)},
\end{split}
\end{equation}
where we have used the fact that $\chi_H$ vanishes on $H\setminus H_0$ if it is not an extension of any element of $\Irr(H_0)$. This completes the proof of Proposition \ref{Frob-Form}.

%%%%%%%%%%%%%%%%%%%%%%%
\section{The Group $\GL_n(q)\lb\sigma\rb$}\label{Section-GLnsigma}
In the rest of the article we will write $G=\GL_n$. We will denote by $T$ the maximal torus of diagonal matrices and by $B$ the Borel subgroup of upper triangular matrices. The Weyl group of $G$ defined by $T$ will be denoted by $W=W_G(T)$. We understand that standard Levi subgroups are defined with respect to $T$ and $B$. Recall that $N=[n/2]$.
\subsection{The Group $\GL_n\lb\sigma\rb$ over $\kk$}\hfill

In this subsection, we work over an algebraically closed field $\kk$. 
\subsubsection{}\label{autostand}
Denote by $J_n$ the matrix
$$
(J_n)_{ij}=\delta_{i,n+1-j}
=
\left(
\begin{array}{ccc}
& & 1\\
&\reflectbox{$\ddots$}&\\
1& &
\end{array}
\right)
$$
If $n$ is even, put $t_0=\diag(a_1,\ldots,a_n)$ with $a_i=1$ if $i\le N$ and $a_i=-1$ otherwise. Put $J'_n=t_0J_n$. Write 
\begin{equation}
\mathscr{J}_n:=
\begin{cases}
J'_n & \text{if $n$ is even}\\
J_n & \text{if $n$ is odd}
\end{cases},
\quad
\mathscr{J}'_n:=J_n \text{ if $n$ is even}
\end{equation}
Define $\sigma\in\Aut G$ as sending $g$ to $\mathscr{J}_ng^{-t}\mathscr{J}^{-1}_n$, where $g^{-t}$ is the transpose-inverse of $g$, and define $\sigma'\in\Aut G$ by replacing $\mathscr{J}_n$ with $\mathscr{J}'_n$ in the definition of $\sigma$. They are exterior involutions of $G$. Their fixed points in $G$ are described as follows:
\begin{equation}
\text{ if $n=2N$, then }
\begin{cases}
(G^{\sigma})^{\circ}\cong\Sp_{2N}\\
(G^{\sigma'})^{\circ}\cong\SO_{2N}
\end{cases},\quad
\text{ if $n=2N+1$, then }
(G^{\sigma})^{\circ}\cong\SO_{2N+1}.
\end{equation}
We say that an automorphism is of symplectic type or orthogonal type according to the type of its centraliser.

\subsubsection{}\label{barG}
An involution $\tau$ of $G$ defines a semi-direct product $G\rtimes\lb\tau\rb$. When $n$ is even, there are two $\PGL_n$-conjugacy classes of exterior involutions in $\Aut G$, represented by $\sigma$ and $\sigma'$ respectively (\cite[Lemma 2.9]{LiSe}). We will write $\leftidx{^s\!}{\bar{G}}=G\rtimes\lb\sigma\rb$ and $\leftidx{^o\!}{\bar{G}}=G\rtimes\lb\sigma'\rb$, in order to indicate that the type of the defining automorphism is symplectic or orthogonal. These two semi-direct products are not isomorphic over any field of characteristic different from $2$ (This follows from \cite[Theorem 1.12]{Shu1}).  It is convenient to regard $\sigma$ also as an element of $\leftidx{^o\!}{\bar{G}}$, identified with $t_0\sigma'$, since $\sigma$ and $t_0\sigma'$ induce the same automorphism of $G$. Note however that in $\leftidx{^o\!}{\bar{G}}$ we have $\sigma^2=-1$. Therefore we write $G\lb\sigma\rb$ instead of $G\rtimes\lb\sigma'\rb$. We will write $\bar{G}=G\lb\sigma\rb$ if $n$ is odd or if there is no need to distinguish $\leftidx{^o\!}{\bar{G}}$ and $\leftidx{^s\!}{\bar{G}}$. 

Since $\sigma$ normalises $T$ and $B$, it is semi-simple as an element of $\bar{G}$. Moreover, according to \cite[Proposition 1.22]{DM94}, $\sigma$ is a quasi-central element of $\bar{G}$. It is easy to see that the action of $\sigma$ on $G$ commutes with $F$, and so $F$ can be extended to $\bar{G}$ in such a way that $\sigma$ is an $F$-stable element.

\subsubsection{}\label{s.s.GLnsigma}
Now suppose $\ch\kk\ne 2$. For the moment, we do not distinguish $\leftidx{^o\!}{\bar{G}}$ and $\leftidx{^s\!}{\bar{G}}$. We apply Proposition \ref{DM1.16} to the connected component $G^1=G\sigma$. 

The subtorus $(T^{\sigma})^{\circ}$ consists of the matrices
\begingroup
\allowdisplaybreaks
\begin{align}
\diag(a_1,a_2,\ldots,a_N,a_N^{-1},\ldots,a_2^{-1},a_1^{-1}), &\text{ if $n=2N$,}\\
\diag(a_1,a_2,\ldots,a_N,1,a_N^{-1},\ldots,a_2^{-1},a_1^{-1}), &\text{ if $n=2N+1$,}
\end{align}
\endgroup
with $a_i\in \kk^{\ast}$ for all $i$, and the commutator $[T,\sigma]$ consists of the matrices
\begingroup
\allowdisplaybreaks
\begin{align}
\diag(b_1,b_2,\ldots,b_N,b_N,\ldots,b_2,b_1),&\text{ if $n=2N$,}\\
\diag(b_1,b_2,\ldots,b_N,b_{N+1},b_N,\ldots,b_2,b_1),&\text{ if $n=2N+1$,}
\end{align}
\endgroup
with $b_i\in \kk^{\ast}$ for all $i$. So the elements of  $[T,\sigma]\cap(T^{\sigma})^{\circ}$ are the matrices 
\begingroup
\allowdisplaybreaks
\begin{align}
\diag(e_1,e_2,\ldots,e_N,e_N,\ldots,e_2,e_1),&\text{ if $n=2N$,}\\
\diag(e_1,e_2,\ldots,e_N,1,e_N,\ldots,e_2,e_1),&\text{ if $n=2N+1$,}
\end{align}
\endgroup 
with $e_i=\pm1$ for all $i$. 

The parametrisation of semi-simple conjugacy classes in $G\sigma$ is then given as follows. Denote by $\hat{\kk}$ the set of $(\bs\mu_2\times\bs\mu_2)$-orbits in $\kk^{\ast}$ for the action:
\begin{equation}
\begin{split}
(-1,1):a\mapsto a^{-1},\quad
(1,-1):a\mapsto -a.
\end{split}
\end{equation} 
Then the semi-simple conjugacy classes in $G\sigma$ are parametrised by the $\mathfrak{S}_N$-orbits in $\hat{\kk}{}^N$, or rather, the $N$-tuples $(\hat{a}_1,\ldots,\hat{a}_N)$, $\hat{a}_i\in\hat{\kk}$, up to permutation, whether $n=2N$ or $n=2N+1$. We thus regard the elements of $\hat{\kk}$ as the eigenvalues of a semi-simple element. For example, we may say that $\sigma$ has eigenvalues $(\hat{1},\ldots,\hat{1})$. Beware that the eigenvalues depend on the choice of $\sigma$ in Proposition \ref{DM1.16}. For example, if $n$ is even and if we had used $\sigma'$ in Proposition \ref{DM1.16}, then $\sigma$ would have eigenvalues $(\hat{\ii},\ldots,\hat{\ii})$. In the rest of the article, we will always use $\sigma\in G^1$ when we specify a semi-simple element in terms of elements of $\hat{\kk}$.
\begin{Lem}\label{C-C^2}
If $C\subset G\sigma$ is a $G$-conjugacy class, we write $C^2=\{g^2\mid g\in C\}$, which is a well-defined conjugacy class in $G$. Then,
\begin{itemize}
\item[(i)] $C$ is semi-simple if and only if $C^2$ is semi-simple;
\item[(ii)] The map $C\mapsto C^2$ defines an injection from the set of semi-simple $G$-conjugacy classes in $G\sigma$ to the set of $\sigma$-stable semi-simple conjugacy classes in $G$.
\end{itemize}
\end{Lem}
\begin{proof}
Part (i) is true for characteristic reason. Part (ii) follows from the above parametrisation.
\end{proof}

\subsubsection{}\label{sigstaLevi}
Let $P\subset G$ be a parabolic subgroup and let $L\subset P$ be a Levi factor. We will be interested in those $L\subset P$ such that $N_{\bar{G}}(L,P)$ meets the connected component $G\sigma$. According to \S \ref{Para-of-G}, this is to consider the $\sigma$-stable Levi factors of $\sigma$-stable parabolic subgroups. If $P$ is a $\sigma$-stable standard parabolic subgroup with respect to $B$, then its unique Levi factor containing $T$ is of the form

\begin{equation}\label{LI}
  \renewcommand{\arraystretch}{1.3}
\left(\!
  \begin{array}{cccccccc}
    \multicolumn{1}{c|}{L_1} & & & & & & &\\
    \cline{1-1}
    & \ddots & & & & & &\\
    \cline{3-3}
    & & \multicolumn{1}{|c|}{L_s\!} &  \multicolumn{2}{c}{} & & & \\
    \cline{3-5}
    & & & \multicolumn{2}{|c|}{} & \multirow{2}*{} & & \\
    & & & \multicolumn{2}{|c|}{\raisebox{.6\normalbaselineskip}[0pt][0pt]{$~\GL_{n_0}~$}} & & & \\
    \cline{4-6}
    & & & & & \multicolumn{1}{|c|}{L_s\!} & & \\
    \cline{6-6}
    & & & & & & \ddots & \\
    \cline{8-8}
    & & & & & & &\multicolumn{1}{|c}{L_1\!\!}\\
     \end{array}
  \right),
\end{equation}
where $L_i\cong\GL_{n_i}$, for some integers $n_i$, $1\le i\le s$. Therefore, up to the conjugation by $G$, we only need to consider these Levi subgroups. Note that for these groups, $N_{\bar{G}}(L,P)=L\sqcup L\sigma$.

\subsection{The Group $\GL_n\lb\sigma\rb$ over $\mathbb{F}_q$}\hfill

In the the rest of this section, we will fix an odd prime power $q$, $\kk$ will denote an algebraic closure of $\mathbb{F}_q$ and $G=\GL_n$ over $\kk$ is equipped with the Frobenius endomorphism $F$ that sends each entry of a matrix to its $q$-th power. Taking the fixed points of $F$ in $G\lb\sigma\rb$, we get a finite group $G(q)\lb\sigma\rb$. Note that $F$ acts trivially on $N_G(L)/L$ if $L$ is a Levi subgroup containing $T$, in particular, it acts trivially on $W$.

\subsubsection{}\label{GLnSigma-F-Class}
By the theorem of Lang-Steinberg, a $G$-conjugacy class in $G\sigma$ contains an element of $G^F\sigma$ if and only if it is $F$-stable. Let $C\subset G\sigma$ be a semi-simple conjugacy class with a representative $t\sigma\in(T^{\sigma})^{\circ}\sigma$. According to the parametrisation of semi-simple conjugacy classes, $C$ is $F$-stable if and only if there exist $w\in W^{\sigma}$ and $s\in [T,\sigma]\cap(T^{\sigma})^{\circ}$ such that $F(t)=wtw^{-1}s$ (See \S \ref{s.s.GLnsigma}). We will however only be concerned with those semi-simple conjugacy classes that have representatives in $(T^{\sigma})^{\circ F}\sigma$. That is, $w=1$ and $s=1$ in the above equation. Now let $C$ be such a class with representative $t\sigma\in (T^{\sigma})^{\circ F}\sigma$. If we represent $C$ by the $N$-tuple $(\hat{a}_1,\ldots,\hat{a}_N)$, $\hat{a}_i\in\hat{\kk}$, then $a_i\in\mathbb{F}_q^{\ast}$ for all $1\le i\le N$. 

\begin{Lem}(\cite[Lemma 6.2.1]{Shu2})\label{Shu2Lem6.2.1}
Denote by $N_+$ (resp. $N_-$) the multiplicity of $\hat{1}$ (resp. $\hat{\mathfrak{i}}$) in $(\hat{a}_1,\ldots,\hat{a}_N)$. Then the centraliser of $t\sigma$ in $G$ is isomorphic to
\begingroup
\allowdisplaybreaks
\begin{align*}
\Sp_{2N_+}\times\Ort_{2N_-}\times\prod_i\GL_{n_i} &\text{ if $n$ is even and,}\\
\Ort_{2N_++1}\times\Sp_{2N_-}\times\prod_i\GL_{n_i} &\text{ if $n$ is odd,}
\end{align*}
\endgroup
for some integers $n_i$.
\end{Lem}

In general $C^F$ is not a single $G^F$-conjugacy class. By \cite[Proposition 4.2.14]{DM20}, the number of $G^F$-conjugacy classes contained in $C^F$ is equal to the number of connected components of $C_G(t\sigma)$. If $n$ is odd, then $C_G(t\sigma)$ always has two connected components. The $G^F$-class of $t\sigma$ has a representative of the form
\begin{equation}\label{C_+}
\diag(a_1,a_2,\ldots,a_N,1,a_N^{-1},\ldots,a_2^{-1},a_1^{-1})\sigma,
\end{equation}
with every $a_i\in\mathbb{F}_q$, while the other $G^F$-class in $C^F$ is represented by
\begin{equation}\label{C_-}
\diag(a_1,a_2,\ldots,a_N,b,a_N^{-1},\ldots,a_2^{-1},a_1^{-1})\sigma,
\end{equation}
with $b\in\mathbb{F}_q^{\ast}\setminus(\mathbb{F}_q^{\ast})^2$.

\subsubsection{}\label{GLnSigma-F-Levi}
Let $I$ be a subset of $\Delta(T,B)$ and let $L_I$ be the corresponding standard Levi subgroup. There are positive integers $\{n_i\mid i\in\Gamma_I\}$ indexed by a finite set $\Gamma_I$, such that $L_I\cong\prod_{i\in\Gamma_I}\GL_{n_i}$. Then $Z_{L_I}\cong \mathbb{G}_m^{\Gamma_I}$, i.e. direct product of copies of $\mathbb{G}_m$ indexed by $\Gamma_I$. For any $r\in\mathbb{Z}_{>0}$, put $\Gamma_{I,r}=\{i\in\Gamma_I\mid n_i=r\}$ and put $N_r=|\Gamma_{I,r}|$. Then the $G$-conjugacy class of $L_I$ is uniquely determined by the sequence $(N_r)_{r\in\mathbb{Z}_{>0}}$, and $W_G(L_I)\cong\prod_r\mathfrak{S}_{N_r}$.

Suppose that $I$ is $\sigma$-stable, so that the associated $\sigma$-stable standard Levi subgroup $L_I$ is of the form (\ref{LI}). Now the action of $\sigma$ on $Z_{L_I}$ induces an involution on $\Gamma_I$. There is at most one element of $\Gamma_I$ fixed by $\sigma$, and we denote it by $0$ so that it corresponds to the direct factor $\GL_{n_0}$ in (\ref{LI}). In case no element of $\Gamma_I$ is fixed by $\sigma$, we define $n_0=0$. For any $r\in\mathbb{Z}_{>0}$, put $N_r'=N_r/2$ if $n_0\ne r$ and $N_r'=(N_r-1)/2$ if $n_0=r$, then $W_{G}(L_I)^{\sigma}\cong\prod_r\mathfrak{W}_{N_r'}$. 

The action of an element of $W_{G}(L_I)^{\sigma}$ on $L_I$ (up to an inner automorphism of $L_I$) can be easily visualised. Let $r\in\mathbb{Z}_{>0}$ and let $\tau$ be a signed cycle of size $d$ in $\mathfrak{W}_{N_r'}$. Then it acts on a subgroup of $L_I$ that is isomorphic to $(\GL_r\times\GL_r)^d$. Depending on whether $\tau$ is a positive cycle or a negative cycle (see \S \ref{signed-partitions}), its action can be schematically described as 
\begin{equation}\label{ActionSymetrique}
  \renewcommand{\arraystretch}{1.3}
\left(\!
  \begin{array}{cccccc}
    \multicolumn{1}{c|}{\rn{1}{\!\substack{\GL_{r}}}\!\!} & & & & & \\
    \cline{1-1}
    & \rn{2}{\ddots} & & & & \\
    \cline{3-3}
    & & \multicolumn{1}{|c|}{\rn{3}{\!\substack{\GL_{r}}}\!\!} &  \multicolumn{2}{c}{} & \\
    \cline{3-4}
    & & &  \multicolumn{1}{|c|}{\rn{4}{\!\substack{\GL_{r}}}\!\!} & & \\
    \cline{4-4}
    & & & & \rn{5}{\ddots} & \\
    \cline{6-6}
    & & & & & \multicolumn{1}{|c}{\rn{6}{\!\!\substack{\GL_{r}}}\!\!}\\
     \end{array}
  \right),\quad\text{ and }
  \left(\!
  \begin{array}{cccccc}
    \multicolumn{1}{c|}{\rn{1'}{\!\substack{\GL_{r}}}\!\!} & & & & & \\
    \cline{1-1}
    & \rn{2'}{\ddots} & & & & \\
    \cline{3-3}
    & & \multicolumn{1}{|c|}{\rn{3'}{\!\substack{\GL_{r}}}\!\!} &  \multicolumn{2}{c}{} & \\
    \cline{3-4}
    & & &  \multicolumn{1}{|c|}{\rn{4'}{\!\substack{\GL_{r}}}\!\!} & & \\
    \cline{4-4}
    & & & & \rn{5'}{\ddots} & \\
    \cline{6-6}
    & & & & & \multicolumn{1}{|c}{\rn{6'}{\!\!\substack{\GL_{r}}}\!\!}\\
     \end{array}
  \right),
  \begin{tikzpicture}[overlay,remember picture]
 \node [below= -0.1 of 1] (1a) {};
 \node [right= 0 of 1] (1b) {};
 \node [left=-0.2 of 2] (2a') {};
 \node [below=-0.15 of 2a'] (2a) {};
 \node [below= -0.2 of 2] (2b') {};
 \node [left= -0.15 of 2b'] (2b) {};
 \node [left= 0 of 3] (3a') {};
 \node [above= -0.1 of 3a'] (3a) {};
 \node [above= 0.1 of 3] (3b) {};
 
 \node [below= -0.1 of 4] (4a) {};
 \node [right= 0 of 4] (4b') {};
 \node [above= -0.1 of 4b'] (4b) {};
 \node [left=-0.2 of 5] (5a') {};
 \node [below=-0.15 of 5a'] (5a) {};
 \node [below= -0.2 of 5] (5b') {};
 \node [left= -0.2 of 5b'] (5b) {};
 \node [left= 0 of 6] (6a') {};
 \node [above= -0.1 of 6a'] (6a) {};
 \node [above= 0.1 of 6] (6b) {};
 
 \node [below= -0.1 of 1'] (1'a) {};
 \node [right= 0 of 1'] (1'b) {};
 \node [left=-0.2 of 2'] (2'a') {};
 \node [below=-0.15 of 2'a'] (2'a) {};
 \node [below= -0.2 of 2'] (2'b') {};
 \node [left= -0.15 of 2'b'] (2'b) {};
 \node [left= 0 of 3'] (3'a') {};
 \node [above= -0.1 of 3'a'] (3'a) {};
 \node [below= 0 of 3'] (3'b) {};
 
 \node [right= -0.1 of 4'] (4'a) {};
 \node [above= 0 of 4'] (4'b) {};
 \node [above= -0.45 of 5'] (5'a') {};
 \node [right= -0.2 of 5'a'] (5'a) {};
 \node [right= -0.2 of 5'] (5'b') {};
 \node [below= -0.2 of 5'b'] (5'b) {};
 \node [left= 0 of 6'] (6'a) {};
 \node [above= 0 of 6'] (6'b) {};
 
        \draw [->] (1a) to [out=265,in=190, looseness=1.3] (2a);
        \draw [->] (2b) to [out=265,in=190, looseness=1.3] (3a);
        \draw [->] (3b) to [out=85,in=10, looseness=1.3] (1b);
        
        \draw [<-] (4a) to [out=265,in=190, looseness=1.3] (5a);
        \draw [<-] (5b) to [out=260,in=190, looseness=1.3] (6a);
        \draw [<-] (6b) to [out=85,in=2, looseness=1.3] (4b);
        
        \draw [->] (1'a) to [out=265,in=190, looseness=1.3] (2'a);
        \draw [->] (2'b) to [out=265,in=190, looseness=1.3] (3'a);
        \draw [->] (3'b) to [out=265,in=185, looseness=1.3] (6'a);
        
        \draw [<-] (4'a) to [out=10,in=80, looseness=1.3] (5'a);
        \draw [<-] (5'b) to [out=15,in=80, looseness=1.3] (6'b);
        \draw [->] (4'b) to [out=85,in=5, looseness=1.3] (1'b);
 \end{tikzpicture}
\end{equation}
respectively. If $r=n_0$, then $\tau$ simply ignores the component $\GL_{n_0}$.

For each $w\in W_{G}(L_I)^{\sigma}$, there exists $\dot{w}\in(G^{\sigma})^{\circ}$ that represents $w$. For example, it is easy to see that in (\ref{ActionSymetrique}), one can find explicit matrices in $\GL_{2rd}$ fixed by $\sigma$ representing the corresponding cycles in $\mathfrak{W}_{N_r'}$. This implies that the inclusion $$W_{(G^{\sigma})^{\circ}}((L_I^{\sigma})^{\circ})\hookrightarrow W_{G}(L_I)^{\sigma}$$is in fact an isomorphism.  Now the groups $L_{I,w}$ defined for $w\in W_{(G^{\sigma})^{\circ}}((L_I^{\sigma})^{\circ})$ in \S \ref{Prop1.40} make sense if we regard $w$ as an element of $W_{G}(L_I)^{\sigma}$. The discussions in \S \ref{Prop1.40} give:
\begin{Prop}
The $G^F$-conjugacy classes of the $F$-stable $G$-conjugates of $L_I\lb\sigma\rb$ are in bijection with the conjugacy classes of $W_{G}(L_I)^{\sigma}$. 
\end{Prop}

\subsection{Deligne-Lusztig inductions}
\subsubsection{}\label{DL-Induction}
Let $H$ be a connected reductive group defined over $\mathbb{F}_q$ equipped with a geometric Frobenius endomorphism $F$. Let $L\subset H$ be an $F$-stable Levi subgroup. The \textit{Deligne-Lusztig induction} $R^H_L$ (\cite[\S 9.1]{DM20}) is a $\ladic$-linear map from the set of $L^F$-invariant functions on $L^F$ to the set of $H^F$-invariant functions on $H^F$ (invariant for the conjugation actions). 

Let $T_H$ be an $F$-stable maximal torus of $H$. Denote by $W_H$ the Weyl group of $H$ defined by $T_H$. By \S \ref{ConnG-L_I-F}, to each $w\in W_H$, we can associate an $F$-stable maximal torus $T_w\subset H$ in such a way that $T_H$ is associated to $1\in W_H$. For any $w\in W_H$, we denote by $\mathbf{1}$ the trivial character of $T_w^F$, then the \textit{Green function} is an $H^F$-invariant function on the subset $H^F_u\subset H^F$ of unipotent elements, defined by
\begin{equation}\label{Green-Function-DL}
Q_{T_w}^H(u):=R_{T_w}^H\mathbf{1}|_{H^F_u}.
\end{equation}
Such a function only depends on the $H^F$-conjugacy class of $T_w$, or rather, the $F$-conjugacy class of $w$.

Let $\tau$ be an automorphism of finite order of $H$ that commutes with $F$ and let $L\subset H$ be an $F$-stable and $\tau$-stable Levi factor of a $\tau$-stable parabolic subgroup $P\subset H$. Then $N_{H\lb\tau\rb}(L,P)=L\lb\tau\rb$. The \textit{generalised Deligne-Lusztig induction} $R^{H\tau}_{L\tau}$ (\cite[\S 2]{DM94}) is a $\ladic$-linear map from the set of $L^F$-invariant functions on $L^F\tau$ to the set of $H^F$-invariant functions on $H^F\tau$.

\subsubsection{}
By \S \ref{Prop1.40} and Remark \ref{Wsigma=WGsigma}, to each $w\in W^{\sigma}$, we can associate an $F$-stable and $\sigma$-stable maximal torus  $T_w$ contained in a $\sigma$-stable Borel subgroup of $G$, in such a way that $T$ is associated to $1\in W^{\sigma}$.  If $B_w\subset G$ is some $\sigma$-stable Borel subgroup containing $T_w$, then $\bar{T}_w:=N_{\bar{G}}(T_w,B_w)=T_w\sqcup T_w\sigma$ and the $G^F$-conjugacy class of $\bar{T}_w$ is determined by the conjugacy class of $w\in W^{\sigma}$. Let $\theta\in\Irr(T_w^F)^{\sigma}$ and denote by $\tilde{\theta}\in\Irr(T_w^F\lb\sigma\rb)$ an extension of $\theta$. By abuse of notation, we may also denote by $\tilde{\theta}$ its restriction on $T_w^F\sigma$. Computation of Deligne-Lusztig inductions is reduced to Green functions according to the character formula below.
\begin{Prop}(\cite[Proposition 2.6]{DM94})\label{char-formula}
Let $s\sigma\in G^F\sigma$ be semi-simple and let $u\in C_G(s\sigma)^{\circ F}$. Then
\begin{equation}\label{eq-char-formula}
R^{G\sigma}_{T_w\sigma}\tilde{\theta}(us\sigma)=\frac{|(T_w^{\sigma})^{\circ F}|}{|T_w^F|\cdot|C_G(s\sigma)^{\circ F}|}\sum_{\{h\in G^{F}\mid hs\sigma h^{-1}\in T_w\sigma\}}Q^{C_G(s\sigma)^{\circ}}_{C_{h^{-1}T_wh}(s\sigma)^{\circ}}(u)\tilde{\theta}(hs\sigma h^{-1}),
\end{equation}
where $Q^{C_G(s\sigma)^{\circ}}_{C_{h^{-1}T_wh}(s\sigma)^{\circ}}(u)$ is the Green function associated to the connected reductive group $C_G(s\sigma)^{\circ}$ and its $F$-stable maximal torus $C_{h^{-1}T_wh}(s\sigma)^{\circ}$.
\end{Prop}
\begin{Rem}
The extra factor $|(T_w^{\sigma})^{\circ F}|$ compared with \cite[Proposition 2.6]{DM94} is due to different normalisations of the Green functions $Q^{C_G(s\sigma)^{\circ}}_{C_{h^{-1}T_wh}(s\sigma)^{\circ}}(u)$.
\end{Rem}

An invariant function on $G^F\sigma$ is called \textit{uniform} if it is a linear combination of $R_{T_w\sigma}^{G\sigma}\tilde{\theta}$ for various $w$ and $\tilde{\theta}$.
\begin{Prop}(\cite[Proposition 6.4]{DM15})\label{s.s.unif}
The characteristic function of a semi-simple conjugacy class in $G^F\sigma$ is uniform.
\end{Prop}

\subsubsection{}\label{A^F->B}
Fix semi-simple element $s\sigma\in G^F\sigma$ and write $L':=C_G(s\sigma)^{\circ}$. Let $\tau\in W^{\sigma}$. Write $$A_{\tau}=\{h\in G\mid hs\sigma h^{-1}\in T_{\tau}\sigma\},\quad A_{\tau}^F=A_{\tau}\cap G^F.$$If $A^F_{\tau}$ is non empty, we may assume that $s\sigma\in T_{\tau}\sigma$, since the problem we will consider is not affected by conjugation by $G^F$. So in particular $A^F_{\tau}$ contains the unit $1$. For any $h\in A_{\tau}^F$, $C_{h^{-1}T_{\tau}h}(s\sigma)^{\circ}$ is an $F$-stable maximal torus of $L'$ by Proposition \ref{1.11DM94}. In particular, $T'_{\tau}:=C_{T_{\tau}}(s\sigma)^{\circ}$ is an $F$-stable maximal torus of $L'$. The $L^{\prime F}$-conjugacy classes of the $F$-stable maximal tori of $L'$ are parametrised by the $F$-conjugacy classes of $W':=W_{L'}(T_{\tau}')$. Then the map $h\mapsto C_{h^{-1}T_{\tau}h}(s\sigma)^{\circ}$  induces a map from $A_{\tau}^F$ to the set of $F$-conjugacy classes of $W'$. Denote by $B_{\tau}$ the image of this map so that we have a surjection $A_{\tau}^F\rightarrow B_{\tau}$. Let $\nu\in W'$ represent an $F$-conjugacy class of $W'$ and denote by $A^F_{\tau,\nu}$ the inverse image in $A_{\tau}^F$ of this $F$-conjugacy class.

We will fix $g_{\tau}$ such that $T_{\tau}=g_{\tau}Tg^{-1}_{\tau}$. Write $\dot{\tau}=g_{\tau}^{-1}F(g_{\tau})$ and $s_0\sigma=g_{\tau}^{-1}s\sigma g_{\tau}$ and $L'_0=C_G(s_0\sigma)^{\circ}$. Let $W'_0$ denote the Weyl group of $L'_0$ defined by $C_T(s_0\sigma)^{\circ}$. Then the $F$-conjugacy classes of $W'$ are in natural bijection with the $\tau$-conjugacy classes of $W'_0$. We will therefore regard $B_{\tau}$ as a set of $\tau$-conjugacy classes of $W'_0$. By Corollary \ref{Nsigma-in-N-2}, $W'_0$ is naturally a subgroup of $W^{\sigma}$.

\subsubsection{}
We will need the following results.
\begin{Lem}\label{Shu210.2.2b}
There is a natural isomorphism: $$N_G(T\sigma)/T\lisom W^{\sigma}.$$
\end{Lem}
\begin{proof}
Since $T\sigma T\sigma= T$, we have  $N_G(T\sigma)\subset N_G(T)$. It is easy to see that $\xi T\sigma \xi^{-1}=T\sigma$ is equivalent to $\xi\sigma(\xi)^{-1}\in T$ for any $\xi\in N_G(T)$.
\end{proof}
\begin{Lem}\label{Shu210.2.2a}
We have $$\{x\in G\mid xs_0\sigma x^{-1}\in T\sigma\}=N_{G}(T\sigma)L_0'.$$
\end{Lem}
\begin{proof}
Since $xs_0\sigma x^{-1}$ lies in $T\sigma$, it normalises $T$ and besides, $$C_T(xs_0\sigma x^{-1})^{\circ}=C_T(s_0\sigma)^{\circ}=T':=C_T(\sigma)^{\circ}.$$By Proposition \ref{1.11DM94}, $T'$ is a maximal torus of $L'$. Now $s\sigma$ normalises $x^{-1}Tx$ and so $C_{x^{-1}Tx}(s_0\sigma)^{\circ}$ is also a maximal torus of $L'$. There exists $l\in L'$ such that $C_{x^{-1}Tx}(s_0\sigma)^{\circ}=lT'l^{-1}$. Note that $$C_{x^{-1}Tx}(s_0\sigma)^{\circ}=x^{-1}C_T(xs_0\sigma x^{-1})^{\circ}x=x^{-1}T'x.$$ We deduce that $xl$ normalises $T'$, thus normalises $T$ by Proposition \ref{1.25DM94}. Write $\xi=xl$. From the relation $xs_0\sigma x^{-1}\in T\sigma$, we see that $\sigma(\xi)\in \xi T$. By the proof of Lemma \ref{Shu210.2.2b}, we have $\xi\in N_G(T\sigma)$. Therefore, $x\in N_G(T\sigma)L'$. Conversely, it is easy to see that every element $x\in N_G(T\sigma)L'$ satisfies $xs_0\sigma x^{-1}\in T\sigma$.
\end{proof}

\begin{Prop}(\cite[Proposition 10.2.13 (ii)]{Shu2})\label{Prop-A^F->B}
We have $$|A^F_{\tau,\nu}|=|T_{\tau}^F|\cdot|L^{\prime F}|\cdot|(T_{\tau}^{\sigma})^{F}|^{-1}\zz_{\tau}\zz_{\nu}^{-1},$$where $\zz_{\tau}$ is the cardinality of the centraliser of $\tau$ in $W^{\sigma}$ and $\zz_{\nu}$ is the cardinality of the stabiliser of an element of the $\tau$-conjugacy class $\nu$ under the $\tau$-twisted action.
\end{Prop}

\subsection{Irreducible characters of $\GL_n(q)\lb\sigma\rb$}
\subsubsection{}\label{unip-char-H}
Let $H=\prod_i\GL_{n_i}$ for some positive integers $n_i$ and suppose that $H$ is defined over $\mathbb{F}_q$ equipped with a Frobenius $F$. Fix an $F$-stable maximal torus $T_H\subset H$ and denote by $W_H$ the Weyl group of $H$ defined by $T_H$. By \S \ref{ConnG-L_I-F}, to each $w\in W_H$, we can associate an $F$-stable maximal torus $T_w$ of $H$ in such a way that $T_H$ is associated to $1\in W$. The $H^F$-conjugacy class of $T_w$ is determined by the $F$-conjugacy class of $w\in W_H$. The natural action of $F$ on $W_H$ induces an action on $\Irr(W_H)$. Denote by $\Irr(W_H)^F$ the set of $F$-stable irreducible characters of $W_H$. Let $\varphi\in\Irr(W_H)^F$, then it can be extended to $\tilde{\varphi}\in\Irr(W_H\rtimes\lb F\rb)$. Such an extension is not unique. We choose such an extension and put $$R^{H}_{\varphi}\mathbf{1}:=|W_H|^{-1}\sum_{w\in W_H}\tilde{\varphi}(wF)R_{T_w}^{H}\mathbf{1},$$where $\mathbf{1}$ is the trivial character of $T_w^F$.
\begin{Thm}(\cite[Theorem 2.2]{LS})
For any $\varphi\in\Irr(W_H)^F$ and some choice of $\tilde{\varphi}$, the virtual character $R^H_{\varphi}\mathbf{1}$ is an irreducible character of $H^F$, and $R^H_{\varphi}\mathbf{1}\ne R^H_{\varphi'}\mathbf{1}$ if $\varphi\ne\varphi'$.
\end{Thm}
Let $\theta\in\Hom(H^F,\bar{\mathbb{Q}}_{\ell}^{\ast})$, i.e. a linear character of $H^F$. Then $\theta\otimes R^H_{\varphi}\mathbf{1}$ is also an irreducible character of $H^F$. Denote by $\theta_{T_w}$ the restriction of $\theta$ to $T_w^F$ for any $w\in W_H$. Then $$\theta\otimes R^H_{\varphi}\mathbf{1}=R^H_{\varphi}\theta:=|W_H|^{-1}\sum_{w\in W_H}\tilde{\varphi}(wF)R_{T_w}^{H}\theta_{T_w}.$$

\subsubsection{}
Let $M\subset G$ be an $F$-stable Levi subgroup of $G$. Fix an $F$-stable maximal torus $T_M\subset M$ and let $T_w$, $w\in W_M$ be defined as in the previous paragraph. Denote by $\Irr_{reg}(M^F)$ the set of regular linear characters of $M^F$ (See \cite[\S 3.1 (a), (b)]{LS}). Concretely, if we choose an isomorphism $M^F\cong\prod_i\GL_{n_i}(q^{d_i})$, for some positive intergers $n_i$, $d_i$, and write a linear character $\theta\in\Irr(M^F)$ as $(\theta_i)_i$ for some $\theta_i\in\Hom(\GL_{n_i}(q^{d_i}),\bar{\mathbb{Q}}_{\ell}^{\ast})$, then $\theta$ is regular if and only if $\theta_i\ne\theta_j$ whenever $i\ne j$ and $\theta_i^{q^r}\ne\theta_i$ for any $1\le r< d_i$ and any $i$.

For any connected reductive group $H$ defined over $\mathbb{F}_q$, define $\epsilon_H:=(-1)^{rk_H}$ where $rk_H$ is the $\mathbb{F}_q$-rank of $H$. For any $\varphi\in\Irr(W_M)^F$ and $\theta\in\Hom(M^F,\bar{\mathbb{Q}}_{\ell}^{\ast})$, choose an extension $\tilde{\varphi}\in\Irr(W_M\rtimes\lb F\rb)$, and put $$R^{G}_{\varphi}\theta:=\epsilon_G\epsilon_M|W_M|^{-1}\sum_{w\in W_M}\tilde{\varphi}(wF)R_{T_w}^{G}\theta_{T_w}.$$Note that $R_{\varphi}^G\theta=\epsilon_G\epsilon_MR^G_M(R^M_{\varphi}\theta)$ (\cite[Proposition 9.1.8]{DM20}).
\begin{Thm}(\cite[Theorem 3.2]{LS})\label{LSThm3.2}
If $\theta$ lies in $\Irr_{reg}(M^F)$, then for some choice of $\tilde{\varphi}$, the virtual character $R^G_{\varphi}\theta$ is an irreducible character of  $G^F$. Moreover, all irreducible characters of $G^F$ are of the form $R^G_{\varphi}\theta$ for a triple $(M,\varphi,\theta)$ as above. The characters associated to the triples $(M,\varphi,\theta)$ and $(M',\varphi',\theta')$ are distinct if and only if one of the following conditions is satisfied
\begin{itemize}
\item[-] $(M,\theta)$ and $(M',\theta')$ are not $G^F$-conjuguate;
\item[-] $(M,\theta)=(M',\theta')$ and $\varphi\ne\varphi'$.
\end{itemize}
\end{Thm}

\subsubsection{}\label{Quad-Unip}
Let $n_0\in\mathbb{Z}_{>0}$, and let $n_+$, $n_-\in\mathbb{Z}_{\ge 0}$ be such that $n_++n_-=n_0$. Let $F$ be the Frobenius of $\GL_{n_0}$ that sends each entry of a matrix to its $q$-th power. Let $M_{00}\subset\GL_{n_0}$ be a standard Levi subgroup isomorphic to $\GL_{n_+}\times\GL_{n_-}$, i.e. 

\begin{equation}\label{q.u.M}
  \renewcommand{\arraystretch}{1.2}
M_{00}= \left(
  \begin{array}{ c c | c c }
     & & \multicolumn{2}{c}{\raisebox{-.6\normalbaselineskip}[0pt][0pt]{$0$}}  \\
     \multicolumn{2}{c|}{\raisebox{.6\normalbaselineskip}[0pt][0pt]{$\GL_{n_+}$}} &  & \\
    \cline{1-4}
     & & & \\
     \multicolumn{2}{c|}{\raisebox{.6\normalbaselineskip}[0pt][0pt]{$0$}} & \multicolumn{2}{c}{\raisebox{.6\normalbaselineskip}[0pt][0pt]{$\GL_{n_-}$}}
     \end{array}
  \right).
\end{equation}
With respect to the isomorphism $M_{00}^F\cong\GL_{n_+}(q)\times\GL_{n_-}(q)$, we define a linear character $\theta\in\Irr(M_{00}^F)$ to be $(\mathbf{1}\circ\det,\eta\circ\det)$, where $\mathbf{1}$ is the trivial character of $\mathbb{F}_q^{\ast}$ and $\eta$ is the order 2 irreducible character of $\mathbb{F}_q^{\ast}$. Now $W_{M_{00}}$ is isomorphic to $\mathfrak{S}_{n_+}\times\mathfrak{S}_{n_-}$ and so $\Irr(W_{M_{00}})$ is in bijection with $\mathcal{P}(n_+)\times\mathcal{P}(n_-)$. By Theorem \ref{LSThm3.2}, each $(\mu_+,\mu_-)\in \mathcal{P}(n_+)\times\mathcal{P}(n_-)$ defines an irreducible character of $\GL_{n_0}(q)$, called a \textit{quadratic-unipotent character}. By \cite[Lemma 5.2.1]{Shu2}, this character is $\sigma$-stable (Here $\sigma$ is the automorphism of $\GL_{n_0}$ defined in the same way as it is defined for $\GL_n$).

\subsubsection{}\label{SigSt-Irr}
Let $M\subset G$ be a Levi subgroup of the form $L_{I,w}$ as defined in \S \ref{Prop1.40}, which is an $F$-stable and $\sigma$-stable Levi factor of some $\sigma$-stable parabolic subgroup $P\subset G$. It can be written as $M_0\times M_1$ with $M_0\cong\GL_{n_0}$ and $M_1\cong\prod_i(\GL_{n_i}\times\GL_{n_i})$ following the notations of \S \ref{sigstaLevi}. The actions of $F$ and $\sigma$ respect the isomorphism $M\cong M_0\times M_1$ and we will also denote by $F$ and $\sigma$ their restrictions to $M_0$ or $M_1$.

Let $T_M\subset M$ be an $F$-stable and $\sigma$-stable maximal torus. Then we can write $T_M=T_0\times T_1$ with $T_0\subset M_0$ and $T_1\subset M_1$. Denote by $W_0$ (resp. $W_1$) the Weyl group of $M_0$ (resp. $M_1$) defined by $T_0$ (resp. $T_1$). Then $\sigma$ induces an action on $W_1$ and so an action on $\Irr(W_1)$. Let $\varphi\in\Irr(W_1)^F\cap\Irr(W_1)^{\sigma}$ and denote by $\tilde{\varphi}$ an extension of $\varphi$ to $W_1\rtimes\lb F\rb$. Let $\theta$ be a linear character of $M_1^F$ that is $\sigma$-stable. Then $\chi_1:=R^{M_1}_{\varphi}\theta$ as defined in \S \ref{unip-char-H} is a $\sigma$-stable irreducible character of $M_1^F$. Let $\chi_0$ be a quadratic-unipotent character of $M_0^F$. It is induced from a Levi subgroup $M_{00}$ of $M_0$ as in \S \ref{Quad-Unip}.
\begin{Thm}(\cite[Proposition 5.2.2, Proposition 5.2.3]{Shu2})\label{sig-st-chi}
Suppose that $\mathbf{1}\boxtimes\eta\boxtimes\theta$ is a regular linear character of $M_{00}^F\times M_1^F$. Then $\epsilon_G\epsilon_M R^G_M(\chi_1\boxtimes\chi_0)$ is a $\sigma$-stable irreducible character of $G^F$. Moreover, every $\sigma$-stable irreducible character of $G^F$ is of this form.
\end{Thm}

Let $(\mu_+,\mu_-)$ be the 2-partition defining the quadratic-unipotent character $\chi_0$. Let $m_+$ and $m_-$ be the non negative integers such that $(m_+,m_+-1,\ldots,1,0)$ and $(m_-,m_--1,\ldots,1,0)$ are the 2-cores of $\mu_+$ and $\mu_-$ respectively. Write $$m=m_+(m_++1)/2+m_-(m_-+1)/2.$$
\begin{Prop}(\cite[Corollary 11.1.2]{Shu2})\label{Shu2Cor11.1.2}
With the notations in the above theorem, the $\sigma$-stable irreducible character $\epsilon_G\epsilon_M R^G_M(\chi_1\boxtimes\chi_0)$ extends to a uniform function on $G^F\sigma$ if and only if $m\le 1$. Moreover, the extensions of these characters (up to a sign) form a basis of the vector space of all uniform functions on $G^F\sigma$.
\end{Prop}

\subsubsection{}\label{Irr-reg-M1}
In the rest of this section we will restrict ourselves to the special case where $M=L_I$ is a $\sigma$-stable standard Levi subgroup. Then we have $\epsilon_G=\epsilon_M$, $M_0^F\cong\GL_{n_0}(q)$,  $M_1^F\cong\prod_i(\GL_{n_i}(q)\times\GL_{n_i}(q))$, $W_1\cong\prod_i(\mathfrak{S}_{n_i}\times\mathfrak{S}_{n_i})$, and $W_1^{\sigma}\cong\prod_i\mathfrak{S}_{n_i}$. The action of $F$ on $W_1$ is trivial. A $\sigma$-stable linear character of $M_1^F$ is of the form $(\theta_i,\theta_i{\!}^{-1})_i$, where $\theta_i$ is a linear character of $\GL_{n_i}(q)$ for each $i$. Then $\mathbf{1}\boxtimes\eta\boxtimes\theta$ is a regular linear character of $M_{00}^F\times M_1^F$ if and only if $\theta$ is regular in the following sense
\begin{Defn}\label{reg-M_1}
We say that a $\sigma$-stable linear character $\theta$ of $M_1^F$ is regular if $\theta_i\ne\theta_j^{\pm 1}$ whenever $i\ne j$ and $\theta_i^2\ne\mathbf{1}$ for all $i$. The set of $\sigma$-stable regular linear characters of $M_1^F$ is denoted by $\Irr_{\reg}^{\sigma}(M_1^F)$.
\end{Defn}
We will give a decomposition formula expressing the extension of $R^G_M(\chi_1\boxtimes\chi_0)$ to $G^F\sigma$ as a linear combination of generalised Deligne-Lusztig characters when $m\le 1$. Therefore we assume $m\le 1$ in what follows.

\subsubsection{}\label{decIrr}
The set $\Irr(W_1)^{\sigma}$ is in bijection with $\prod_i\mathcal{P}(n_i)$, which is then in bijection with $\Irr(W_1^{\sigma})$. We define a bijection 
\begin{equation}\label{W1bracket}
\begin{split}
\Irr(W^{\sigma}_1)&\lisom\Irr(W_1)^{\sigma}\\
\varphi&\longmapsto\{\varphi\}
\end{split}
\end{equation}
as the composition of these two natural bijections. 

Write $N_{\pm}=(n_{\pm}-m_{\pm}(m_{\pm}+1)/2)/2$ and so $n_0=m+2N_++2N_-$. There exists a unique pair $(h_1,h_2)\in\mathbb{N}\times\mathbb{Z}$ such that 
\begin{equation}
\begin{split}
m_+&=\sup\{h_1+h_2,-h_1-h_2-1\},\\
m_-&=\sup\{h_1-h_2,h_2-h_1-1\}.
\end{split}
\end{equation}
Note that exchanging $\mu_+$ and $\mu_-$ changes $(h_1,h_2)$ into $(h_1,-h_2)$. More explicitly, if $n$ is even, then $m=m_+=m_-=0$, and so $h_1=h_2=0$, and if $n$ is odd, then $m=1$, in which case $h_1$ is always equal to $0$ while $h_2=1$ if $m_+=1$ and $h_2=-1$ if $m_-=1$. Therefore $h_1$ is redundant, and we will write $\epsilon=h_2$. 

 Fix some integers $r_+>l(\mu_+)$ and $r_->l(\mu_-)$ satisfying:
\begin{quote}
\textbf{Assumption.}
\begin{itemize}
\item $r_+$ and $r_-$ are odd if $\epsilon=0$;\\
\item $r_+$ is even if and only if $n_+$ is odd, $r_-$ is even if and only if $n_-$ is odd, if $\epsilon=\pm1$.
 \end{itemize}
\end{quote}
(See \cite[Remark 9.4.8]{Shu2}.) Note that the parity of $\epsilon$ only depends on that of $n$. Let $(\alpha_+,\beta_+)_{r_+}$ and $(\alpha_-,\beta_-)_{r_-}$ be the $2$-quotients of $\mu_+$ and $\mu_-$ respectively (See \S \ref{quotcore}). The $2$-partition $(\alpha_+,\beta_+)_{r_+}$ (resp. $(\alpha_-,\beta_-)_{r_-}$) determines an irreducible character of $\mathfrak{W}_{N_+}$ (resp. $\mathfrak{W}_{N_-}$), denoted by $\varphi_+$ (resp. $\varphi_-$). Then with the fixed $r_+$ and $r_-$, the $2$-partitions $(\mu_+,\mu_-)$ are in bijection with the data $(h_1,h_2,\varphi_+,\varphi_-)$, or rather $(\epsilon,\varphi_+,\varphi_-)$, via the quotient-core decomposition. 

\subsubsection{}\label{sssec-decIrr}
To simplify, we will write $\mathfrak{W}_+=\mathfrak{W}_{N_+}$ and $\mathfrak{W}_-=\mathfrak{W}_{N_-}$. Recall that $W_0$ is the Weyl group of $M_0$. Since $W_0^{\sigma}\cong\mathfrak{W}_{N_++N_-}$, we will regard $\mathfrak{W}_+\times\mathfrak{W}_-$ as a subgroup of $W_0^{\sigma}$ in a natural way. Write $T=T_0\times T_1$, with $T_0\subset M_0$ and $T_1\subset M_1$. To each $\mathbf{w}=(w_+,w_-)\in\mathfrak{W}_+\times\mathfrak{W}_-$ is associated an $F$-stable and $\sigma$-stable maximal torus $T_{\mathbf{w}}\subset M_0$, in such a way that $T_0$ is associated to $1\in\mathfrak{W}_+\times\mathfrak{W}_-$. There is an isomorphism $T_{\mathbf{w}}\cong T_{w_+}\times T_{w_-}\times\GL_{m}$, where $T_{w_{\pm}}$ is isomorphic to $(k^{\ast})^{2N_{\pm}}$ equipped with the Frobenius twisted by $w_{\pm}$. To each $w_1\in W_1^{\sigma}$ is associated an $F$-stable and $\sigma$-stable maximal torus $T_{w_1}$ of $M_1$, in such a way that $T_1$ is associated to $1\in W_1^{\sigma}$. Then $T_{w_1,\mathbf{w}}:=T_{w_1}\times T_{\mathbf{w}}$ is an $F$-stable and $\sigma$-stable maximal torus of $M$. Given $\theta\in\Irr_{\reg}^{\sigma}(M_1^F)$ and $\epsilon$, we can define a $\sigma$-stable linear character $\theta_{w_1,\mathbf{w}}$ of $T_{w_1,\mathbf{w}}^F$ for any $w_1\in W_1^{\sigma}$ and $\mathbf{w}\in \mathfrak{W}_+\times\mathfrak{W}_-$ in the following manner. The component of $\theta_{w_1,\mathbf{w}}$ on $T_{w_1}^F$ is simply the restriction of $\theta$ to $T_{w_1}^F$. The component of $\theta_{w_1,\mathbf{w}}$ on $T_{w_+}^F$ is the trivial character $\mathbf{1}$. The component of $\theta_{w_1,\mathbf{w}}$ on $T_{w_-}^F$ is the order 2 character $\eta$ of $\mathbb{F}_q^{\ast}$ composed with the product of norm maps. If $\epsilon=1$, then we require that the component on $\GL_m(q)$ is $\mathbf{1}$, and if $\epsilon=-1$, then we require that the same component is $\eta$. Now $\theta_{w_1,\mathbf{w}}$ extends to $T_{w_1,\mathbf{w}}^F\lb\sigma\rb$, and we denote by $\tilde{\theta}_{w_1,\mathbf{w}}$ the extension which gives $1$ at $\sigma$.\begin{Rem}\label{ext-theta}
We can indeed require that $\tilde{\theta}_{w_1,\mathbf{w}}$ takes value $1$ at $\sigma$. To define an extension $\tilde{\theta}_{w_1,\mathbf{w}}$ is to define the action of $\sigma$ on the given 1-dimensional representation $\rho$ corresponding to $\theta_{w_1,\mathbf{w}}$. Denote by $\tilde{\rho}(\sigma)$ this action of $\sigma$. In order for it to be well-defined, it must satisfy $\tilde{\rho}(\sigma)^2=\rho(\sigma^2)$. It is easy to check that this condition also suffices. Now if $n$ is odd or if $\bar{G}=\leftidx{^s\!}{\bar{G}}$, then $\sigma^2=1$, and so we can obviously put $\tilde{\rho}(\sigma)=1$. If $\bar{G}=\leftidx{^o\!}{\bar{G}}$ and $\sigma^2=-1$, we check that $\rho(-1)=1$, and so the same definition of $\tilde{\rho}(\sigma)$ works in this case as well.
\end{Rem}
\begin{Thm}(\cite[Theorem 11.1.1]{Shu2})\label{thm-decIrr}
Let $\chi$ be the $\sigma$-stable irreducible character $R^G_M(\chi_1\boxtimes\chi_0)$, and let $\tilde{\chi}$ be an extension of $\chi$ to $\bar{G}^F$. Suppose that $\chi_1$ is defined by $\theta\in\Irr^{\sigma}_{\reg}(M_1^F)$ and $\{\varphi\}\in\Irr(W_1)^{\sigma}$, and that $\chi_0$ is defined by a 2-partition which determines the data $N_+$, $N_-$ and $(\epsilon,\varphi_+,\varphi_-)$. Then for some choice of the extension $\tilde{\varphi}$, the following equality holds up to a sign,
\begin{equation}\label{thm-decIrr-eq}
\tilde{\chi}|_{G^F\sigma}=|\mathfrak{W}_+\times\mathfrak{W}_-\times W_1^{\sigma}|^{-1}\!\!\!\!\!\!\!\!\sum_{\substack{(w_+,w_-,w_1)\\\in\mathfrak{W}_+\times\mathfrak{W}_-\times W_1^{\sigma}}}\!\!\!\!\!\!\!\!\varphi_+(w_+)\varphi_-(w_-)\tilde{\varphi}(w_1F)R^{G\sigma}_{T_{w_1,\mathbf{w}}\sigma}\tilde{\theta}_{w_1,\mathbf{w}}.
\end{equation}
\end{Thm}
\begin{Rem}
Since the action of $F$ on $W_1$ is trivial, the extension $\tilde{\varphi}$ can in fact be chosen to be the trivial one, i.e. $\tilde{\varphi}(w_1F)=\varphi(w_1)$.
\end{Rem}

\subsection{Types}\hfill

We introduce some combinatorial data called \textit{types} that are used to describe $\sigma$-stable irreducible characters of $G(q)$.
\subsubsection{}\label{remaintype}
Types are data of the form $$\bs\omega=\bs\omega_+\bs\omega_-(\omega_i)_{1\le i\le l},$$ where $\bs\omega_+$ and $\bs\omega_-$ are 2-partitions, and $(\omega_i)$ is an unordered sequence of nontrivial partitions, the length of which could be $0$. We will often write $\bs\omega_{\ast}=(\omega_i)_{1\le i\le l}$ and $\bs\omega=\bs\omega_+\bs\omega_-\bs\omega_{\ast}$ for brevity. Such data can equally be written as $\bs\omega_+\bs\omega_-(m_{\lambda})_{\lambda\in\mathcal{P}}$, where $m_{\lambda}$ is the multiplicity of $\lambda$ in the sequence $\bs\omega_{\ast}$. Two types $\bs\omega_+\bs\omega_-(m_{\lambda})$ and $\bs\omega'_+\bs\omega'_-(m'_{\lambda})$ are regarded as the same if and only if $\bs\omega_+=\bs\omega_+'$, $\bs\omega_-=\bs\omega_-'$ and $m_{\lambda}=m'_{\lambda}$ for all $\lambda\in\mathcal{P}$. The size of a type is $|\bs\omega|:=|\bs\omega_+|+|\bs\omega_-|+\sum_i|\bs\omega_i|.$ The set of types of size $a$ will be denoted by $\mathfrak{T}(a)$ and we will write $\mathfrak{T}=\sqcup_{a\in\mathbb{Z}_{\ge0}}\mathfrak{T}(a)$. For any $\bs\omega\in\mathfrak{T}$, we will write $\bs{\Lambda}(\bs\omega)=\bs{\Lambda}(\bs\omega_+)\bs{\Lambda}(\bs\omega_-)\bs\omega_{\ast}$ (Recall that $\bs\Lambda(-)$ means the symbol corresponding to a given 2-partition). Denote by $\tilde{\mathfrak{T}}$ the set of \textit{ordered types}, i.e. the data $\bs\omega_+\bs\omega_-(\omega_i)$ with $(\omega_i)$ being an ordered sequence. There is an obvious map from $\tilde{\mathfrak{T}}$ to $\mathfrak{T}$, therefore anything that can be defined for elements of $\mathfrak{T}$ is naturally defined for elements of $\tilde{\mathfrak{T}}$. Given $\bs\alpha=\bs\alpha_+\bs\alpha_-(\alpha_i)_{1\le i\le l_1}$, $\bs\beta=\bs\beta_+\bs\beta_-(\beta_i)_{1\le i\le l_2}\in\tilde{\mathfrak{T}}$, we write $\bs\alpha\thickapprox\bs\beta$, if $l_1=l_2=l$ and for each $1\le i\le l$, we have $|\alpha_i|=|\beta_i|$, and moreover $|\bs\alpha_+|=|\bs\beta_+|$ and $|\bs\alpha_-|=|\bs\beta_-|$. A type $\bs\omega$ (ordered or unordered) can be augmented by $\epsilon\in\{-1,0,1\}$, the resulting data written as $\epsilon\bs\omega$. Define $|\epsilon\bs\omega|:=|\bs\omega|$.

We will denote by $\mathfrak{T}^{\circ}$ the set of unordered sequences $\lambda_+\lambda_-(\lambda_i)_{1\le i\le l}$, defined in a way similar to types except that $\lambda_+$ and $\lambda_-$ are partitions. We may also write such a sequence as $\lambda_+\lambda_-\bs\lambda_{\ast}$, with $\bs\lambda_{\ast}=(\lambda_i)$. The ordered version $\tilde{\mathfrak{T}}^{\circ}$ of these data can be defined in an obvious way. Define $|\bs\lambda|:=|\lambda_+|+|\lambda_-|+\sum_i|\lambda_i|$.

The subset $\{\bs\omega\in\mathfrak{T}\mid\bs\omega_+=\bs\omega_-=\varnothing\}$ will be denoted by $\mathfrak{T}_{\ast}$. We may also regard it as a subset of $\mathfrak{T}^{\circ}$, so that $\mathfrak{T}_{\ast}=\mathfrak{T}\cap\mathfrak{T}^{\circ}$. Thus for any $\bs\alpha\in\mathfrak{T}_{\ast}$, the size $|\bs\alpha|$ is automatically defined, and we will denote by $l(\bs\alpha)$ its length. For any $\bs\alpha=(m_{\lambda})_{\lambda\in\mathcal{P}}\in\mathfrak{T}_{\ast}$, define
\begin{equation}\label{Nomega}
N(\bs\alpha)=\prod_{\lambda}m_{\lambda}!,
\end{equation}
and
\begin{equation}\label{Komega}
K(\bs\alpha)=(-1)^{l(\bs\alpha)}l(\bs\alpha)!.
\end{equation}

If $\bs{\alpha}$ is a 2-partition, then we define $\{\bs{\alpha}\}_1$ to be the partition with 2-core $(1)$ and 2-quotient $\bs\alpha$, and define $\{\bs{\alpha}\}_0$ to be the partition with trivial 2-core and 2-quotient $\bs\alpha$.  For any $\bs\alpha=(m_{\lambda})_{\lambda\in\mathcal{P}}\in\mathfrak{T}_{\ast}$, define $\bs\alpha^2:=(2m_{\lambda})_{\lambda\in\mathcal{P}}\in\mathfrak{T}_{\ast}$. Define a natural map 
\begin{equation}\label{type-brace}
\begin{split}
\{~\}:\{-1,0,1\}\times\mathfrak{T}&\longrightarrow\mathfrak{T}^{\circ}\\
\epsilon\bs\omega&\longmapsto \lambda_+\lambda_-\bs\lambda_{\ast}=\{\epsilon\bs\omega\}
\end{split}
\end{equation}
as follows. The partition $\lambda_{\pm}$ has $\omega_{\pm}$ as the 2-quotient, and has trivial 2-core except when $\epsilon=+1$ and $\lambda_{+}=\{\bs\omega_+\}_1$, or when $\epsilon=-1$ and $\lambda_{-}=\{\bs\omega_-\}_1$. We require that $\bs\lambda_{\ast}=\bs\omega_{\ast}^2$.

We define a natural map 
\begin{equation}\label{type-bracket}
[~]:\mathfrak{T}\longrightarrow \mathcal{P}^2
\end{equation}
as follows. Let $\bs\omega=\bs\omega_+\bs\omega_-\bs\omega_{\ast}\in\mathfrak{T}$. Regarding $\mathcal{P}^2$ as signed partitions, it sends each part of $\omega_i$ to a positively signed part with the same size, and keeps the sign and size of each part of $\bs\omega_+$ and $\bs\omega_-$. The union of these signed parts is the image of $[~]$ and is denoted by $[\bs\omega]$.

Denote by $\tilde{\mathfrak{T}}_s$ the subset of $\tilde{\mathfrak{T}}$ consisting of elements of the form $$\bs\omega=(\varnothing,(1)^{m_+})(\varnothing,(1^{m_-}))((1^{m_i}))_i$$ where $m_+=|\bs\omega_+|$, $m_-=|\bs\omega_-|$ and $m_i=|\omega_i|$. Equivalently, the elements of $\tilde{\mathfrak{T}}_s$ can be written as a sequence of integers: $\bs\omega=m_+m_-(m_i)_i$ and $\bs\omega_{\ast}=(m_i)_i$. We define its dual type by $$\bs\omega^{\ast}=((m_+),\varnothing)((m_-),\varnothing)((m_i))_i.$$ Define the surjective map 
\begin{equation}
\tilde{S}:\tilde{\mathfrak{T}}\longrightarrow\tilde{\mathfrak{T}}_s
\end{equation}
by $\tilde{S}(\bs\omega):=|\bs\omega_+||\bs\omega_-|(|\omega_i|)_i$. If $\bs\omega\in\tilde{\mathfrak{T}}_s$, then we will denote by $\tilde{\mathfrak{T}}(\bs\omega)$ the inverse image of $\bs\omega$ under $\tilde{S}$. Two ordered types $\bs\alpha$ and $\bs\beta$ satisfy $\bs\alpha\thickapprox\bs\beta$ if and only if they lie in the same $\tilde{\mathfrak{T}}(\bs\omega)$ for some $\bs\omega\in\tilde{\mathfrak{T}}_s$. The unordered version of $S$ and $\mathfrak{T}_s$ can be similarly defined.

\subsubsection{}\label{relevance-types}
The relevance of types is as follows. The group $\mathfrak{W}_+\times\mathfrak{W}_-\times W_1^{\sigma}$ in Theorem \ref{thm-decIrr} is isomorphic to $\mathfrak{W}_{N_+}\times\mathfrak{W}_{N_-}\times\prod_{1\le i\le l}\mathfrak{S}_{n_i}$ for some integers $n_i$ and $l$. The irreducible characters and the conjugacy classes of this group are both parametrised by $\tilde{\mathfrak{T}}(\bs\omega)$, where $\bs\omega=N_+N_-(n_i)_i$. Therefore, the characters $(\varphi_+,\varphi_-,\varphi)$ and the conjugacy class of $(w_1,\mathbf{w})$ as in Theorem \ref{thm-decIrr} can be represented by an ordered type. In view of Lemma \ref{Shu2Lem6.2.1}, if $t\in(T^{\sigma})^{\circ}$ (and $N_-=0$ if $n$ is even), then the $C_G(t\sigma)^{\circ F}$-conjugacy classes of the $F$-stable maximal tori of $C_G(t\sigma)^{\circ}$ are also paremetrised by $\tilde{\mathfrak{T}}(\bs\omega)$. A semi-simple conjugacy class $C\subset G\sigma$ is determined by an $N$-tuple of elements of $\hat{\kk}$, thus we define the type of $C$ to be the element of $\mathfrak{T}_s$ encoding the multiplicities of "eigenvalues", with $m_+$ being the multiplicity of $1$ and $m_-$ that of $\mathfrak{i}$. 

Let $\chi\in\Irr(\GL_n(q))^{\sigma}$, then there exists some Levi subgroup $M$ such that $\chi=\epsilon_G\epsilon_M R^G_M(\chi_1\boxtimes\chi_0)$ as in Theorem \ref{sig-st-chi}. Denote by $\Irr_{st}^{\sigma}\subset \Irr(\GL_n(q))^{\sigma}$ the subset of characters such that $M$ can be chosen to be a standard Levi subgroup. Define a map
\begin{equation}
\pi_{\circ}:\Irr_{st}^{\sigma}\longrightarrow\mathfrak{T}^{\circ}
\end{equation}
as follows. We first fix a choice of $M$ and a realisation of $\chi$ as an induction: $\epsilon_G\epsilon_M R^G_M(\chi_1\boxtimes\chi_0)$. Let $(\lambda_+,\lambda_-)$ be the 2-partition determined by the quadratic-unipotent part $\chi_0$. The character $\chi_1$ of $M_1^F$ is determined by some $\varphi_1\in\Irr(W_1)$, which can be represented by a sequence of partitions $\bs\lambda_{\ast}\in\mathfrak{T}_{\ast}$, since $W_1$ is a product of symmetric groups. Now $\lambda_+\lambda_-\bs\lambda_{\ast}$ lies in $\mathfrak{T}^{\circ}$. This is the desired $\pi_{\circ}(\chi)$. Note that $\varphi_1$ in fact lies in $\Irr(W_1)^{\sigma}$, and so the multiplicity of any partition $\lambda$ in $\bs\lambda_{\ast}$ is an even number. Define a map
\begin{equation}
\pi_{\sigma}:\Irr_{st}^{\sigma}\longrightarrow\{-1,0,+1\}\times\mathfrak{T}
\end{equation}
Let $(\epsilon,\varphi_+,\varphi_-)$ be determined by $\chi_0$ as in \S \ref{decIrr}, and let $\varphi\in\Irr(W_1^{\sigma})$ be the inverse image of $\varphi_1$ under the bijection (\ref{W1bracket}). Now $(\varphi_+,\varphi_-,\varphi)$ is an irreducible character of $\mathfrak{W}_+\times\mathfrak{W}_-\times W_1^{\sigma}$, and so gives an element $\bs\omega\in\mathfrak{T}$. We then define $\pi_{\sigma}(\chi)=\epsilon\bs\omega$. In \S \ref{decIrr}, we have described a bijection between the data $(\epsilon,\varphi_+,\varphi_-)$ and $\mathcal{P}^2(n_0)$, and a bijection $\{~\}:\Irr(W^{\sigma}_1)\isom\Irr(W_1)^{\sigma}$. The map (\ref{type-brace}) should be though of as sending $(\epsilon,\varphi_+,\varphi_-,\varphi)$ to the character of the Weyl group of $M_{00}\times M_1$ that underlies $R^G_M(\chi_1\boxtimes\chi_0)$. Now we have a commutative diagram 
\begin{equation}\label{cd-brace}
\begin{tikzcd}[row sep=1em, column sep=1em]
&\{-1,0,+1\}\times\mathfrak{T} \arrow[dd, "\{~\}"] \\
\Irr_{st}^{\sigma} \arrow[ur, "\pi_{\sigma}"] \arrow[dr, swap, "\pi_{\circ}"]\\
& \mathfrak{T}^{\circ}
\end{tikzcd}
\end{equation}
Define 
\begin{equation}
\Irr^{\sigma}_{\bs\lambda}:=\pi_{\circ}^{-1}(\bs\lambda),\quad\Irr^{\sigma}_{\epsilon\bs\omega}:=\pi_{\sigma}^{-1}(\epsilon\bs\omega).
\end{equation}
Its elements are called $\sigma$-stable characters of type $\bs\lambda$ and type $\epsilon\bs\omega$ respectively. Note that $\Irr^{\sigma}_{\bs\lambda}$ is empty unless each partition in $\bs\lambda_{\ast}$ has even multiplicity.

By definition, all $\sigma$-stable characters of the same type can be obtained from a common $M$ as in Theorem \ref{sig-st-chi}, but with different $\theta\in \Irr_{\reg}^{\sigma}(M_1^F)$. Observe that the quadratic-unipotent part of a character is completely determined by its type. Therefore we have a surjective map
\begin{equation}\label{reg->eomega}
\Irr_{\reg}^{\sigma}(M_1^F)\longrightarrow\Irr^{\sigma}_{\epsilon\bs\omega}.
\end{equation}
The cardinality of the fibre of this map is equal to $2^{l(\bs\omega_{\ast})}N(\bs\omega_{\ast})$. If we write $\theta=(\theta_i)_{1\le i\le l}$, with $l=l(\bs\omega_{\ast})$, then the factor 2 comes from the permutation $\theta_i\leftrightarrow\theta_i^{-1}$, and the factor $m_{\lambda}!$ of $N(\bs\omega_{\ast})$ comes from the permutation of those $\theta_i$ corresponding to the same $\lambda$. 

We also have an explanation of the map $[~]$ (\ref{type-bracket}). Recall (\S \ref{sssec-decIrr}) that the group $\mathfrak{W}_+\times\mathfrak{W}_-\times W_1^{\sigma}$ is naturally a subgroup of $W_0^{\sigma}\times W_1^{\sigma}$, which is contained in $W^{\sigma}$. It is easy to see that $[~]$ is just the map between conjugacy classes induced by the inclusion $\mathfrak{W}_+\times\mathfrak{W}_-\times W_1^{\sigma}\rightarrow W^{\sigma}$. Alternatively, by Lemma \ref{Shu2Lem6.2.1}, the Weyl group $W'$ of the centraliser of a semi-simple element $C_G(t\sigma)^{\circ}$ is isomorphic to $\mathfrak{W}_{N_+}\times\mathfrak{W}_{N_-}\times\prod_{1\le i\le l}\mathfrak{S}_{n_i}$ for some integers $n_i$ and $l$ (assuming $N_-=0$ if $n$ is even). Corollary \ref{Nsigma-in-N-2} gives an injective map $W'\rightarrow W^{\sigma}$. Again, the map induced between conjugacy classes is just $[~]$.

Finally, the set $B_{\tau}$ that we introduced in \S \ref{A^F->B} to compute the Deligne-Lusztig characters has an alternative description in terms of types.
\begin{Lem}\label{comb-Btau}
Let $s\sigma\in T^F\sigma$ be a semi-simple element. Let $W'$ be the Weyl group of $C_G(s\sigma)^{\circ}$ defined by $(T^{\sigma})^{\circ}$. Suppose that there is an isomorphism $W'\isom\mathfrak{W}_{N_+}\times\mathfrak{W}_{N_-}\times\prod_{1\le i\le l}\mathfrak{S}_{n_i}$ so that its conjugacy classes are parametrised by $\tilde{\mathfrak{T}}(\bs\beta)$, with $\bs\beta=N_+N_-(n_i)_{1\le i\le l}\in\tilde{\mathfrak{T}}_s$ (and $N_-=0$ if $n$ is even). Let $\tau\in W^{\sigma}$ and represent its conjugacy class by a 2-partition $\bs\tau$. Then $$B_{\tau}=\{\bs\nu\in\tilde{\mathfrak{T}}(\bs\beta)\mid[\bs\nu]=\bs\tau\}.$$
\end{Lem}
\begin{proof}
This is simply unwinding the definitions.
\end{proof}

%%%%%%%%%%%%%%%%%%%%%%
\section{$\GL_n\lb\sigma\rb~$-Character Varieties}\label{Section-Char-Var}
In this section, we work over an algebraically closed field $\kk$ with $\ch\kk\ne 2$.

\subsection{Definition of $\GL_n\lb\sigma\rb~$-Character Varieties}\label{Defn-Char.Var}

\subsubsection{}
Let $p':\tilde{\Sigma}'\rightarrow \Sigma'$ be a branched double covering of compact Riemann surfaces, and let $\mathcal{R}\subset\Sigma'$ be the ramification locus. Put $\Sigma=\Sigma'\setminus\mathcal{R}$ and $\tilde{\Sigma}=\tilde{\Sigma}'\setminus p^{\prime -1}(\mathcal{R})$ and denote by $p:\tilde{\Sigma}\rightarrow \Sigma$ the restriction of $p'$. Fix a base point $x\in \Sigma$ and let $\tilde{x}\in p^{-1}(x)$ be the base point of $\tilde{\Sigma}$. We will write $\pi_1(\Sigma)=\pi_1(\Sigma,x)$ and $\pi_1(\tilde{\Sigma})=\pi_1(\tilde{\Sigma},\tilde{x} )$. The representation variety $\Rep(\Sigma)$ is defined as the space of homomorphisms $\rho$ that make the following diagram commute:
$$
\begin{tikzcd}[row sep=2.5em, column sep=0.2em]
\pi_1(\Sigma) \arrow[dr, swap, "q_1"] \arrow[rr, "\rho"] && G\lb\sigma\rb \arrow[dl, "q_2"]\\
& \Gal(\tilde{\Sigma}/\Sigma)\cong\bs\mu_2
\end{tikzcd}
$$ 
where $q_1$ is the quotient by $\pi_1(\tilde{\Sigma})$ and $q_2$ is the quotient by the identity component $G$. The character variety $\Ch(\Sigma)$ is defined as the categorical quotient of $\Rep(\Sigma)$ for the conjugation action of $G$ on $G\lb\sigma\rb$. This is the $G\lb\sigma\rb~$-character variety.

\subsubsection{}
Denote by $g$ the genus of $\Sigma'$. By Riemann-Hurwitz formula, there are an even number of branch points. Let $2k=|\mathcal{R}|$ be the cardinality and write $\mathcal{R}=\{x_j\}_{1\le j\le 2k}$. We may choose the generators $\alpha_i$, $\beta_i$, $1\le i\le g$, and $\gamma_j$, $1\le j\le 2k$, of $\pi_1(X)$ that satisfy the relation 
\begin{equation}
\prod_{i=1}^g[\alpha_i,\beta_i]\prod_{j=1}^{2k}\gamma_j=1.
\end{equation}
The generators $\alpha_i$, $\beta_i$, $1\le i\le g$, are the images of some elements of $\pi_1(\tilde{\Sigma})$, while $\gamma_j$, $1\le j\le 2k$, being small loops around the $x_j$'s, lie in $\pi_1(\Sigma)\setminus\pi_1(\tilde{\Sigma})$. Let $\mathcal{C}=(C_j)_{1\le j\le 2k}$ be a tuple of semi-simple $G$-conjugacy classes contained in the connected component $G\sigma$. We define the subvariety $\Rep_{\mathcal{C}}(\Sigma)\subset\Rep(\Sigma)$ as consisting of $\rho\in\Rep(\Sigma)$ such that $\rho(\gamma_j)\in C_j$ for all $j$. Then we define $\Ch_{\mathcal{C}}(\Sigma):=\Rep_{\mathcal{C}}(\Sigma)\ds G$. The representation variety has the following presentation
\begin{equation}\label{CharVarEq}
\Rep_{\mathcal{C}}(\Sigma)=\{(A_i,B_i)_i(X_j)_j\in G^{2g}\times \prod_{j=1}^{2k} C_j\mid\prod_{i=1}^g[A_i,B_i]\prod_{j=1}^{2k}X_j=1\}.
\end{equation}

\subsubsection{}
There is a bijection between the semi-simple conjugacy classes in $\leftidx{^o\!}{\bar{G}}\setminus G$ and those in $\leftidx{^s\!}{\bar{G}}\setminus G$, sending the class of $t\sigma$ to that of $t\sigma$, with the $\sigma$ in $\leftidx{^o\!}{\bar{G}}$ interpreted in the appropriate sense. If $\mathcal{C}$ is a tuple of semi-simple conjugacy classes in $\leftidx{^s\!}{\bar{G}}\setminus G$, then we define a tuple $\mathcal{C}^{\ast}$ of conjugacy classes in $\leftidx{^o\!}{\bar{G}}\setminus G$ by applying this bijection componentwise. 

It is not difficult to see that the $\leftidx{^o\!}{\bar{G}}$-character variety $\Ch_{\mathcal{C}}$ is isomorphic to the $\leftidx{^s\!}{\bar{G}}$-character variety $\Ch_{\mathcal{C}^{\ast}}$. Therefore, if $n$ is even, we may, and will, only work with $\leftidx{^s\!}{\bar{G}}$-character varieties. 

\subsection{Generic Conjugacy Classes}

\subsubsection{}\label{GCC}
Let $\mathcal{C}=(C_j)_{j}$ be a $2k$-tuple of semi-simple conjugacy classes contained in $G\sigma$. According to \S \ref{s.s.GLnsigma}, each class $C_j$ has a representative $t_j\sigma\in(T^{\sigma})^{\circ}\sigma$. The following definition is a special case of \cite[Definition 3.3]{Shu1}, as is explained in \cite[\S 4.4]{Shu1}.

Each $t_j$ is given by an $N$-tuple $(a_{j,1},\ldots,a_{j,N})$ with each $a_{j,\gamma}\in\kk^{\ast}$. Write $\Lambda=\{1,\ldots,N\}$. For any $j$, any subset $\mathbf{A}\subset\Lambda$ and any $|\mathbf{A}|$-tuple of signs $\smash{\mathbf{e}=(e_{\gamma})_{\gamma\in\mathbf{A}}}$, $e_{\gamma}\in\{\pm 1\}$, write $\smash{[\mathbf{A},\mathbf{e}]_j=\prod_{\gamma\in\mathbf{A}}a_{j,\gamma}^{2e_{\gamma}}}$. We say that $\mathcal{C}$ is \textit{generic} if for any $1\le N' \le N$, any $2k$-tuple $(\mathbf{A}_1,\ldots,\mathbf{A}_{2k})$ of subsets of $\Lambda$ such that $|\mathbf{A}_1|=\cdots=|\mathbf{A}_{2k}|=N'$, and any $2k$-tuple of $N'$-tuples $(\mathbf{e}^1,\ldots,\mathbf{e}^{2k})$ of signs, we have
\begin{equation}\label{GenConGLn1}
[\mathbf{A}_1,\mathbf{e}^1]_1\cdots[\mathbf{A}_{2k},\mathbf{e}^{2k}]_{2k}\ne1.
\end{equation}
We say that $\mathcal{C}$ is \textit{strongly generic} if for any $N'$, $(\mathbf{A}_1,\ldots,\mathbf{A}_{2k})$, and $(\mathbf{e}^1,\ldots,\mathbf{e}^{2k})$ as above, we have
\begin{equation}\label{GenConGLn2}
[\mathbf{A}_1,\mathbf{e}^1]_1\cdots[\mathbf{A}_{2k},\mathbf{e}^{2k}]_{2k}\ne\pm1.
\end{equation}

\subsubsection{}\label{explicit-genconj-Fq}
Generic conjugacy classes in the finite group $G(q)\lb\sigma\rb$ can be defined in the same way, but only for some particular conjugacy classes.

Let $\mathcal{C}=(C_j)_{j}$ be a $2k$-tuple of semi-simple $G(q)$-conjugacy classes contained in $G(q)\sigma$. Assume that for each $1\le j\le 2k$, the $G$-conjugacy class containing $C_j$ has a representative $t_j\sigma\in(T^{\sigma})^{\circ F}\sigma$, that is, $t_j$ is determined by an $N$-tuple $(a_{j,1},\ldots,a_{j,N})$ with each $a_{j,\gamma}\in\mathbb{F}_q^{\ast}$. Note that we do not require $C_j$ itself to have a representative of this form. As is explained in \S \ref{GLnSigma-F-Class}, this is a very restrictive condition on the conjugacy classes. Generic condition can be defined only for these conjugacy classes.

With the same notations as in the previous paragraph, we say that $\mathcal{C}$ is generic (resp. strongly generic) if for any $1\le N' \le N$, $(\mathbf{A}_1,\ldots,\mathbf{A}_{2k})$, and $(\mathbf{e}^1,\ldots,\mathbf{e}^{2k})$, we have
\begin{equation}
[\mathbf{A}_1,\mathbf{e}^1]_1\cdots[\mathbf{A}_{2k},\mathbf{e}^{2k}]_{2k}\ne1\text{ (resp. $\pm1$)}.
\end{equation}
\begin{Lem}\label{disct-e.v.-genconj}
Let $(C_j)_{1\le j\le 2k}$ be a generic tuple of semi-simple conjugacy classes contained in $G\sigma$, with each $C_j$ given by an $N$-tuple $(a_{j,1},\ldots,a_{j,N})$ as in \S \ref{s.s.GLnsigma}. Let $N'$, $N''\in\{1,\ldots,N\}$. Let $(\mathbf{A}_1,\ldots,\mathbf{A}_{2k})$ and $(\mathbf{A}'_1,\ldots,\mathbf{A}'_{2k})$ be two $2k$-tuples of subsets of $\Lambda$ such that $|\mathbf{A}_1|=\cdots=|\mathbf{A}_{2k}|=N'$, $|\mathbf{A}'_1|=\cdots=|\mathbf{A}'_{2k}|=N''$ and for any $1\le j\le 2k$, $\mathbf{A}_j\cap \mathbf{A}'_j=\emptyset$. Let $(\mathbf{e}^1,\ldots,\mathbf{e}^{2k})$ (resp. $(\mathbf{e}^{\prime 1},\ldots,\mathbf{e}^{\prime2k})$) be a $2k$-tuple of $N'$-tuples (resp. $N''$-tuples) of signs. Then,$$[\mathbf{A}_1,\mathbf{e}^1]_1\cdots[\mathbf{A}_{2k},\mathbf{e}^{2k}]_{2k}\ne\left([\mathbf{A}'_1,\mathbf{e}^{\prime1}]_1\cdots[\mathbf{A}'_{2k},\mathbf{e}^{\prime2k}]_{2k}\right)^{\pm 1}.$$
\end{Lem}
\begin{proof}
Suppose $$[\mathbf{A}_1,\mathbf{e}^1]_1\cdots[\mathbf{A}_{2k},\mathbf{e}^{2k}]_{2k}=\left([\mathbf{A}'_1,\mathbf{e}^{\prime1}]_1\cdots[\mathbf{A}'_{2k},\mathbf{e}^{\prime2k}]_{2k}\right)^{\epsilon},$$with $\epsilon=\pm 1$. For each $j$, put $\tilde{\mathbf{A}}_j=\mathbf{A}_j\sqcup\mathbf{A}'_j$. For each $j$, define a $(N'+N'')$-tuple of signs $\tilde{\mathbf{e}}^j$ by $\tilde{e}^j_{\gamma}=e^j_{\gamma}$ if $\gamma\in\mathbf{A}_j$, and $\tilde{e}^j_{\gamma}=-\epsilon e^{\prime j}_{\gamma}$ if $\gamma\in\mathbf{A}'_j$. Then,$$[\tilde{\mathbf{A}}_1,\tilde{\mathbf{e}}^1]_1\cdots[\tilde{\mathbf{A}}_{2k},\tilde{\mathbf{e}}^{2k}]_{2k}=1,$$which is a contradiction.
\end{proof}

\subsubsection{}\label{CCCL}
Here we introduce another constraint on the tuple $\mathcal{C}=(C_j)_{1\le j\le 2k}$ of semi-simple conjugacy classes.  

\begin{description}
\item[(CCL)\label{CCL}] We say that a semi-simple conjugacy class $C\subset G\sigma$ satisfies \ref{CCL} if it has no "eigenvalue" equal to $\pm\mathfrak{i}$. We say that $\mathcal{C}$ satisfies \ref{CCL} if $C_j$ satisfies \ref{CCL} for all $j$. 
\end{description}

In view of Lemma \ref{Shu2Lem6.2.1}, that a semi-simple conjugacy class in $G\sigma$ satisfies \ref{CCL} essentially means that it has a representative $t\sigma\in(T^{\sigma})^{\circ}\sigma$ such that $C_{G}(t\sigma)^{\circ}$ is a Levi subgroup of $C_{G}(\sigma)^{\circ}$. The exceptional case is when the multiplicity of $\mathfrak{i}$ among the "eigenvalues" is equal to one, and $n$ is even, since $\SO_2\cong\mathbb{G}_m$.

Our main results in \S \ref{Section-E-poly}  will assume the \ref{CCL} property.

\subsection{The $R$-Model}\hfill

It is a general fact that $\Ch_{\mathcal{C}}$ is defined over a finite-type subring of $\mathbb{C}$. But we want to be specific about the base ring so that when passing to finite fields the conjugacy classes remain generic.
\subsubsection{}\label{R}
The conjugacy classes $(C_j)_j$ are represented by some $N$-tuples of complex numbers (\textit{cf.} \S \ref{s.s.GLnsigma}) $$\mathbf{E}^j=(1,\ldots,1,\mathfrak{i},\ldots,\mathfrak{i},a^j_1,\ldots,a^j_1,\ldots,a^j_{l_j},\ldots,a^j_{l_j}),$$satisfying
\begin{itemize}
\item[(i)] $a^j_r\notin\{a^j_s,(a^j_s)^{-1},-a^j_s,-(a^j_s)^{-1}\}$ for any $j$ and any $r\ne s$;
\item[(ii)] $(a^j_r)^4\ne1$ for any $j$ and $r$;
\item[(iii)] $[\mathbf{A}_1,\mathbf{e}^1]_1\cdots[\mathbf{A}^{2k},\mathbf{e}^{2k}]_{2k}\ne1$, for any integers $N'$ with $1\le N'\le N$, any $2k$-tuple of subsets $(\mathbf{A}_1,\ldots,\mathbf{A}_{2k})$ of $\{1,\ldots,N\}$, such that $|\mathbf{A}_1|=\cdots=|\mathbf{A}_{2k}|=N'$, and any $2k$-tuple of $N'$-tuples of signs $(\mathbf{e}^1,\ldots,\mathbf{e}^{2k})$ (\textit{cf.} \S \ref{GCC}).
\end{itemize} 
Denote by $R_0$ the subring of $\mathbb{C}$ generated by $\{(a^j_r)^{\pm 1}\mid\text{all }r,j\}$ and $\mathfrak{i}$ if any $\mathbf{E}^j$ contains $\mathfrak{i}$.  Let $S\subset R_0$ be the multiplicative subset generated by
\begin{itemize}
\item[(i)] $(a^j_{r})^2-(a^j_{s})^2$ and $(a^j_{r})^2-(a^j_{s})^{-2}$ for any $j$ and $r\ne s$;
\item[(ii)] $(a^j_{r})^4-1$ for any $j$ and $r$;
\item[(iii)] $[\mathbf{A}_1,\mathbf{e}^1]_1\cdots[\mathbf{A}^{2k},\mathbf{e}^{2k}]_{2k}-1$, for any $N'$, $(\mathbf{A}_1,\ldots,\mathbf{A}_{2k})$ and $(\mathbf{e}^1,\ldots,\mathbf{e}^{2k})$ as above.
\end{itemize}
Define the ring of generic eigenvalues as $R:=S^{-1}R_0$. We will see below that the character variety is defined over $R$. 

Let $\bs\mu^j=m^j_+m^j_-(m^j_r)\in\mathfrak{T}_s$ be the type of $C_j$, so that $l_j=l(\bs\mu_{\ast}^j)$. Write $n^j_-=2m^j_-$ and write $n^j_+=2m^j_+$ or $n^j_+=2m^j_++1$ according to the parity of $n$.

\subsubsection{}\label{A_0}
Let $\mathcal{A}'_0$ be the polynomial ring over $R$ with $n^2(2g+2k)$ indeterminates. These indeterminates should be thought of as the entries of some $n\times n$ matrices $A_1,B_1,\ldots A_g,B_g,X_1,\ldots X_{2k}$. Let $\mathcal{A}_0$ be the localisation of $\mathcal{A}'_0$ at the determinants $\det A_i$, $\det B_i$, $\det X_j$, $1\le i\le g$, $1\le j\le 2k$. Let $I_0\subset \mathcal{A}_0$ be the ideal generated by 
\begin{itemize}
\item[(i)] The entries of $[A_1,B_1]\cdots[A_g,B_g]X_1\sigma \cdots X_{2k}\sigma-\Id$ (Note that $X_1\sigma X_2\sigma=X_1\sigma(X_2)$ and $\sigma$ is defined over $\mathbb{Z}$);
\item[(ii)] For all $1\le j\le 2k$, the entries of
\begin{equation}
((X_j\sigma(X_j))^2-\Id)\prod_{r=1}^{l_j}(X_j\sigma(X_j)-(a^j_r)^2\Id)(X_j\sigma(X_j)-(a_r^j)^{-2}\Id);
\end{equation}
\item[(iii)] For all $1\le j\le 2k$, the entries of the coefficients of the following polynomial in an auxiliary variable $t$:
\begin{equation}
\det(t\Id-X_j\sigma(X_j))-(t-1)^{n^j_+}(t+1)^{n^j_-}\prod_r^{l_j}(t-(a^j_r)^2)^{m^j_r}\prod_r^{l_j}(t-(a^j_r)^{-2})^{m^j_r}.
\end{equation}
\end{itemize} 

Define $\mathcal{A}:=\mathcal{A}_0/\sqrt{I_0}$. By Lemma \ref{C-C^2}, the relations (ii) and (iii) guarantee that the base change to $\mathbb{C}$ recovers the complex representation variety with the correct conjugacy classes. Then $\sRep_{\mathcal{C}}:=\Spec\mathcal{A}$ is the $R$-model of $\Rep_{\mathcal{C}}$. Let $G$ act on $A_i$ and $B_i$, $1\le i\le g$, by conjugation, and act on $X_j$, $1\le j\le 2k$, by the $\sigma$-twisted conjugation. Then $\sCh_{\mathcal{C}}:=\Spec\mathcal{A}^{G(R)}$ is the $R$-model of $\Ch_{\mathcal{C}}$, since taking invariants commutes with flat base change (\cite[\S I.2 Lemma 2]{Ses77}).

\subsubsection{}
Let $\phi:R\rightarrow \mathbb{F}_q$ be any ring homomorphism and let $\bar{\phi}:R\rightarrow\bar{\mathbb{F}}_q$ be its composition with $\mathbb{F}_q\hookrightarrow\bar{\mathbb{F}}_q$. Regard $G=\GL_n$ as an algebraic group over $\bar{\mathbb{F}}_q$. For each $1\le j\le 2k$, denote by $C_j^{\phi}$ the subvariety of $G$ over $\mathbb{F}_q$ defined by \S \ref{A_0} (ii), (iii). Similarly denote by $C_j^{\bar{\phi}}$ the subvariety of $G$ over $\bar{\mathbb{F}}_q$. We may regard $C_j^{\bar{\phi}}$ as a semi-simple conjugacy class contained in $G\sigma$. We have $C_j^{\phi}\otimes_{\mathbb{F}_q}\bar{\mathbb{F}}_q \cong C_j^{\bar{\phi}}$. Each class $C_j^{\bar{\phi}}$ has the same type as the original complex conjugacy class, and has "eigenvalues" (\textit{cf.} \S \ref{s.s.GLnsigma}) $\phi(a^j_i)$ (resp. $1$, resp. $\phi(\mathfrak{i})$) of multiplicity $\mu^j_i$ (resp $\mu^j_+$, resp. $\mu^j_-$) in the corresponding conjugacy class. In particular, $C_j^{\bar{\phi}}$ has a representative in $(T^{\sigma})^{\circ F}\sigma$. Denote by $\Rep^{\phi}_{\mathcal{C}}$ and $\Rep^{\bar{\phi}}_{\mathcal{C}}$ the varieties over $\mathbb{F}_q$ and $\bar{\mathbb{F}}_q$ respectively obtained by base change from $\sRep_{\mathcal{C}}$. And similarly for $\sCh_{\mathcal{C}}$. The variety $\Rep^{\bar{\phi}}_{\mathcal{C}}$ is defined by the same equation as (\ref{CharVarEq}), but over $\bar{\mathbb{F}}_q$. Then $\Rep^{\phi}_{\mathcal{C}}(\mathbb{F}_q)$ can be identified with $\Rep^{\bar{\phi}}_{\mathcal{C}}(\bar{\mathbb{F}}_q)^F$:
\begin{equation}\label{RepCFq}
\Rep^{\phi}_{\mathcal{C}}(\mathbb{F}_q)=\{(A_i,B_i)(X_j)\in G(q)^{2g}\times \prod^{2k}_{j=1} C_j^{\phi}(\mathbb{F}_q)\mid\prod_{i=1}^g[A_i,B_i]\prod_{j=1}^{2k}X_j=1\},
\end{equation}
where $C_j^{\phi}(\mathbb{F}_q)=C_j^{\bar\phi}(\bar{\mathbb{F}}_q)^F$ is considered as contained in $\GL_n\sigma$.

\begin{Nota}\label{A(C)}
For each $j$, choose $s_j\in C_j$, and define $$\mathbf{A}(\mathcal{C}):=\prod_jC_{G}(s_j)/C_{G}(s_j)^{\circ}.$$ Different choices of the $s_j$'s give isomorphic $\mathbf{A}(\mathcal{C})$. Its direct factors are either the trivial group or the 2-element group $\bs\mu_2$.
\end{Nota}

\begin{Nota}\label{Cj+-}
Recall \S \ref{GLnSigma-F-Class} that $\smash{C_j^{\bar\phi}(\bar{\mathbb{F}}_q)^F}$ is in general not a single $G(q)$-conjugacy classes. Denote by $C_{j,+}$ the $G(q)$-conjugacy class contained in $\smash{C_j^{\bar\phi}(\bar{\mathbb{F}}_q)}$ which has a representative in $(T^{\sigma})^{\circ F}$ and by $C_{j,-}$ the other class if $\smash{C_j^{\bar\phi}(\bar{\mathbb{F}}_q)^F}$ is not a single $G(q)$-conjugacy class. For a given $2k$-tuple $\mathcal{C}^{\bar\phi}=(C_j^{\bar\phi})$ of semi-simple conjugacy classes in $G\sigma$ and any $\mathbf{e}=(e_j)\in\mathbf{A}(\mathcal{C})$, we will denote by $\mathcal{C}_{\mathbf{e}}=(C_{j,e_j})$ the $2k$-tuple of $G(q)$-conjugacy classes contained in $\mathcal{C}^{\bar\phi}$. If $\mathscr{C}$ is a tuple of conjugacy classes of the form $\mathcal{C}_{\mathbf{e}}$, we define $\sgn\mathscr{C}=\prod_je_j$.
\end{Nota}

In the following Proposition, we fix the generic tuple of conjugacy classes and omit the subscript $\mathcal{C}$. By the definition of the base ring $R$, the tuple $\mathcal{C}^{\bar{\phi}}$ is also generic.
\begin{Prop}\label{M=U/G}
We have the following formula.
\begin{equation}
|\Ch^{\phi}(\mathbb{F}_q)|=\frac{1}{|G(q)|}|\Rep^{\phi}(\mathbb{F}_q)|.
\end{equation}
\end{Prop}
\begin{proof}
By \cite[\S II.4 Theorem 3]{Ses77}, there is a natural bijection of sets
\begin{equation}
\Ch^{\bar\phi}(\bar{\mathbb{F}}_q)\lisomLR(\Rep^{\bar{\phi}}\ds G)(\bar{\mathbb{F}}_q).
\end{equation}
Since $\mathcal{C}^{\bar\phi}$ is generic, each element of $\Rep^{\bar\phi}(\bar{\mathbb{F}}_q)$ is an irreducible $G\lb\sigma\rb~$-representation according to \cite[Proposition 3.8]{Shu1}. By \cite[Proposition 4.4]{Shu1}, irreducible $G\lb\sigma\rb~$-representations have finite abelian stabilisers, thus every $G(\bar{\mathbb{F}}_q)$-orbit is closed. We then have a natural bijection of sets $(\Rep^{\bar{\phi}}\ds G)(\bar{\mathbb{F}}_q)\isom \Rep^{\bar{\phi}}(\bar{\mathbb{F}}_q)/G(\bar{\mathbb{F}}_q)$. Therefore,
\begin{equation}
\Ch^{\phi}(\mathbb{F}_q)\cong\Ch^{\bar\phi}(\bar{\mathbb{F}}_q)^F\cong(\Rep^{\bar{\phi}}(\bar{\mathbb{F}}_q)/G(\bar{\mathbb{F}}_q))^F,
\end{equation}
i.e. the set of $F$-stable $G$-orbits. Since $G$ is connected, each $F$-stable $G(\bar{\mathbb{F}}_q)$-orbit in $\Rep^{\bar{\phi}}(\bar{\mathbb{F}}_q)$ must contain some $F$-stable point by the Lang-Steinberg theorem. We will prove that the number of $F$-stable points in each such orbit is exactly $|G(\mathbb{F}_q)|$.

Let $O$ be an $F$-stable $G(\bar{\mathbb{F}}_q)$-orbit in $\Rep^{\bar{\phi}}(\bar{\mathbb{F}}_q)$, then $O^F$ splits into some $G(\mathbb{F}_q)$-orbits according to the stabliliser in $G$ of some $F$-stable point, say $x\in O^F$. By \cite[Proposition 4.4]{Shu1} again, the stabiliser is a finite abelian group $H$. The number of $G(\mathbb{F}_q)$-orbits in $O^F$ is equal to the number of $F$-conjugacy classes in $H$. 

Since $H$ is abelian, each $F$-conjugacy class of it is of the form $\{h_0hF(h)^{-1}\mid h\in H\}$ for some $h_0\in H$. Again because $H$ is abelian, the map $h\mapsto hF(h)^{-1}$ is a group homomorphism, with kernel $K=\{h\in H\mid F(h)=h\}$. Denote by $I$ the image of this homomorphism. Then the $F$-conjugacy classes in $H$ are the cosets $h_0I$, therefore there are $|H|/|I|=|K|$ of them.

That is, the number of $G(\mathbb{F}_q)$-orbits in $O$ is $|K|$. On the other hand, $x$ has $K$ as its stabiliser in $G(\mathbb{F}_q)$, so the cardinality of the $G(\mathbb{F}_q)$-orbit containing $x$ is $|G(\mathbb{F}_q)|/|K|$. If for some $g\in G(\bar{\mathbb{F}}_q)$, $g.x$ is an $F$-stable point contained in another $G(\mathbb{F}_q)$-orbit, then its stabiliser in $G(\mathbb{F}_q)$ is just $gKg^{-1}$. Indeed, we have $g^{-1}F(g)\in\Stab_G(x)=H$, thus for any $k\in K$, 
\begin{equation}
F(gkg^{-1})=F(g)kF(g)^{-1}=g(g^{-1}F(g))k(F(g)^{-1}g)g^{-1}=gkg^{-1}.
\end{equation}
Thus all $G(\mathbb{F}_q)$-orbits in $O$ have cardinality $|G(\mathbb{F}_q)|/|K|$. We conclude that $|O(\mathbb{F}_q)|=|G(\mathbb{F}_q)|$, which is independent of $O$. Therefore, 
\begin{equation}
|\Ch^{\phi}(\mathbb{F}_q)|=\frac{1}{|G(\mathbb{F}_q)|}|\Rep^{\phi}(\mathbb{F}_q)|.
\end{equation}
\end{proof}

%%%%%%%%%%%%%%%%%%%%%%%
\section{Computation of Linear Characters}\label{Section-Lem}
In this section, we fix an odd prime power $q$. The base field is $\kk=\bar{\mathbb{F}}_q$.
\subsection{A constraint on types}
\subsubsection{}
Theorem \ref{thm-decIrr} concerns only those irreducible characters of $G(q)$ that are induced from $\sigma$-stable standard Levi subgroups. We want to show that these characters suffice. 
\begin{Lem}\label{restrictTypes}
Let $(C_1,\ldots,C_{2k})$ be a tuple of semi-simple conjugacy classes in $G(q)\sigma$ such that for each $1\le j\le 2k$, the $G$-conjugacy class containing $C_j$ has a representative $s_j\sigma\in(T^{\sigma})^{\circ F}\sigma$. Let $M$ be a $\sigma$-stable and $F$-stable Levi factor of a $\sigma$-stable parabolic subgroup of $G$. Suppose that $(C_j)_j$ is strongly generic and that for all $j$, $C_j\cap M\sigma\ne\emptyset$. Then $M$ is $G(q)$-conjugate to a $\sigma$-stable standard Levi subgroup.
\end{Lem}
\begin{proof}
We may assume that $M=L_{I,w}$ for some $\sigma$-stable $I$ and some $w\in W_G(L_I)^{\sigma}$ (see \S \ref{Prop1.40} and \S \ref{GLnSigma-F-Levi}). There exists $\dot{w}\in(G^{\sigma})^{\circ}$ representing $w$ and $g\in(G^{\sigma})^{\circ}$ such that $g^{-1}F(g)=\dot{w}$. Then $L_{I,w}=gL_Ig^{-1}$.

For each $j$, let $t_j\sigma\in C_j\cap M\sigma$. For any $j$, $g^{-1}t_j\sigma g$ lies in $L_I\sigma$. By Proposition \ref{DM1.16}, for each $j$, there exists $l_j\in L_I$ such that $l_jg^{-1}t_j\sigma gl_j^{-1}$ lies in $(T^{\sigma})^{\circ}\sigma$, and moreover, for each $j$, there exists $w_j\in W^{\sigma}$ and $z_j\in (T^{\sigma})^{\circ}\cap [T,\sigma]$ such that $$l_jg^{-1}t_j\sigma gl_j^{-1}=w_js_jw_j^{-1}z_j\sigma.$$ For any connected reductive algebraic group $H$, denote by 
$$\mathbf{D}_H:H\longrightarrow Z_H^{\circ}/(Z_H^{\circ}\cap[H,H])$$
the natural surjection, identifying $H/[H,H]\cong Z_H^{\circ}/(Z_H^{\circ}\cap[H,H])$. For $H=\GL_n$, this is the determinant. We have $$\mathbf{D}_{L_I}((l_jg^{-1}t_j\sigma gl_j^{-1})^2)=\mathbf{D}_{L_I}((g^{-1}t_j\sigma g)^2).$$ Since $s_j$ lies in $(T^{\sigma})^{\circ}$, $z_j\sigma$ commutes with $w_js_jw_j^{-1}$ for all $j$. Since $z_j$ lies in $[T,\sigma]$, $(z_j\sigma)^2=\sigma^2$. Note also that $\sigma^2=1$ by definition. We deduce that 
\begin{equation}\label{l}
\mathbf{D}_{L_I}\big(\prod_{j=1}^{2k}(g^{-1}t_j\sigma g)^2\big)=\mathbf{D}_{L_I}\big(\prod_{j=1}^{2k}(w_js_j^2w_j^{-1})\big),
\end{equation}
and we denote by $l$ this element of $Z_{L_I}^{\circ}/(Z_{L_I}^{\circ}\cap[L_I,L_I])$. 

Let us choose an isomorphism $L_I\cong \GL_{n_0}\times\prod_i(\GL_{n_i}\times\GL_{n_i})$, which induces $$Z_{L_I}^{\circ}/(Z_{L_I}^{\circ}\cap[L_I,L_I])\cong \kk^{\ast}\times \prod_{i}(\kk^{\ast}\times \kk^{\ast}).$$ The action of $\sigma$ on the factor $\kk^{\ast}\times\kk^{\ast}$ corresponding to each $i$ sends $(x,y)$ to $(y^{-1},x^{-1})$, and on the first factor sends  $x$ to $x^{-1}$. From the right hand side of (\ref{l}), we see that $l$ is in fact a $\sigma$-stable element for this action. Thus we can write $l=(l_0,(l_i,l_i^{-1}))$ under the above isomorphism, where $l_0$ in fact equals to $1$.

In view of the right hand side of (\ref{l}), each $l_i$ is of the form $[\mathbf{A}_1,\mathbf{e}^1]_1\cdots[\mathbf{A}_{2k},\mathbf{e}^{2k}]_{2k}$ (See \S \ref{explicit-genconj-Fq} for the notations). We can then apply Lemma \ref{disct-e.v.-genconj} and conclude that $l_i\ne l_j^{\pm 1}$ whenever $i\ne j$. Moreover, every $l_i$ lies in $\mathbb{F}_q$. However, the left hand side of (\ref{l}) shows that $l$ is an $F_w$-stable element (see (\ref{F_w})). The action of a (signed) cycle in $w$ on $L_I$ is given by (\ref{ActionSymetrique}). It is easy to see that if there is some cycle of size larger than $1$, then $l_i=l_j$ for some $i\ne j$, which is a contradiction. Finally, if there is some negative cycle of size $1$, then the corresponding factor $l_i$ satisfies $l_i=l_i^q=l_i^{-1}$ and so must be equal to $\pm 1$. This possibility is ruled out by the assumption (\ref{GenConGLn2}). We conclude that $w=1$.
\end{proof}
By Theorem \ref{sig-st-chi}, each $\sigma$-stable irreducible character of $G(q)$ is of the form $\chi=\epsilon_G\epsilon_M R^G_M(\chi_1\boxtimes\chi_0)$ for some $F$-stable and $\sigma$-stable Levi factor $M$ of some $\sigma$-stable parabolic subgroup. By Theorem \ref{thm-decIrr}, the extension $\tilde{\chi}$ of $\chi$ to $G(q)\sigma$ is a linear combination of inductions from the twisted maximal tori in $M\sigma$, assuming that such an extension is a uniform function. By Proposition \ref{char-formula}, if $C_j\cap M\sigma=\emptyset$, then $\tilde{\chi}(C_j)=0$. The above lemma then implies that only those $\chi$ with $M$ being a standard Levi subgroup can have an extension $\tilde{\chi}$ that is non-vanishing on all $C_j$, as long as $\mathcal{C}$ is a strongly generic tuple of semi-simple conjugacy classes.

\subsection{M\"obius inversion functions}\hfill

The results in this subsection will only be used in \S \ref{subsec-Sum-LChar}.
\subsubsection{}\label{MoInv}
Let $m$ be a positive integer and let $I$ be a set of cardinality $x>m$. Denote by $\Pi$ the set of the partitions of $\mathbb{I}(m)=\{1,\ldots,m\}$. It is a partially ordered set. Two partitions satisfy $P_1\prec P_2$ if $P_2$ refines $P_1$. Denote by $P_0$ the partition into $m$ parts, each consisting of a single element. Then $P_0$ is the maximal element. Each $P\in\Pi$ can be written as a collection of disjoint subsets $p_1\cdots p_l$ of $\mathbb{I}(m)$. Each $p_i$ is called a part of $P$ and $l$ is called the length of $P$, denoted by $l(P)$. For any $P\in\Pi$, denote by $(I^m)_{P}$ the subset of $I^m$ consisting of the elements $(i_r)_{1\le r\le m}$ such that $i_r=i_s$ whenever $r$ and $s$ are in the same part of $P$. Denote by $(I^m)_{P,\reg}$ the set of the elements $(i_r)$ of $(I^m)_{P}$ such that $i_r\ne i_s$ whenever $r$ and $s$ are not in the same part of $P$. Obviously, $$(I^m)_{P_2}=\bigsqcup_{P_1\prec P_2}(I^m)_{P_1,\reg},$$ for any $P_2\in\Pi$. A $\ladic$-valued function $f$ defined on the set of the subsets of $I^m$ is called additive if $f(U\cup V)=f(U)+f(V)$ for any two disjoint subsets $U$ and $V$. Given such a function $f$, we can define two functions $F$ and $F'$ on $\Pi$ by $$F(P):=f((I^m)_P),\quad F'(P):=f((I^m)_{P,\reg}).$$ Then $$F(P_2)=\sum_{P_1\prec P_2}F'(P_1).$$ By the M\"obius inversion formula, we have 
\begin{equation}\label{MoInvFor}
F'(P_2)=\sum_{P_1\prec P_2}\mu(P_1,P_2)F(P_1),
\end{equation}
where $\mu(P_1,P_2)$ is the M\"obius inversion function for the partially ordered set $\Pi$. 

Define $c_{P}(x):=|(I^m)_{P}|$, $c'_{P}(x):=|(I^m)_{P,reg}|$. We have $c_{P}(x)=(x)^{l(P)}$ and $$c'_{P}(x)=x(x-1)\cdots(x-l(P)+1).$$ These are the functions defined by counting elements, and so are additive. Inserting these expressions into the M\"obius inversion formula, we have a polynomial identity in $x$ that is valid for all large enough $x$. Specialising this equality at $x=-1$, we get 
\begin{equation}\label{x=-1}
(-1)^mm!=\sum_{P\prec P_0}\mu(P,P_0)(-1)^{l(P)}.
\end{equation}

\subsubsection{}\label{MoInvDouble}
Let $I$, $m$ and $\Pi$ be as above. Let $I^{\ast}$ be a set in bijection with $I$ and we fix a bijection between them. For any $i\in I$, let $i^{\ast}$ denote the element of $I^{\ast}$ corresponding to $i$ under the given bijection. Write $\bar{I}=I\sqcup I^{\ast}$. For any $i$, $i'\in\bar{I}$, we write $[i]=[i']$ if $i'\in\{i,i^{\ast}\}$. For any $P\in\Pi$, denote by $(\bar{I}^m)_P$ the subsets of elements $(i_r)_{1\le r\le m}$, $i_r\in\bar{I}$, such that $[i_r]=[i_s]$ whenever $r$ and $s$ are in the same part of $P$. Denote by $(\bar{I}^m)_{P,\reg}$ the set of the elements $(i_r)$ of $(\bar{I}^m)_P$ such that $[i_r]\ne[i_s]$ whenever $r$ and $s$ are not in the same part of $P$. We say that an element of $\bar{I}^m$ is regular if it is regular with respect to the maximal partition $P_0$. Obviously, $$(\bar{I}^m)_{P_2}=\bigsqcup_{P_1\prec P_2}(\bar{I}^m)_{P_1,\reg},$$ for any $P_2\in\Pi$. 

Now, put $\bar{I}=\Irr(\mathbb{F}^{\ast}_q)\setminus\{1,\eta\}$ and $\tilde{I}=(\Irr(\mathbb{F}^{\ast}_q)\setminus\{1,\eta\})/\!\sim$ for the equivalence relation $\sim$ that identifies $\alpha$ and $\alpha^{-1}$. For each equivalence class we choose a representative and denote by $I$ the set of these representatives, and so for any $\alpha\in I$, $\alpha^{\ast}=\alpha^{-1}$. Let $(a_i)_{1\le i\le m}$ be an $m$-tuple of elements of $\mathbb{F}_q^{\ast}$. Define an additive function on the set of the subsets of $\bar{I}^m$ by $$f(J):=\sum_{(\alpha_i)_{1\le i\le m}\in J}\prod_{1\le i\le m}\alpha_i(a^2_i),$$ for $J\subset \bar{I}^m$. We define the functions on $\Pi$: $$F(P):=f((\bar{I}^m)_P),\quad F'(P):=f((\bar{I}^m)_{P,\reg}).$$ Then,
\begin{equation}\label{MoInvFF'}
F'(P_0)=\sum_{P_1\prec P_0}\mu(P_1,P_0)F(P_1).
\end{equation}

\subsubsection{}\label{MoInvPart}
Let $s$ be a positive integer. Let $(a_i)_{1\le i\le s}$ be an $s$-tuple of elements of $\mathbb{F}_q^{\ast}$ such that $\prod_{1\le i\le s}a_i^{2e_i}\ne 1$ for any $(e_i)_{1\le i\le s}\in\bs\mu_2^s$. 
\begin{Lem}\label{MoInvLem}
With the notations of \S \ref{MoInvDouble}, we have the identity:
\begin{equation*}
\sum_{\alpha\in I}~\sum_{(e_i)_{1\le i\le s}\in\bs\mu_2^s}~\prod_{1\le i\le s}\alpha(a^{2e_i}_i)=-2^s.
\end{equation*}
\end{Lem}
\begin{proof}
Observe that 
\begin{equation*}
\sum_{(e_i)_{1\le i\le s}\in\bs\mu_2^s}\alpha^{-1}(\prod_{1\le i\le s}a^{2e_i}_i)=\sum_{(e_i)_{1\le i\le s}\in\bs\mu_2^s}\alpha(\prod_{1\le i\le s}a^{-2e_i}_i)=\sum_{(e_i)_{1\le i\le s}\in\bs\mu_2^s}\alpha(\prod_{1\le i\le s}a^{2e_i}_i),
\end{equation*}
so the desired quantity is equal to $$\frac{1}{2}\sum_{\alpha\in \bar{I}}~\sum_{(e_i)_{1\le i\le s}\in\bs\mu_2^s}\alpha(\prod_{1\le i\le s}a^{2e_i}_i).$$ Since $\prod_{1\le i\le s}a_i^{2e_i}\ne 1$, we have $$0=\sum_{\alpha\in\Irr(\mathbb{F}_q^{\ast})}\alpha(\prod_{1\le i\le s}a_i^{2e_i})=\sum_{\alpha\in\bar{I}}\alpha(\prod_{1\le i\le s}a_i^{2e_i})+2,$$that is, $$\sum_{\alpha\in \bar{I}}\alpha(\prod_{1\le i\le s}a^{2e_i}_i)=-2.$$ We get $\frac{1}{2}2^s\cdot (-2)=-2^s$.
\end{proof}

We use the above lemma to compute $F(P_1)$ in (\ref{MoInvFF'}).
\begin{Lem}\label{MoInvLem2}
We have $F(P_1)=(-1)^{l(P_1)}2^m$.
\end{Lem}
\begin{proof}
Write $P_1=p_1\cdots p_l$ with $l=l(P_1)$. For each $k\in\{1,\ldots,l\}$, denote by $\bar{I}^{p_k}$ the components of $\bar{I}^m$ indexed by the elements of $p_k$, and denote by $(\bar{I}^{p_k})_0$ the subset consisting of $(i_r)_{r\in p_k}$ satisfying $[i_r]=[i_{r'}]$ for any $r$, $r'\in p_k$. Since $(\bar{I}^m)_{P_1}=\prod_k(\bar{I}^{p_k})_0$, we have $$F(P_1)=\sum_{(\alpha_i)_{1\le i\le m}\in (\bar{I}^m)_{P_1}}\prod_{1\le i\le m}\alpha_i(a^2_i)=\prod_{k=1}^l\sum_{(\alpha_i)_{i\in p_k}\in (\bar{I}^{p_k})_0}\prod_{i\in p_k}\alpha_i(a^2_i).$$ Note that $$\sum_{(\alpha_i)_{i\in p_k}\in (\bar{I}^{p_k})_0}\prod_{i\in p_k}\alpha_i(a^2_i)=\sum_{\alpha\in I}~\sum_{(e_i)_{i\in p_k}\in\bs\mu_2^{p_k}}~\prod_{i\in p_k}\alpha^{e_i}(a^{2}_i),$$which is equal to $-2^{|p_k|}$ by Lemma \ref{MoInvLem}. It then only remains to take the product over $k$.
\end{proof}

\subsection{A formula for linear characters}
\subsubsection{}
Let $T:=(\kk^{\ast})^{\bar{\mathbb{I}}(m)}$ be the torus with components indexed by $$\bar{\mathbb{I}}(m)=\{1,\ldots,m,-m,\ldots,-1\}.$$ Let $F$ be the Frobenius of $T$ sending $t$ to $t^q$. Let $w_+$ (resp. $w_-$) be the automorphism of $T$ that permutes its factors according to (\ref{w_+}) (resp. (\ref{w_-})). Write $F_+=\ad w_+\circ F$ and $F_-=\ad w_-\circ F$. These are Frobenius endomorphisms of $T$ such that $T^{F_+}\cong \mathbb{F}^{\ast}_{q^m}\times\mathbb{F}^{\ast}_{q^m}$ and $T^{F_-}\cong \mathbb{F}^{\ast}_{q^{2m}}$. 
\begin{Lem}\label{evaluate-st}
Let $$u=(u_1,\ldots,u_m,u_m^{-1},\ldots,u_1^{-1})\in T^F,$$and let $$t=(t_1,\ldots,t_m,t_m,\ldots,t_1)\in T.$$Suppose that $ut$ lies in $T^{F_+}$ or in  $T^{F_-}$. Then,
\begin{itemize}
\item[(i)] $u_{i+1}^2=u_i^2$ for all $i$;
\item[(ii)] $u_i^4=1$ for all $i$ if $ut\in T^{F_-}$;
\item[(iii)] $\prod_i^mt_i^2=\prod_{i=1}^mt_1^{2q^{i-1}}$ and, 
$$
\begin{cases}
t_1^{1+q+\cdots+q^{m-1}}\in\mathbb{F}_q^{\ast} & \text{ if } ut\in T^{F_+};\\
t_1^{1+q+\cdots+q^{m-1}}\in\mathbb{F}_q^{\ast} & \text{ if } ut\in T^{F_-} \text{ and } u_i^2=1;\\
t_1^{2(1+q+\cdots+q^{m-1})}\in\mathbb{F}_q^{\ast}\setminus(\mathbb{F}_q^{\ast})^2 & \text{ if } ut\in T^{F_-} \text{ and } u_i^2=-1.
\end{cases}
$$
\end{itemize}
\end{Lem}
\begin{proof}
By the assumption on $u$, we have $u_i^q=u_i$ for all $i$. Therefore, $$t_i^qu_i=t_{i+1}u_{i+1},\quad
t_i^qu_i^{-1}=t_{i+1}u_{i+1}^{-1},$$
and so $u_{i+1}^2=u_i^2$.

Suppose $ut\in T^{F_+}$. Then $t_1^{q^m}u_1=t_1u_1$, and so $t_1^{q^m-1}=1$. This proves the first part of (iii).

Now suppose $ut\in T^{F_-}$. Since $$t_1^{q^m}u_1=t_1u_1^{-1},\quad t_1^{q^m}u_1^{-1}=t_1u_1,$$ we deduce that $u_1^4=1$ and
\begin{equation}\label{eq-a_1^2}
t_1^{q^m-1}=u_1^2.
\end{equation}

The proof of part (i) shows that $t_{i+1}^2=t_i^{2q}$. Applying the norm map to $ut$ we deduce that $$\prod_{i=1}^mt_1^{2q^{i-1}}=\prod_{i=1}^mt_i^2\in\mathbb{F}_q^{\ast}.$$Write $x=\prod_{i=1}^mt_1^{q^{i-1}}$. Then $x^2\in\mathbb{F}_q^{\ast}$, while (\ref{eq-a_1^2}) means that $x^{q-1}=u_1^2$, whence the assertions.
\end{proof}

Let $\alpha$, $\beta\in\Irr(\mathbb{F}_q^{\ast})$. We regard $\alpha$ as an element of $\Irr(\mathbb{F}_{q^m}^{\ast})$ via the norm map. Then $\alpha$ defines an irreducible character $(\alpha,\alpha^{-1})$ of $T^{F_+}\cong \mathbb{F}^{\ast}_{q^m}\times\mathbb{F}^{\ast}_{q^m}$, denoted by $\theta_{\alpha}$. We regard $\beta$ as an element of $\Irr(\mathbb{F}_{q^{2m}}^{\ast})$ via the norm map. Then $\beta$ defines an irreducible character of $T^{F_-}\cong \mathbb{F}^{\ast}_{q^{2m}}$, denoted by $\theta'_{\beta}$.
\begin{Lem}\label{2-evaluate-st}
With the same notations as in Lemma \ref{evaluate-st}, we have
\begin{itemize}
\item[(i)] $\theta_{\alpha}(ut)=\alpha(\prod_i^mu_i^2)$ if $ut\in T^{F_+}$;
\item[(ii)] $\theta'_{\beta}(ut)=\beta(\prod_i^mt_i^2)=\beta(t_1^{2(1+q+\cdots+q^{m-1})})$ if $ut\in T^{F_-}$.
\end{itemize}
\end{Lem}
\begin{proof}
Obvious.
\end{proof}

\subsubsection{}\label{value-tildetheta}
Let $w\in W^{\sigma}$, then there exists $\dot{w}\in(G^{\sigma})^{\circ}$ representing $w$. Let $g\in(G^{\sigma})^{\circ}$ be such that $g^{-1}F(g)=\dot{w}$ and put $T_w=gTg^{-1}$. Suppose $n=2N+1$. We can write $$w=((e_1,\ldots,e_N),\tau)\in(\mathbb{Z}/2\mathbb{Z})^N\rtimes\mathfrak{S}_N$$ as in \S \ref{signed-partitions}. Then $\tau$ is a product of cycles: $\tau=\prod_{i\in\Lambda}c_{I_i}$, where $\Lambda$ is just a set indexing the cycles. The set $\Lambda$ is divided into two parts: $\Lambda_+$ and $\Lambda_-$ so that $\bar{e}_i=-1$ if and only if $i\in\Lambda_-$ (see \S \ref{signed-partitions}). The size of $I_i$ is denoted by $\tau_i$ so that $((\tau_i)_{i\in\Lambda_+},(\tau_i)_{i\in\Lambda_-})$ is the 2-partition corresponding to the conjugacy class of $w$. Now $$T_w^F\cong T^{F_w}\cong\mathbb{F}_q^{\ast}\times\prod_{i\in\Lambda_+}(\mathbb{F}_{q^{\tau_i}}^{\ast}\times\mathbb{F}_{q^{\tau_i}}^{\ast})\times\prod_{i\in\Lambda_-}\mathbb{F}_{q^{2\tau_i}}^{\ast}.$$
For each $i\in\Lambda_+$, let $\alpha_i\in\Irr(\mathbb{F}_q^{\ast})$ and for each $i\in\Lambda_-$, let $\beta_i=\mathbf{1}$ or $\eta$. Let $\theta_0=\mathbf{1}$ or $\eta$ considered as associated to the first component under the above isomorphism. Put 
\begin{equation}\label{epsilon-theta}
\epsilon(\theta_0)=
\begin{cases}
1 & \text{ if } \theta_0=\mathbf{1},\\
-1 & \text{ if } \theta_0=\eta.
\end{cases}
\end{equation}
The characters $\alpha_i$, $\beta_i$ and $\theta_0$ are not necessarily distinct. Under the above isomorphism, we regard $$\theta:=\left(\theta_0,(\alpha_i,\alpha_i^{-1})_{i\in\Lambda_+},(\beta_i)_{i\in\Lambda_-}\right)$$ as a character of $T^{F_w}$ via the norm maps. Then $\theta$ is obviously a $\sigma$-stable character, and so can be extended to a character of $T^{F_w}\lb\sigma\rb$, denoted by $\tilde{\theta}$. We require that $\tilde{\theta}$ takes value $1$ at $\sigma$. Define $\epsilon(\theta):=\epsilon(\theta_0)$ if $n$ is odd.

\begin{Lem}\label{theta-st}
Let $\theta=(\theta_0,(\alpha_i,\alpha_i^{-1})_{i\in\Lambda_+},(\beta_i)_{i\in\Lambda_-})$ be the character of $T^{F_w}$ as above and let 
\begin{equation}\label{s'sigma}s'=(a'_1,\ldots,a'_N,c,a_N^{\prime-1},\ldots,a_1^{\prime-1})\in T^F.
\end{equation} Suppose that there exists $h\in G(q)$ such that $hs'\sigma h^{-1}\in T_w\sigma$. Then the following assertions hold.

(i). The element $g^{-1}hs'\sigma h^{-1}g$ can be written as $st\sigma$, where $$t=\diag(t_1,\ldots,t_N,t_{N+1},t_N,\ldots,t_1)\in T,$$ and $$s=\diag(a_1,\ldots,a_N,c,a_N^{-1},\ldots,a_1^{-1})\in T^F$$ satisfies $a_{\gamma}=a^{\prime e_{\gamma}}_{\rho^{-1}(\gamma)}$, $1\le\gamma\le N$, for some permutation $\rho$ of $\{1,\ldots,N\}$ and some $N$-tuple of signs $(e_1,\ldots,e_N)$. 

(ii). Define $$\Lambda_-^{\circ}=\{i\in\Lambda_-\mid a_{\gamma}^2=-1\text{ for any $\gamma\in I_i$}\},$$ and $$\Lambda_-^{\circ\circ}=\{i\in\Lambda_-^{\circ}\mid \beta_i=\eta\}.$$ The set $\Lambda_-^{\circ}$ is well-defined according to Lemma \ref{evaluate-st}. Then we have $$\tilde{\theta}(st\sigma)=\theta_0(c)\prod_{i\in\Lambda_+}\alpha_i(\prod_{\gamma\in I_i}a^2_{\gamma})\prod_{i\in\Lambda_-^{\circ\circ}}(-1)\prod_{i\in\Lambda_-^{\circ}}\epsilon(\theta_0),$$
\end{Lem}
\begin{proof}
Put $x=g^{-1}h$. That $xs'\sigma x^{-1}\in T\sigma$ implies that $x=\xi l$ with $\xi\in N_G(T\sigma)$ and $l\in C_G(s'\sigma)^{\circ}$ by Lemma \ref{Shu210.2.2a}. Write $\xi$ as $t'\dot{v}$, with $t'\in T$ and $\dot{v}$ fixed by $\sigma$. Such $\dot{v}$ exists since $\sigma$ is quasi-central. Write $$\dot{v}s'\dot{v}^{-1}=:s=\diag(a_1,\ldots,a_N,c,a_N^{-1},\ldots,a_1^{-1}),$$for some $a_i\in\mathbb{F}_q^{\ast}$, $1\le i\le N$, and write $$t'=\diag(t'_1,\ldots,t'_{N+1},\ldots,t'_{2N+1}),$$for some $t'_i\in\kk^{\ast}$, $1\le i\le 2N+1$, then $$t'\dot{v}s'\sigma \dot{v}^{-1}t^{\prime-1}=\diag(t'_1t'_{2N+1}a_1,\ldots,t'_Nt'_{N+2}a_N,t_{N+1}^{\prime 2}c,t'_Nt'_{N+2}a_N^{-1},\ldots,t'_1t'_{2N+1}a_1^{-1})\sigma.$$ Put $t=\diag(t_1,\ldots,t_N,t_{N+1},t_N,\ldots,t_1)$, with $t_i=t'_it'_{2N+2-i}$, $1\le i\le N+1$. Then, $$g^{-1}hs'\sigma h^{-1}g=st\sigma.$$Let $v$ be the image of $\dot{v}$ under the natural map $N_G(T\sigma)\rightarrow W^{\sigma}$ (see Lemma \ref{Shu210.2.2b}). Then $v$ can be written as $((e_1,\ldots,e_N),\rho)$ under the isomorphism $W^{\sigma}\cong(\mathbb{Z}/2\mathbb{Z})^N\rtimes\mathfrak{S}_N$. In fact, $\rho$ and $(e_1,\ldots,e_N)$ are the desired permutation and $N$-tuple of signs. 

To compute $\theta(st)=\tilde{\theta}(st\sigma)$, we apply Lemma \ref{evaluate-st} and Lemma \ref{2-evaluate-st} to each cycle in $\tau$ and get $$\prod_{i\in\Lambda_+}\alpha_i(\prod_{k\in I_i}a^{2}_{\gamma})\prod_{i\in\Lambda_-^{\circ\circ}}(-1).$$ It remains to tackle the entry $t_{N+1}c$ of $st$. From the equality $g^{-1}hs'\sigma h^{-1}g=st\sigma$, we see that $st$ is $F_w$-stable, and so $t_{N+1}c\in\mathbb{F}_q$. Since we have required that $\theta_0$ is either $\mathbf{1}$ or $\eta$, it suffices to determine whether $t_{N+1}c$ lies in $(\mathbb{F}_q^{\ast})^2$. We are going to compute $\det(ts)$ in two different ways. On the one hand,
\begingroup
\allowdisplaybreaks
\begin{align*}
\det(ts)&=t_1^2\cdots t_N^2t_{N+1}c\\
&\stackrel{\circled{1}}{\equiv} t_{N+1}c\prod_{i\in\Lambda_-}(\prod_{\gamma\in I_i}t_{\gamma}^2)\mod(\mathbb{F}_q^{\ast})^2\\
&\stackrel{\circled{2}}{\equiv} t_{N+1}c\prod_{i\in\Lambda_-^{\circ}}(\prod_{\gamma\in I_i}t_{\gamma}^2)\mod(\mathbb{F}_q^{\ast})^2,
\end{align*}
\endgroup
where the equivalences \circled{1} and \circled{2} follow from the first and the second part of Lemma \ref{evaluate-st} (iii). On the other hand, 
\begingroup
\allowdisplaybreaks
\begin{align*}
\det(st)&=\det(g^{-1}hs'\sigma(h)^{-1}g)\\
&=\det(hs'\sigma(h^{-1}))\\
&=\det(h^2s')\\
&\equiv \det(s')=c\mod(\mathbb{F}_q^{\ast})^2.
\end{align*}
\endgroup
Comparing the two expressions for $\det(st)$, and using Lemma \ref{evaluate-st} (iii) and Lemma \ref{2-evaluate-st} (ii), we deduce that $$\eta(t_{N+1})=\prod_{i\in\Lambda_-^{\circ}}(-1),$$and so $$\theta_0(t_{N+1}c)=\theta_0(c)\prod_{i\in\Lambda_-^{\circ}}\epsilon(\theta_0).$$
\end{proof}
\begin{Cor}\label{Cor-theta-st}
With the same assumptions as in Lemma \ref{theta-st}, we assume further that $a_{\gamma}^2\ne-1$ for all $1\le\gamma\le N$. Then $$\tilde{\theta}(g^{-1}hs'\sigma h^{-1}g)=\theta_0(c)\prod_{i\in\Lambda_+}\alpha_i(\prod_{\gamma\in I_i}a^{2}_{\gamma}).$$
\end{Cor}
\begin{proof}
Clear.
\end{proof}
\begin{Rem}
If $n$ is even, then the same arguments show that $$\tilde{\theta}(g^{-1}hs'\sigma h^{-1}g)=\prod_{i\in\Lambda_+}\alpha_i(\prod_{\gamma\in I_i}a^{2}_{\gamma})\prod_{i\in\Lambda_-^{\circ\circ}}(-1),$$where each term in this expression is defined in a way similar to the odd dimensional case. And if we assume $a_{\gamma}^2\ne-1$ for all $\gamma$, then $$\tilde{\theta}(g^{-1}hs'\sigma h^{-1}g)=\prod_{i\in\Lambda_+}\alpha_i(\prod_{\gamma\in I_i}a^{2}_{\gamma}).$$
\end{Rem}

\subsection{Sum of linear characters}\label{subsec-Sum-LChar}
\subsubsection{}\label{Irr-reg-sig}
Let $M=L_I$ be a $\sigma$-stable standard Levi subgroup. As before, we write $M=M_0\times M_1$, where $M_0=\GL_{n_0}$ and $M_1=\prod_{i=1}^l(\GL_{n_i}\times\GL_{n_i})$ for some integers $l$, $n_0$, and $n_i$. Let $N_+$ and $N_-$ be some non negative integers such that $N_++N_-=[n_0/2]$, and write $\mathfrak{W}_+:=\mathfrak{W}_{N_+}$, $\mathfrak{W}_-:=\mathfrak{W}_{N_-}$. For each $j\in\{1,\ldots 2k\}$, let $$w_j\in \mathfrak{W}_+\times\mathfrak{W}_-\times W_1^{\sigma}\subset W_M^{\sigma}(T).$$To each $w_j$ is associated an $F$-stable and $\sigma$-stable maximal torus $T_{w_j}\subset M$ as in \S \ref{decIrr}. Recall that $\Irr_{\reg}^{\sigma}(M_1^F)$ is the set of $\sigma$-stable regular linear characters of $M_1^F$ in the sense of Definition \ref{reg-M_1}. Write $\bar{I}:=\Irr(\mathbb{F}^{\ast}_q)\setminus\{\mathbf{1},\eta\}$. Then the elements of $\Irr_{\reg}^{\sigma}(M_1^F)$ are exactly the regular elements of $\bar{I}^l$ in the sense of \S \ref{MoInvDouble}. Given $\theta\in\Irr_{\reg}^{\sigma}(M_1^F)$ and $\epsilon\in\{-1,0,+1\}$, the procedure of \S \ref{decIrr} gives a $\sigma$-stable linear character of $T_{w_j}^F$ for each $j$, which is denoted by $\theta_j$. And we denote by $\tilde{\theta}_j$ its extension to $T_{w_j}^F\lb\sigma\rb$ which equals to $1$ at $\sigma$. If $n$ is odd, then $\epsilon(\theta_j)=\epsilon$ by the definition of $\theta_j$.

Let $\mathcal{C}=(C_1,\ldots,C_{2k})$ be a tuple of semi-simple conjugacy classes in $G(q)\sigma$ with representatives $s_j\sigma$ of the form (\ref{s'sigma}). For each $j$, regard $s_j\sigma$ as the $s'\sigma$ in Lemma \ref{theta-st}, then let $c_j$ and $h_j$ be associated to $s_j\sigma$ and $w_j$ in the way $c$ and $h$ were associated to $s'\sigma$ and $w$. Similarly, let $\Lambda_j$, $\Lambda_{j,+}$, $\Lambda_{j,-}$, and $I_{j,i}$ with $i$ running over $\Lambda_j$, be defined like $\Lambda$, $\Lambda_+$, $\Lambda_-$, and $I_i$ respectively. For each $j$, we can write $$\theta_j:=\left(\theta_{j,0},(\alpha_{j,i},\alpha_{j,i}^{-1})_{i\in\Lambda_{j,+}},(\beta_{j,i})_{i\in\Lambda_{j,i}}\right),$$ for some characters $\theta_{j,0}$, $\alpha_{j,i}$ and $\beta_{j,i}$ of $\mathbb{F}_q^{\ast}$. The characters $\theta_{j,0}$'s are the same for all $j$, and only depend on $\epsilon$ by definition. Thus we will write $\theta_0$ instead of $\theta_{j,0}$.

\begin{Lem}\label{sumofreg}
With the above notations, put (\textit{cf.} Notation \ref{Cj+-})
\begin{equation}
\Delta_{\epsilon,\sgn\mathcal{C}}=
\begin{cases}
-1 & \text{if $\epsilon=-$ and $\sgn\mathcal{C}=-1$};\\
1 & \text{otherwise}.
\end{cases}
\end{equation}
Suppose that $\mathcal{C}$ is strongly generic and satisfies \ref{CCL}. We have:
\begin{equation}
\sum_{\theta\in\Irr_{\reg}^{\sigma}(M_1^F)}\prod^{2k}_{j=1}\tilde{\theta}_j(h_js_j\sigma h_j^{-1})=(-2)^ll!\Delta_{\epsilon,\sgn\mathcal{C}}.
\end{equation}
\end{Lem}
\begin{proof}
We will assume $\epsilon\ne0$. The case $\epsilon=0$ follows from similar arguments. Under \ref{CCL}, we are in the situation of Corollary \ref{Cor-theta-st}, and so for each $j$ we have,
$$\tilde{\theta}_j(h_js_j\sigma h^{-1}_j)=\theta_{0}(c_j)\prod_{i\in\Lambda_{j,+}}\alpha_{j,i}(\prod_{\gamma\in I_{j,i}}a^{2}_{j,\gamma}),$$where $a_{j,\gamma}$ is defined like $a_{\gamma}$ in Lemma \ref{theta-st}. Obviously, $$\prod_{j}^{2k}\theta_0(c_j)=\Delta_{\epsilon,\sgn\mathcal{C}},$$ and for any $j$ and $i\in\Lambda_{j,+}$, we have $\alpha_{j,i}(\prod_{\gamma\in I_{j,i}}a^{2}_{j,\gamma})=1$ if $\alpha_{j,i}=\mathbf{1}$ or $\eta$. We will only need to consider the remaining terms.

Recall that $\theta$ is defined by a set $\{\theta_1,\ldots,\theta_l\}$ of distinct characters of $\mathbb{F}_q^{\ast}$. For each $r\in\{1,\ldots,l\}$ and each $j$, put $$I_{j,r}:=\bigsqcup_{\{i\in\Lambda_{j,+}\mid \alpha_{j,i}=\theta_r\}}I_{j,i},$$and for each $r$, $$a_r:=\prod_{j=1}^{2k}\prod_{\gamma\in I_{j,r}}a_{j,\gamma}.$$ Then, $$\prod_{j=1}^{2k}\prod_{i\in\Lambda_{j,+}}\alpha_{j,i}(\prod_{\gamma\in I_{j,i}}a^{2}_{j,\gamma})=\prod_{r=1}^{l}\theta_r(a_r^2).$$ 
Therefore, 
$$
\sum_{\theta\in\Irr^{\sigma}_{\reg}(M_1^F)}\prod^{2k}_{j=1}\tilde{\theta}_j(h_js_j\sigma h_j^{-1})=\Delta_{\epsilon,\sgn\mathcal{C}}\sum_{\substack{(\theta_1,\ldots,\theta_l)\\\text{regular}}}\prod^l_{r=1}\theta_r(a_r^2).
$$ 

Using (\ref{MoInvFF'}) and Lemma \ref{MoInvLem2}, we get $$
\sum_{\substack{(\theta_1,\ldots,\theta_l)\\\text{regular}}}\prod^l_{r=1}\theta_r(a_r^2)=\sum_{P_1\prec P_0}\mu(P_1,P_0)(-1)^{l(P_1)}2^l,$$
which is equal to $(-2)^ll!$ by (\ref{x=-1}).
\end{proof}

%%%%%%%%%%%%%%%%%%%%%%%
\section{Symmetric Functions Associated to Wreath Products}\label{Section-SymF}
\subsection{Ring of Symmetric Functions}
\subsubsection{}
For $k\in\{0,1\}$, let $\mathbf{x}^{(k)}=(x^{(k)}_1,x^{(k)}_2,\cdots)$ be an infinite series of variables. Denote by $\mathbf{x}=(\mathbf{x}^{(0)},\mathbf{x}^{(1)})$ all of these variables. For any commutative ring with unity $R$, denote by $\SymF_R[\Xo,\Xl]$ the ring of symmetric functions in two variables (symmetric in each $\Xk$) with coefficents in $R$. We will always assume that $\mathbb{Q}$ is contained in $R$. The usual choice of $R$ will be the fields $\mathbb{Q}$, $\mathbb{Q}(z)$ and $\mathbb{Q}(z,w)$, i.e. the function fields in variables $z$ and $w$. If the base field is $\mathbb{Q}(z,w)$, then we write $\SymF_{z,w}[\Xo,\Xl]=\SymF_{\mathbb{Q}(z,w)}[\Xo,\Xl]$, and similarly for $\mathbb{Q}(z)$. We will omit the subscript $R$ if there is no confusion with the base ring. For any $r\in\mathbb{Z}_{>0}$, put
\begin{equation}
\begin{split}
p^{(0)}_r(\mathbf{x}):=p_r(\mathbf{x}^{(0)})+p_r(\mathbf{x}^{(1)}),\\
p^{(1)}_r(\mathbf{x}):=p_r(\mathbf{x}^{(0)})-p_r(\mathbf{x}^{(1)}),
\end{split}
\end{equation}
where $p_r(\mathbf{x}^{(k)})$ is the usual power sum in $\Xk$. Put $p_0^{(k)}(\mathbf{x})=1$ for both values of $k$. For a 2-partition $\boldsymbol{\lambda}=(\lambda^{(0)},\lambda^{(1)})$, with each partition $\lambda^{(k)}$ written as $\lambda^{(k)}_1\ge\lambda^{(k)}_2\ge\cdots$, the \textit{power sum} in $\mathbf{x}$ is defined by
\begin{equation}
p_{\boldsymbol{\lambda}}(\mathbf{x}):=\prod_ip^{(0)}_{\lambda^{(0)}_i}(\mathbf{x})\prod_jp^{(1)}_{\lambda^{(1)}_j}(\mathbf{x}).
\end{equation}
The \textit{Schur function}, \textit{monomial symmetric function} and the \textit{complete symmetric function} are respectively defined by
\begingroup
\allowdisplaybreaks
\begin{align*}
s_{\boldsymbol{\lambda}}(\mathbf{x}):=&s_{\lambda^{(0)}}(\Xo)s_{\lambda^{(1)}}(\Xl),\\
m_{\boldsymbol{\lambda}}(\mathbf{x}):=&m_{\lambda^{(0)}}(\Xo)m_{\lambda^{(1)}}(\Xl),\\
h_{\boldsymbol{\lambda}}(\mathbf{x}):=&h_{\lambda^{(0)}}(\Xo)h_{\lambda^{(1)}}(\Xl),
\end{align*}
\endgroup
where $s_{\lambda^{(k)}}(\Xk)$, $m_{\lambda^{(k)}}(\Xk)$, $h_{\lambda^{(k)}}(\Xk)$ are the usual Schur functions, monomial symmetric functions and complete symmetric functions associated to partitions. The usual results on $\SymF[\Xo]$ and $\SymF[\Xl]$ imply that each family of the symmetric functions above, with $\boldsymbol{\lambda}$ running through all 2-partitions, forms a basis of $\SymF[\Xo,\Xl]$.

In \cite{S01}, Shoji defined two families of symmetric functions $\{P_{\bs{\Lambda}}(\X;z)\}_{\bs\Lambda}$ and $\{Q_{\bs{\Lambda}}(\X;z)\}_{\bs\Lambda}$  in $\SymF_z[\Xo,\Xl]$. Each of these two families of functions forms a basis of $\SymF_z[\Xo,\Xl]$, and they are both called \textit{Hall-Littlewood functions}.
\begin{Rem}
In \cite{S01}, Shoji works with reflection groups $(\mathbb{Z}/r\mathbb{Z})^n\rtimes\mathfrak{S}_n$ with $r$ not necessarily equal to $2$. In this general setting, he defines the symmetric functions $P^{\pm}_{\bs{\Lambda}}(\X;z)$ and $Q^{\pm}_{\bs{\Lambda}}(\X;z)$ depending on a sign. When $n=2$, they degenerate into $P_{\bs{\Lambda}}(\X;z)$ and $Q_{\bs{\Lambda}}(\X;z)$.
\end{Rem}

\subsubsection{}\label{tildeH}
For any partition $\lambda=(1^{m_1},2^{m_2},\ldots)$, define $\zz_{\lambda}=\prod_ii^{m_i}m_i!$. For any 2-partition $\bs{\lambda}=(\lambda^{(0)},\lambda^{(1)})$, define $\zz_{\bs{\lambda}}=2^{l(\bs{\lambda})}\zz_{\lambda^{(0)}}\zz_{\lambda^{(1)}}$. 

According to \cite[Chapter I, Appendix B]{Mac}, we have
\begingroup
\allowdisplaybreaks
\begin{align}
p_{\boldsymbol{\beta}}(\X)&=\sum_{\bs{\alpha}\in\mathcal{P}_{n,2}}\chi^{\bs{\alpha}}_{\bs{\beta}}~s_{\bs{\alpha}}(\X);\\
s_{\bs{\lambda}}(\X)&=\sum_{\bs{\tau}}\frac{1}{\zz_{\bs{\tau}}}\chi^{\bs{\lambda}}_{\bs{\tau}}~p_{\bs{\tau}}(\X),
\end{align}
\endgroup
where $\chi^{\bs{\alpha}}_{\bs{\beta}}$ is the value of the character of $(\mathbb{Z}/2\mathbb{Z})^m\rtimes\mathfrak{S}_m$ ($m=|\bs{\alpha}|=|\bs{\beta}|$) of class $\bs{\alpha}$ at an element of class $\bs{\beta}$.

The \textit{Kostka-Foulkes polynomials} are the entries of the transition matrix
\begin{equation}
s_{\bs{\beta}}(\X)=\sum_{\bs{\alpha}}K_{\bs{\beta},\bs{\alpha}}(z)P_{\bs{\Lambda(\alpha)}}(\X;z),
\end{equation}
and the \textit{modified Kostka-Foulkes polynomials} are defined by 
\begin{equation}\label{modified-Kostka}
\tilde{K}_{\bs{\beta},\bs{\alpha}}(z)=z^{a(\bs{\Lambda})}K_{\bs{\beta},\bs{\alpha}}(z^{-1})
\end{equation}
with $\bs{\Lambda}=\bs{\Lambda}(\bs{\alpha})$.

The \textit{transformed Hall-Littlewood function} is defined by
\begin{equation}
\bar{H}_{\bs{\Lambda(\alpha)}}(\X;z)=\sum_{\bs{\beta}}K_{\bs{\beta},\bs{\alpha}}(z)s_{\bs{\beta}}(\X),
\end{equation}
and the \textit{modified Hall-Littlewood function} is defined by
\begin{equation}\label{mH-L}
\tilde{H}_{\bs{\Lambda(\alpha)}}(\X;z)=\sum_{\bs{\beta}}\tilde{K}_{\bs{\beta},\bs{\alpha}}(z)s_{\bs{\beta}}(\X).
\end{equation}

\subsection{Orthogonality}
\subsubsection{Plethysm}
Let $\mathbb{K}$ denote the base field, which will be $\mathbb{Q}$, $\mathbb{Q}(z)$ or $\mathbb{Q}(z,w)$ depending on the circumstance. Let $R$ be a $\lambda$-ring containing $\mathbb{K}$, then the ring $\SymF_R[\Xo]$ of symmetric functions over $R$ has a natural $\lambda$-ring structure $\{p_n\}_{n\in\mathbb{Z}_{>0}}$ with $p_n$ sending $p_1(\Xo)$ to $p_n(\Xo)$ and acting on the coefficients according to the $\lambda$-ring structure of $R$. If we take $R=\SymF_{\mathbb{K}}[\Xo]$, then $\SymF_R[\Xl]\cong\SymF_{\mathbb{K}}[\Xo,\Xl]$ acquires a natural $\lambda$-ring structure, and $p_n$ sends $p_1(\Xl)$ to $p_n(\Xl)$. If $\mathbb{K}$ contains the indeterminates $z$ or $w$, then $p_n$ sends them to the $n$'th power. By abuse of notation, we will write $\Xk=p_1(\Xk)$.

Let $A$ be a $\lambda$-ring containg $\mathbb{K}$ as a $\lambda$-subring. Given any elements $f_0$ and $f_1$ of $A$, there is a unique $\lambda$-ring homomorphism $\varphi_{f_0,f_1}:\SymF_{\mathbb{K}}[\Xo,\Xl]\rightarrow A$ sending $\Xo$ to $f_0$ and $\Xl$ to $f_1$. For any $F\in \SymF_{\mathbb{K}}[\Xo,\Xl]$, its image under this map is denoted by $F[f_0,f_1]$. Concretely, by expressing $F$ as a polynomial $f(p_n(\Xo),p_m(\Xl)\mid n,m\in\mathbb{Z}_{>0})$, one defines $F[f_0,f_1]:=f(p_n(f_0),p_m(f_1)\mid n,m\in\mathbb{Z}_{>0})$. This is the plethystic substitution for $F$. For example, if $F=s_{\alpha}(\Xo)$ is the usual Schur symmetric function in $\Xo$ for some partition $\alpha$, and $A=\SymF_{\mathbb{K}}[\Xo,\Xl]$, then $F[\Xl,\Xo]=s_{\alpha}(\Xl)$.

In what follows, we will often write the variables of a symmetric function in a column vector:
\begin{equation}
u(
\begin{bmatrix}
\Xo\\
\Xl
\end{bmatrix}
)=u(\X)\in\SymF[\Xo,\Xl].
\end{equation}
Then $\X^t$ is a row vector. The expression
\begin{equation}
u(
\left(
\begin{array}{cc}
a & b\\
c & d
\end{array}
\right)
\begin{bmatrix}
\Xo\\
\Xl
\end{bmatrix}
)=
u[\begin{bmatrix}
a\Xo+b\Xl\\
c\Xo+d\Xl
\end{bmatrix}],
\end{equation}
is the plethystic substitution by $f_0=a\Xo+b\Xl$ and $f_1=c\Xo+d\Xl$.

The following identity will be useful: for any partition $\alpha$,
\begin{equation}\label{s(-z)}
s_{\alpha}[-\Xo]=(-1)^{|\alpha|}s_{\alpha^{\ast}}(\Xo),
\end{equation}
where $\alpha^{\ast}$ is the dual partition of $\alpha$.

\subsubsection{}
The Hall inner product on $\SymF[\Xo]$ is defined by
$$
\langle s_{\lambda}(\Xo),s_{\mu}(\Xo)\rangle=\delta_{\lambda,\mu},
$$
for any partitions $\lambda$ and $\mu$. It is known that
$$
\langle p_{\lambda}(\Xo),p_{\mu}(\Xo)\rangle=\zz_{\lambda}\delta_{\lambda,\mu}.
$$
The \textit{plethystic exponential} is defined by 
\begin{equation}
\Exp[\Xo]=\sum_{n\in\mathbb{Z}_{\ge0}}h_n(\Xo)=\prod_i\frac{1}{1-x^{(0)}_i}.
\end{equation}
It satisfies $\Exp[\Xo+\Xl]=\Exp[\Xo]\cdot\Exp[\Xl]$. Let $\{u_{\lambda}\}_{\lambda\in\mathcal{P}}$ and $\{v_{\lambda}\}_{\lambda\in\mathcal{P}}$ be two families of symmetric functions  such that for any $n\in\mathbb{Z}_{>0}$, each of $\{u_{\lambda}\}_{\lambda\in\mathcal{P}_n}$ and $\{v_{\lambda}\}_{\lambda\in\mathcal{P}_n}$ is a basis of the homogenous degree $n$ subspace of $\SymF[\Xo]$. Then the following conditions are equivalent:
\begin{itemize}
\item $\langle u_{\lambda},v_{\mu}\rangle=\delta_{\lambda,\mu}$ for any $\lambda$, $\mu$;
\item $\sum_{\lambda}u_{\lambda}(\Xo)v_{\lambda}(\Xl)=\Exp[\Xo\Xl]$,
\end{itemize}
where the bracket means the plethystic substitution, and $\Xo\Xl$ is the family of variables $(x^{(0)}_ix^{(1)}_j)_{i,j}$.

\subsubsection{}
The Hall inner product on $\SymF_z[\Xo,\Xl]$ is defined by
\begin{equation}
\langle s_{\bs{\alpha}}(\X),s_{\bs{\beta}}(\X)\rangle=\delta_{\bs{\alpha},\bs{\beta}},
\end{equation}
for all 2-partitions $\bs{\alpha}$ and $\bs{\beta}$. By \cite[Proposition 2.5]{S01}, we have
$$
\langle p_{\bs{\alpha}}(\X),p_{\bs{\beta}}(\X)\rangle=\zz_{\bs{\alpha}}\delta_{\bs{\alpha},\bs{\beta}}.
$$
The $z$-inner product on $\SymF_z[\Xo,\Xl]$ is defined by
$$
\left\langle F[\X],G[\X]\right\rangle_{z}:=\left\langle F[\X],G[
\left(
\begin{array}{cc}
1 & -z\\
-z & 1
\end{array}
\right)^{-1}\X]\right\rangle
$$
for any $F$, $G\in\SymF_{z}[\X]$. Thus, if $u_{\bs\alpha}(\X,z)$ and $v_{\bs\alpha}(\X,z)$ are dual basis under the $z$-inner product, then 
$$
\Exp[\X^t\left(
\begin{array}{cc}
1 & -z\\
-z & 1
\end{array}
\right)\Y]=\sum_{\bs{\alpha}}u_{\bs\alpha}(\X;z)v_{\bs{\alpha}}(\Y;z).
$$
We have an explicit expression
\begingroup
\allowdisplaybreaks
\begin{align*}
\Exp[\X^t\left(
\begin{array}{cc}
1 & -z\\
-z & 1
\end{array}
\right)\Y]=&\prod_{i,j}\frac{1-z\mathbf{x}^{(1)}_i\mathbf{y}^{(0)}_j}{1-\mathbf{x}^{(0)}_i\mathbf{y}^{(0)}_j}\prod_{i,j}\frac{1-z\mathbf{x}^{(0)}_i\mathbf{y}^{(1)}_j}{1-\mathbf{x}^{(1)}_i\mathbf{y}^{(1)}_j}.
\end{align*}
\endgroup
This is exactly the generating function $\Omega(\X,\Y,t)$ in \cite[Proposition 2.5]{S01}.

\subsubsection{}
The Hall-Littlewood symmetric functions $P_{\bs{\Lambda}}(\X;z)$ and $Q_{\bs{\Lambda}'}(\X;z)$ form dual basis with respect to the $z$-inner product.
\begin{Thm}(\cite{S01} Corollary 4.6)\label{ShCor4.6}
There exists a block diagonal matrix $b_{\bs{\Lambda},\bs{\Lambda}'}(z)$, with each block corresponding to a similarity class, such that
\begingroup
\allowdisplaybreaks
\begin{align}
\Exp[\X^t\left(
\begin{array}{cc}
1 & -z\\
-z & 1
\end{array}
\right)\Y]&=\sum_{\bs{\Lambda},\bs{\Lambda}'}b_{\bs{\Lambda},\bs{\Lambda}'}(z)P_{\bs{\Lambda}}(\X;z)P_{\bs{\Lambda}'}(\Y;z),\\
\Exp[\X^t\left(
\begin{array}{cc}
1 & -z\\
-z & 1
\end{array}
\right)\Y]&=\sum_{\bs{\Lambda}}P_{\bs{\Lambda}}(\X;z)Q_{\bs{\Lambda}}(\Y;z).
\end{align}
\endgroup
\end{Thm}

\begin{Rem}\label{Q=bP}
In general, $Q_{\bs{\Lambda}}(\Y;z)$ is a linear combination of $P_{\bs{\Lambda}'}(\Y;z)$ with coefficients $b_{\bs{\Lambda},\bs{\Lambda}'}(z)$. The minimal element $\bs{\Lambda}_0=\bs{\Lambda}(\bs{\alpha}_0)$ with $\bs{\alpha}_0=(\varnothing,(1^m))$, is alone in its similarity class. If we write $b_{\bs{\Lambda}_0}(z):=b_{\bs{\Lambda}_0,\bs{\Lambda}_0}(z)$, then $$Q_{\bs{\Lambda}_0}(\X;z)=b_{\bs{\Lambda}_0}(z)P_{\bs{\Lambda}_0}(\X;z).$$ We may also write $b_{\bs{\alpha}_0}(z)=b_{\bs{\Lambda}_0}(z)$.
\end{Rem}

\begin{Lem}\label{XoXlbasechange}
We have:
\begin{itemize}
\item[(i)]
$Q_{\bs{\Lambda}}(
\begin{bmatrix}
\Xo\\
\Xl
\end{bmatrix}
,z)=
\bar{H}_{\bs{\Lambda}}(
\left(
\begin{array}{cc}
1 & -z\\
-z & 1
\end{array}
\right)
\begin{bmatrix}
\Xo\\
\Xl
\end{bmatrix}
,z)$\\for any $\bs{\Lambda}$;
\item[(ii)]
$
h_{\bs{\alpha}_0^{\ast}}(
\left(
\begin{array}{cc}
1 & -z\\
-z & 1
\end{array}
\right)^{-1}
\begin{bmatrix}
\Xo\\
\Xl
\end{bmatrix}
)=
(-1)^mz^{-a(\bs{\alpha}_0)-m}b_{\bs{\alpha}_0}(z^{-1})^{-1}\tilde{H}_{\bs{\Lambda}(\bs{\alpha}_0)}(\X;z)
$\\for $\bs\alpha_0=(\varnothing,(1^m))$.
\end{itemize}
\end{Lem}
\begin{proof}
From the definition of the transformed Hall-Littlewood functions and the orthogonality of the Schur functions, we see that 
$$
\Exp[\X^t\Y]=\sum_{\bs{\Lambda}}P_{\bs{\Lambda}}(\X;z)\bar{H}_{\bs{\Lambda}}(\Y;z),
$$
and so
$$
\Exp[\X^t\left(
\begin{array}{cc}
1 & -z\\
-z & 1
\end{array}
\right)\Y]=\sum_{\bs{\Lambda}}P_{\bs{\Lambda}}(\X;z)\bar{H}_{\bs{\Lambda}}(\left(
\begin{array}{cc}
1 & -z\\
-z & 1
\end{array}
\right)\Y;z).
$$
This shows that the right hand side of (i) gives a dual basis of $\{P_{\bs{\Lambda}}(\X;z)\}$ for the $z$-inner product, and so must coincide with $Q_{\bs{\Lambda}}(\X;z)$.

To prove (b), we calculate
\begingroup
\allowdisplaybreaks
\begin{align*}
&(-1)^mz^{-a(\bs{\alpha}_0)-m}b_{\bs{\alpha}_0}(z^{-1})^{-1}\tilde{H}_{\bs{\Lambda}_0}(\X;z)\\
%%% Circle 1
\stackrel{\circled{1}}{=}&(-z)^{-m}b_{\bs{\alpha}_0}(z^{-1})^{-1}Q_{\bs{\Lambda}_0}(
\left(
\begin{array}{cc}
1 & -z^{-1}\\
-z^{-1} & 1
\end{array}
\right)^{-1}
\begin{bmatrix}
\Xo\\
\Xl
\end{bmatrix}
,z^{-1})\\%%% Circle 2
\stackrel{\circled{2}}{=}&(-z)^{-m}P_{\bs{\Lambda}_0}(
\left(
\begin{array}{cc}
1 & -z^{-1}\\
-z^{-1} & 1
\end{array}
\right)^{-1}
\begin{bmatrix}
\Xo\\
\Xl
\end{bmatrix}
,z^{-1})\\%%% Circle 3
\stackrel{\circled{3}}{=}&(-z)^{-m}s_{\bs{\alpha}_0}(
\left(
\begin{array}{cc}
1 & -z^{-1}\\
-z^{-1} & 1
\end{array}
\right)^{-1}
\begin{bmatrix}
\Xo\\
\Xl
\end{bmatrix}
)\\%%%
=&(-z)^{-m}s_{\bs{\alpha}_0}(
\frac{-z}{1-z^2}
\left(
\begin{array}{cc}
z & 1\\
1 & z
\end{array}
\right)
\begin{bmatrix}
\Xo\\
\Xl
\end{bmatrix}
)\\%%% Circle 4
\stackrel{\circled{4}}{=}&(-z)^{-m}s_{(1^m)}(
\frac{-z}{1-z^2}
(\Xo+z\Xl)
)\\%%% Circle 5
\stackrel{\circled{5}}{=}&s_{(m)}(\frac{\Xo+z\Xl}{1-z^2})\\
=&h_{m}(\frac{\Xo+z\Xl}{1-z^2})\\
=&h_{((m),\varnothing)}(
\frac{1}{1-z^2}
\left(
\begin{array}{cc}
1 & z\\
z & 1
\end{array}
\right)
\begin{bmatrix}
\Xo\\
\Xl
\end{bmatrix}
)\\%%%
=&h_{((m),\varnothing)}(
\left(
\begin{array}{cc}
1 & -z\\
-z & 1
\end{array}
\right)^{-1}
\begin{bmatrix}
\Xo\\
\Xl
\end{bmatrix}
).
\end{align*}
\endgroup
Equality \circled{1} follows from the equality $\tilde{H}_{\bs{\Lambda}_0}(\X;z)=z^{a(\bs{\alpha}_0)}\bar{H}_{\bs{\Lambda}_0}(\X;z^{-1})$ and part (i) of the lemma. Equality \circled{2} uses Remark \ref{Q=bP}. In equality \circled{3} we have used the second statement about  $P_{\bs{\Lambda}}$ in \cite[Theorem 4.4]{S01}. In equality \circled{4}, we denote by $s_{(1^m)}$ the usual Schur function associated to partitions. In equality \circled{5}, we use (\ref{s(-z)}).
\end{proof}

\subsection{Digression - Green Functions}\hfill

For $\bs{\Lambda}_0=\bs{\Lambda}(\bs{\alpha}_0)$, define the \textit{Green function} by
\begin{equation}\label{Green-Function-Comb}
\mathcal{Q}^{\bs{\Lambda}_0}_{\bs{\beta}}(z)=\sum_{\bs{\gamma}}\chi^{\bs{\gamma}}_{\bs{\beta}}\tilde{K}_{\bs{\gamma},\bs{\alpha}_0}(z),
\end{equation}
where $\chi^{\bs{\gamma}}_{\bs{\beta}}$ is the value of the character $\chi^{\bs{\gamma}}\in\Irr(\mathfrak{W}_m)$ on the conjugacy class corresponding to $\bs\beta$. Let $H$ be $\Sp_{2m}$ or $\SO_{2m+1}$ over $\bar{\mathbb{F}}_q$, equipped with a split Frobenius $F$. As in \S \ref{DL-Induction}, we denote by $W_H$ the Weyl group of $H$ defined by a fixed maximal torus $T_H$, and we have an $F$-stable maximal torus $T_w$ for each $w\in W_H$. The purpose of this subsection is to explain the equality
\begin{equation}\label{eq-Q=Q}
Q_{T_w}^H(1)=\mathcal{Q}^{\bs{\Lambda}_0}_{\bs{\beta}}(q),
\end{equation} 
where the left hand side is the Green functions defined by (\ref{Green-Function-DL}) evaluated at the identity, $\bs\beta$ is the 2-partition corresponding to the conjugacy class of $w$, and $z$ is specialised to the prime power $q$. 

This equality is known, but not explicitly mentioned in \cite{S01}.

\subsubsection{}
Let $b(z)$ be the matrix $(b_{\bs{\Lambda},\bs{\Lambda}'}(z))$ as in Theorem \ref{ShCor4.6}. Denote by $H(q)$ the finite group $H^F$. Define $$\tilde{\Pi}=\frac{|H(q)|}{q^{m}}T^{-1}b(q^{-1})^{-1}T^{-1},$$ where $T$ is the diagonal matrix with entries $q^{a(\bs{\alpha})}$. Let $\tilde{K}(z)$ be the matrix of modified Kostka-Foulkes polynomials (see (\ref{modified-Kostka})). By \cite[Theorem 5.4]{S01}, the matrices $P:=\tilde{K}(q)$ satisfies the following matrix equation: 
\begin{equation}\label{POmega2}
P\tilde{\Pi} P^t=\Xi.
\end{equation}
The matrix $\Xi$ is determined by the inner products of Green functions of classical groups, which are known. The matrix $\tilde{\Pi}$ is a block diagonal matrix with nonsingular blocks. The matrix $P$ is a lower triangular block matrix with each diagonal block being the scalar $q^{a(\bs{\alpha})}$, with $\bs\alpha\in\mathcal{P}^2(m)$. The polynomials $\tilde{K}_{\bs{\beta},\bs{\alpha}}(q)$ are uniquely determined as the entries of $P$ satisfying the above equation.
\begin{Rem}
\cite[Theorem 5.4]{S01} states that these matrices satisfy \cite[(1.5.2)]{S01}. For Coxeter groups, this is equivalent to \cite[(1.4.2)]{S01}, since the complex conjugate of an irreducible character coincides with itself.
\end{Rem}

\subsubsection{}\label{Q=Q}
Let $\chi\in\Irr(W_H)$ and let $(C,\phi)$ be the pair corresponding to $\chi$ under the Springer correspondence (\ref{Spr-Corr}). Put $$Q_{\chi}:=\frac{1}{|W_H|}\sum_{w\in W_H}\chi(w)Q_{T_w}^H.$$ Then $Q_{\chi}$ coincides with the characteristic function of the intersection cohomology complex $\IC(\bar{C},\mathcal{L}_{\phi})$, where $\mathcal{L}_{\phi}$ is the simple $H$-equivariant local system on $C$ corresponding to $\phi$ (see \cite[(5.2)]{S87}). Let $I_0$ be a finite set that is in bijection with the image of (\ref{Spr-Corr}), equipped with the total order induced from symbols. If $\chi$ corresponds to $i\in I_0$, then we may write $Q_i=Q_{\chi}$. For any $(C,\phi)=i\in I_0$, define the $H^F$-invariant function $\psi_i$ on $H^F_u$ as the characteristic function of $\mathcal{L}_{\phi}$ (for a given isomorphism $F^{\ast}\mathcal{L}_{\phi}\isom\mathcal{L}_{\phi}$). By \cite[Theorem 5.3]{S87}, there exists a square matrix $P$ with entries in $\bar{\mathbb{Q}}_{\ell}$ satisfying
\begin{equation}\label{Q=Ppsi}
Q_i=\sum_{j\in I_0}P_{ij}\psi_j.
\end{equation}
Beware that this matrix $P$ is the transpose of the $P$ in \cite{S87}. Since $Q_i$ is supported on the closure of $C$, the matrix $P$ is block lower triangular. By taking the inner products of functions on the unipotent part $H^F_u$, we get the following matrix equation for $P$: $$P\bs\Lambda P^t=\bs\Pi$$ as \cite[(5.6)]{S87}. This equation (i.e. entries of $\bs\Lambda$ and $\bs\Pi$) only differs by a scalar from the equation (\ref{POmega2}). The solution is unique, and so the entries of this $P$ are also equal to the modified Kostka-Foulkes polynomials $\tilde{K}_{\bs{\beta},\bs{\alpha}}(q)$. Finally, observe that if $C=\{1\}$, then $\phi$ must be the trivial character, thus \begin{equation}\label{Q(1)=P}
Q_i(1)=P_{ii_0},
\end{equation} where $i_0$ is the minimal element corresponding to $(C=\{1\},\phi=\mathbf{1})$. Combined with the definition of $Q_{\chi}$, this implies (\ref{eq-Q=Q}).
\begin{Rem}
If we want to compute $Q_{T_w}^H(u)$ for some unipotent element $u\ne 1$, then (\ref{Q(1)=P}) should be replaced by a sum as in (\ref{Q=Ppsi}). There is no obvious way to incorporate the functions $\psi_j$ into symmetric functions.
\end{Rem}

\subsubsection{}\label{b_lambda}
From \cite[(1.1.1),(1.4.1)]{S01} one sees that $\tilde{\Pi}_{\bs{\Lambda}_0,\bs{\Lambda}_0}=1$, i.e.
\begin{equation}\label{mathbbG}
\frac{1}{|H(q)|}=q^{-2a(\bs{\alpha}_0)-n}b_{\bs{\alpha}_0}(q^{-1})^{-1}.
\end{equation}
For $\GL_n(q)$, it is known that (see \cite[(2.6), (2.7)]{Mac})
\begin{equation}\label{1/|GLn|}
\frac{1}{|\GL_n(q)|}=q^{-2n((1^n))-n}b_{(1^n)}(q^{-1})^{-1},
\end{equation}
where for any partition $\lambda$, the function $b_{\lambda}(z)$ is defined by $b_{\lambda}(z)P_{\lambda}(\X,z)=Q_{\lambda}(\X,z)$, i.e. the difference between the two Hall-Littlewood symmetric functions associated to the partition $\lambda$.

\subsection{Symmetric Functions Associated to Types}\label{SymF-Types}
\subsubsection{}
For any symmetric function $u(\Xo)\in\SymF[\Xo]$, we define
$$
u(
\begin{bmatrix}
\Xo\\
\Xl
\end{bmatrix}
):=u[\Xo+\Xl]\in\SymF[\Xo,\Xl].
$$
We have $$u[
\left(
\begin{array}{cc}
a & b\\
c & d
\end{array}
\right)
\begin{bmatrix}
\Xo\\
\Xl
\end{bmatrix}
]=u[(a+c)\Xo+(b+d)\Xl].$$ In particular 
\begingroup
\allowdisplaybreaks
\begin{align}
\nonumber
u(
\left(
\begin{array}{cc}
1 & -z\\
-z & 1
\end{array}
\right)
\begin{bmatrix}
\Xo\\
\Xl
\end{bmatrix}
)&=u[(1-z)(\Xo+\Xl)],\\
\nonumber
u(
\left(
\begin{array}{cc}
1 & -z\\
-z & 1
\end{array}
\right)^{-1}
\begin{bmatrix}
\Xo\\
\Xl
\end{bmatrix}
)&=u[\frac{\Xo+\Xl}{1-z}].
\end{align}
\endgroup 
We will write 
\begin{equation}\label{P}
\mathbf{P}=\left(
\begin{array}{cc}
1 & -z\\
-z & 1
\end{array}
\right)^{-1}
\end{equation}
in what follows.

\subsubsection{}
For any $\bs\omega=\bs\omega_+\bs\omega_-(\omega_i)\in\mathfrak{T}$, we define the Schur symmetric function associated to $\bs\omega$ as 
$$
s_{\bs\omega}(\X):=s_{\bs\omega_+}(\X)s_{\bs\omega_-}(\X)\prod_is_{\omega_i}(\X).
$$
Regarding $\mathfrak{T}_{\ast}$ as a subset of $\mathfrak{T}$, this definition also makes sense for $\bs\omega\in\mathfrak{T}_{\ast}$. Monomial symmetric functions, complete symmetric functions and power sum symmetric functions can be similarly defined. Define 
\begin{equation}
P_{\bs{\Lambda}(\bs\omega)}(\X,z):=P_{\bs{\Lambda}(\bs\omega_+)}(\X,z)P_{\bs{\Lambda}(\bs\omega_-)}(\X,z)\prod_iP_{\omega_i}(\X,z),
\end{equation}
where $P_{\omega_i}(\X,z)$ is the Hall-littlewood functions defined in \cite[Chapter III]{Mac}, and similarly for $Q_{\bs{\Lambda}(\omega)}(\X,z)$. Note that for any partition $\lambda$, we have $$p_{\lambda}[\Xo+\Xl]=p_{(\lambda,\varnothing)}(\X).$$ This implies in particular that $p_{\bs\omega}(\X)=p_{[\bs\omega]}(\X)$.

\subsubsection{}
For any $\bs\omega\in\mathfrak{T}$, define $\zz_{\bs\omega}:=\zz_{\bs\omega_+}\zz_{\bs\omega_-}\prod_i\zz_{\omega_i}$. Define $$a(\bs{\Lambda}(\bs\omega))=a(\bs{\Lambda}(\bs\omega_+))+a(\bs{\Lambda}(\bs\omega_-))+\sum_in(\omega_i).$$We may write $a(\bs\omega):=a(\bs{\Lambda}(\bs\omega))$. For any $\bs\lambda=\lambda_+\lambda_-(\lambda_i)\in\mathfrak{T}^{\circ}$, define $$n(\bs\lambda)=n(\lambda_+)+n(\lambda_-)+\sum_in(\lambda_i).$$ In particular, for $\bs\alpha\in\mathfrak{T}_{\ast}$, we have $a(\bs\alpha)=n(\bs\alpha)$. For $\bs\omega=(\varnothing,(1)^{m_+})(\varnothing,(1^{m_-}))((1^{m_i}))_i\in\mathfrak{T}_s$, define (see Remark \ref{Q=bP} and \S \ref{b_lambda}) $$b_{\bs\omega}(z)=b_{\bs\omega_+}(z)b_{\bs\omega_-}(z)\prod_ib_{\omega_i}(z).$$

\subsubsection{}
For any $\bs\alpha$, $\bs\beta\in\tilde{\mathfrak{T}}$, if $\bs\alpha\thickapprox\bs\beta$, define $\chi^{\bs\alpha}_{\bs\beta}:=\chi^{\bs\alpha_+}_{\bs\beta_+}\chi^{\bs\alpha_-}_{\bs\beta_-}\prod_i\chi^{\alpha_i}_{\beta_i}$; otherwise we put $\chi^{\bs\alpha}_{\bs\beta}=0$. With these definitions we have
\begin{equation}
\begin{split}
p_{\bs\beta}(\X)&=\sum_{\bs\alpha}\chi^{\bs\alpha}_{\bs\beta}~s_{\bs\alpha}(\X);\\
s_{\bs\lambda}(\X)&=\sum_{\bs\tau}\frac{1}{\zz_{\bs\tau}}\chi^{\bs\lambda}_{\bs\tau}~p_{\bs\tau}(\X).
\end{split}
\end{equation}
\begin{Rem}
One needs to be careful with the summations in these expressions. For example in the first equation, $\bs\alpha$ runs over the ordered types $\tilde{\mathfrak{T}}$. Although the symmetric functions are independent of the ordering in the types. 
\end{Rem}
 
Define $$K_{\bs\beta,\bs\alpha}(z):=K_{\bs\beta_+,\bs\alpha_+}(z)K_{\bs\beta_-,\bs\alpha_-}(z)\prod_iK_{\beta_i,\alpha_i}(z),\text{ if $\bs\alpha\thickapprox\bs\beta$},$$ and put $K_{\bs\beta,\bs\alpha}(z)=0$ otherwise. We then have
\begin{equation}
s_{\bs\beta}(\X)=\sum_{\bs\alpha}K_{\bs\beta,\bs\alpha}(z)P_{\bs{\Lambda}(\bs\alpha)}(\X;z).
\end{equation}
Define $\tilde{K}_{\bs\beta,\bs\alpha}(z)=z^{a(\bs\alpha)}K_{\bs\beta,\bs\alpha}(z^{-1})$.

With the obvious definitions, we have the following identities for types:
\begin{eqnarray}
\mathcal{Q}^{\bs{\Lambda}(\bs\alpha)}_{\bs\beta}(z)&=&\sum_{\bs\gamma}\chi^{\bs\gamma}_{\bs\beta}\tilde{K}_{\bs\gamma,\bs\alpha}(z),\\%%
\bar{H}_{\bs{\Lambda}(\bs\alpha)}(\X;z)&=&\sum_{\bs\beta}K_{\bs\beta,\bs\alpha}(z)s_{\bs\beta}(\X),\\%%
\tilde{H}_{\bs{\Lambda}(\bs\alpha)}(\X;z)&=&\sum_{\bs\beta}\tilde{K}_{\bs\beta,\bs\alpha}(z)s_{\bs\beta}(\X).
\end{eqnarray}
In the first equality, $\bs\alpha$ lies in $\tilde{\mathfrak{T}}_s$.

\subsubsection{}
Lemma \ref{XoXlbasechange} (ii) can also be stated for types.
\begin{Cor}\label{XoXlBCType}
Let $\bs\beta\in\mathfrak{T}_s$. Then the following identity holds:
$$
h_{\bs\beta^{\ast}}(\mathbf{P}\X)=
(-1)^{|\bs\beta|}z^{-a(\bs\beta)-|\bs\beta|}b_{\bs\beta}(z^{-1})^{-1}\tilde{H}_{\bs{\Lambda}(\bs\beta)}(\X;z).
$$
\end{Cor}
\begin{proof}
This follows from Lemma \ref{XoXlbasechange} and \cite[Lemma 2.3.6]{HLR}.
\end{proof}

\begin{Lem}\label{<s,H>}
Let $\bs\alpha\in\tilde{\mathfrak{T}}$ and $\bs\beta\in\tilde{\mathfrak{T}}_s$. Then 
$$
\langle s_{\bs\alpha}(\X),\tilde{H}_{\bs{\Lambda}(\bs\beta)}(\X,z)\rangle=\sum_{\bs\tau\in\tilde{\mathfrak{T}}}\frac{\zz_{[\bs\tau]}\chi^{\bs\alpha}_{\bs\tau}}{\zz_{\bs\tau}}\sum_{\substack{\bs\nu\in\tilde{\mathfrak{T}}\\ [\bs\nu]=[\bs\tau] }}\frac{Q^{\bs{\Lambda}(\bs\beta)}_{\bs\nu}(z)}{\zz_{\bs\nu}}.$$
\end{Lem}
\begin{Rem}
The inner product only depends on the corresponding unordered types.
\end{Rem}
\begin{proof}
We use the transition matrices $$s_{\bs\alpha}(\X)=\sum_{\bs\tau}\chi^{\bs\alpha}_{\bs\tau}\frac{p_{\bs\tau}(\X)}{\zz_{\bs\tau}},$$ and, 
\begin{equation*}
\begin{split}
\tilde{H}_{\bs{\Lambda}(\bs\beta)}(\X,z)
=&\sum_{\bs\tau}\tilde{K}_{\bs\tau,\bs\beta}(z)s_{\bs\tau}(\X)\\
=&\sum_{\bs\tau}\sum_{\bs\nu}\chi^{\bs\tau}_{\bs\nu}\tilde{K}_{\bs\tau,\bs\beta}(z)\frac{p_{\bs\nu}(\X)}{\zz_{\bs\nu}}\\
=&\sum_{\bs\nu}Q^{\bs{\Lambda}(\bs\beta)}_{\bs\nu}(z)\frac{p_{\bs\nu}(\X)}{\zz_{\bs\nu}}.
\end{split}
\end{equation*}
The result follows by taking inner product, noting that $$\langle p_{\bs\tau}(\X),p_{\bs\nu}(\X)\rangle=\langle p_{[\bs\tau]}(\X),p_{[\bs\nu]}(\X)\rangle=\delta_{[\bs\tau],[\bs\nu]}\zz_{[\bs\tau]}.$$
\end{proof}

%%%%%%%%%%%%%%%%%%%%%%%
\section{Computation of E-polynomials}\label{Section-E-poly}  

\subsection{A combinatorial formula for irreducible characters}
\subsubsection{}
Let us first give a combinatorial expression for the value of an irreducible character at a semi-simple element. Let $\chi$ be a $\sigma$-stable irreducible character of $G(q)$, which is induced from some $\sigma$-stable standard Levi subgroup $M=M_0\times M_1$ as in Theorem \ref{sig-st-chi}. Denote its extension to $G(q)\lb\sigma\rb$ by $\tilde{\chi}$. If $\tilde{\chi}$ is uniform, then it can be written as a linear combination of generalised Deligne-Lusztig characters:
\begin{equation}\label{cptechi1}
\tilde{\chi}|_{G^F\sigma}=|\mathfrak{W}_+\times\mathfrak{W}_-\times W_1^{\sigma}|^{-1}\!\!\!\!\!\!\!\!\sum_{\substack{\tau=(w_+,w_-,w_1)\\\in\mathfrak{W}_+\times\mathfrak{W}_-\times W_1^{\sigma}}}\!\!\!\!\!\!\!\!\varphi_+(w_+)\varphi_-(w_-)\tilde{\varphi}(w_1F)R^{G\sigma}_{T_{w_1,\mathbf{w}}\sigma}\tilde{\theta}_{w_1,\mathbf{w}}
\end{equation}
by Theorem \ref{thm-decIrr}. We will write $T_{\tau}=T_{w_1,\mathbf{w}}$ and $\theta_{\tau}=\theta_{w_1,\mathbf{w}}$. Suppose that $\chi$ is of type $\epsilon\bs\alpha$. As is explained in \S \ref{relevance-types}, the irreducible character $(\varphi_+,\varphi_-,\varphi)$ of $\mathfrak{W}_+\times\mathfrak{W}_-\times W_1^{\sigma}$ corresponds to a unique ordered type, the unordered version of which is exactly $\bs\alpha$, with the ordering provided by a specific choice of isomorphism $M_1\cong\prod_i(\GL_{n_i}\times\GL_{n_i})$. In this subsection, we fix such an isomorphism and so $\bs\alpha$ will be regarded as an element of $\tilde{\mathfrak{T}}$. The sum over $\tau$ only depends on its conjugacy class, which we denote by $\mathcal{O}(\tau)$. Note that the conjugacy classes of $\mathfrak{W}_+\times\mathfrak{W}_-\times W_1^{\sigma}$ are parametrised by a subset of $\tilde{\mathfrak{T}}$. The type of $\mathcal{O}(\tau)$ will be denoted by $\bs\tau$. We have 
\begin{equation}
\frac{|\mathcal{O}(\tau)|}{|\mathfrak{W}_+\times\mathfrak{W}_-\times W_1^{\sigma}|}=\frac{1}{\zz_{\bs\tau}},\quad\varphi_+(w_+)\varphi_-(w_-)\tilde{\varphi}(w_1F)=\chi^{\bs\alpha}_{\bs\tau}.
\end{equation}
We will therefore replace $\sum_{\tau}(-)\text{ by }\sum_{\mathcal{O}(\tau)}|\mathcal{O}(\tau)|\cdot(-)$. However, for each $\bs\tau$, we will choose a $\tau$ representing it so that $T_{\tau}$ and $\theta_{\tau}$ have definite meanings.

If $s\sigma\in T^F\sigma$ is a semi-simple element, then the character formula (\ref{eq-char-formula}) reads
\begin{equation}\label{cptechi2}
R^{G\sigma}_{T_{\tau}\sigma}\tilde{\theta}_{\tau}(s\sigma)=\frac{|(T_{\tau}^{\sigma})^{\circ F}|}{|T_{\tau}^F|\cdot|C_G(s\sigma)^{\circ F}|}\sum_{\{h\in G^{F}\mid hs\sigma h^{-1}\in T_{\tau}\sigma\}}Q^{C_G(s\sigma)^{\circ}}_{C_{h^{-1}T_{\tau}h}(s\sigma)^{\circ}}(1)\tilde{\theta}_{\tau}(hs\sigma h^{-1}).
\end{equation}
Denote by $\bs\beta\in\tilde{\mathfrak{T}}_s$ the type of the conjugacy class of $s\sigma$, which we assume to satisfy \ref{CCL}. Recall that the $C_{G}(s\sigma)^{\circ F}$-conjugacy classes of the $F$-stable maximal tori of $C_{G}(s\sigma)^{\circ}$ are also parametrised by a subset of $\tilde{\mathfrak{T}}$. If $C_{h^{-1}T_{\tau}h}(s\sigma)^{\circ}$ is of type $\bs\nu\in\tilde{\mathfrak{T}}$, we have $\mathcal{Q}^{\bs{\Lambda}(\bs\beta)}_{\bs\nu}(q)=Q^{C_G(s\sigma)^{\circ}}_{C_{h^{-1}T_{\tau}h}(s\sigma)^{\circ}}(1)$, according to (\ref{eq-Q=Q}) together with the analogous known identity for Green functions of $\GL_n(q)$. 

Let $A_{\tau}^F$ and $B_{\tau}$ be as in  \ref{A^F->B}. Combining (\ref{cptechi1}) and (\ref{cptechi2}) gives
\begingroup
\allowdisplaybreaks
\begin{align}
\nonumber
\tilde{\chi}(s\sigma)=&\sum_{\bs\tau}\sum_{h\in A^F_{\tau}}\frac{|(T^{\sigma}_{\tau})^{\circ F}|}{|T_{\tau}^F|\cdot|C_G(s\sigma)^{\circ F}|}\cdot\frac{\chi^{\bs\alpha}_{\bs\tau}Q^{C_G(s\sigma)^{\circ}}_{C_{h^{-1}T_{\tau}h}(s\sigma)^{\circ}}(1)}{\zz_{\bs\tau}}\tilde{\theta}_{\tau}(hs\sigma h^{-1})\\
%%%
\label{tildechi(ssigma)}
\stackrel{\circled{1}}{=}&\sum_{\bs\tau}\sum_{\{\bs\nu\in B_{\tau}\}}\frac{|(T^{\sigma}_{\tau})^{\circ F}|}{|T_{\tau}^F|\cdot|C_G(s\sigma)^{\circ F}|}\cdot\frac{\chi^{\bs\alpha}_{\bs\tau}\mathcal{Q}^{\bs{\Lambda}(\bs\beta)}_{\bs\nu}(q)}{\zz_{\bs\tau}}\sum_{h\in A^F_{\tau,\bs\nu}}\tilde{\theta}_{\tau}(hs\sigma h^{-1}).
\end{align}
\endgroup
Equality \circled{1} uses the surjective map $A^F_{\tau}\rightarrow B_{\tau}$. Note that $\theta_{\tau}$ depends on $\epsilon$.

\subsubsection{}
Now we work in the context of \S \ref{Irr-reg-sig}. We will write $\Irr^{\sigma}_{\reg}=\Irr_{\reg}^{\sigma}(M_1^F)$ (see Definition \ref{reg-M_1}). For each $j$, let $\bs\beta_j$ be the type of the semi-simple conjugacy class $C_j$. Recall that $\Irr^{\sigma}_{\epsilon\bs\omega}$ is the set of $\sigma$-stable irreducible characters of type $\epsilon\bs\omega$.
\begin{Lem}
We have
$$\sum_{\chi\in\Irr^{\sigma}_{\epsilon\bs\omega}}\prod^{2k}_{j=1}\tilde{\chi}(C_{j})=\Delta_{\epsilon,\sgn\mathcal{C}}\frac{K(\bs\omega_{\ast})}{N(\bs\omega_{\ast})}\prod^{2k}_{j=1}\langle s_{\bs\omega}(\X_j),\tilde{H}_{\bs{\Lambda}(\bs{\beta}_j)}(\X_j,q)\rangle,$$ where for each $j$, $\X_j=(\mathbf{x}^{(0)}_j,\mathbf{x}^{(1)}_j)$ is an independent family of variables.
\end{Lem}
\begin{proof}
We calculate 
\begingroup
\allowdisplaybreaks
\begin{align}
\nonumber
%%%
&\sum_{\chi\in\Irr^{\sigma}_{\epsilon\bs\omega}}\prod^{2k}_{j=1}\tilde{\chi}(C_{j})\\
\nonumber
%%%
\stackrel{\circled{1}}{=}&\frac{1}{2^{l(\bs\omega_{\ast})}N(\bs\omega_{\ast})}\sum_{\bs\tau_1,\ldots,\bs\tau_{2k}}\sum_{\substack{(\bs\nu_1,\ldots,\bs\nu_{2k})\\\in\prod_jB_{\tau_j}}}\prod_{j=1}^{2k}\frac{|(T^{\sigma}_{\tau_j})^{\circ F}|}{|T_{\tau_j}^F|\cdot|C_G(s_j\sigma)^{\circ F}|}\\
\nonumber
&\cdot\frac{\chi^{\bs\omega}_{\bs\tau_j}\mathcal{Q}^{\bs{\Lambda}(\bs\beta_j)}_{\bs\nu_j}(q)}{\zz_{\bs\tau_j}}\sum_{\substack{(h_1,\ldots,h_{2k})\\\in\prod_jA^F_{\bs\tau_j,\bs\nu_j}}}\sum_{\theta\in\Irr^{\sigma}_{\reg}}\prod_{j=1}^{2k}\tilde{\theta}_j(h_js_j\sigma h_j^{-1})\\
\nonumber
%%%
\stackrel{\circled{2}}{=}&\frac{1}{2^{l(\bs\omega_{\ast})}N(\bs\omega_{\ast})}\sum_{\bs\tau_1,\ldots,\bs\tau_{2k}}\sum_{\substack{(\bs\nu_1,\ldots,\bs\nu_{2k})\\\in\prod_jB_{\tau_j}}}\prod_{j=1}^{2k}\frac{|(T^{\sigma}_{\tau_j})^{\circ F}|}{|T_{\tau_j}^F|\cdot|C_G(s_j\sigma)^{\circ F}|}\\
\nonumber
&\cdot\frac{\chi^{\bs\omega}_{\bs\tau_j}\mathcal{Q}^{\bs{\Lambda}(\bs\beta_j)}_{\bs\nu_j}(q)}{\zz_{\bs\tau_j}}\sum_{\substack{(h_1,\ldots,h_{2k})\\\in\prod_jA^F_{\bs\tau_j,\bs\nu_j}}}(-1)^{l(\bs\omega_{\ast})}2^{l(\bs\omega_{\ast})}l(\bs\omega_{\ast})!\Delta_{\epsilon,\sgn\mathcal{C}}\\
\nonumber
%%%
\stackrel{\circled{3}}{=}&\Delta_{\epsilon,\sgn\mathcal{C}}\frac{K(\bs\omega_{\ast})}{N(\bs\omega_{\ast})}\prod^{2k}_{j=1}\sum_{\bs\tau_j}\frac{\zz_{[\bs\tau_j]}\chi^{\bs\omega}_{\bs\tau_j}}{\zz_{\bs\tau_j}}\sum_{\{\bs\nu_j\mid[\bs\nu_j]=[\bs\tau_j]\}}\frac{\mathcal{Q}^{\bs{\Lambda}(\bs\beta_j)}_{\bs\nu_j}(q)}{\zz_{\bs\nu_j}}\\
%%%
\stackrel{\circled{4}}{=}&\Delta_{\epsilon,\sgn\mathcal{C}}\frac{K(\bs\omega_{\ast})}{N(\bs\omega_{\ast})}\prod^{2k}_{j=1}\langle s_{\bs\omega}(\X_j),\tilde{H}_{\bs{\Lambda}(\bs{\beta}_j)}(\X_j,q)\rangle,
\end{align}
\endgroup
In equality \circled{1}, we have applied (\ref{tildechi(ssigma)}) to each $\tilde{\chi}(C_j)$. The factor $2^{l(\bs\omega_{\ast})}N(\bs\omega_{\ast})$ comes from the surjective map (\ref{reg->eomega}). In equality \circled{2}, we have used Lemma \ref{sumofreg}. We have \circled{3} because the summation over $h_j$ is independent of $h_j$, which produces a factor $$\frac{|(T^{\sigma}_{\tau_j})^{\circ F}|}{|T_{\tau_j}^F|\cdot|C_G(s_j\sigma)^{\circ F}|}\zz_{[\bs\tau_j]}\zz_{\bs\nu_j}^{-1}$$by Proposition \ref{Prop-A^F->B}, if $A^F_{\bs\tau_j,\bs\nu_j}$ is non empty. If some $A^F_{\bs\tau_j,\bs\nu_j}$ is empty, then by Lemma \ref{comb-Btau}, we must have $[\bs\nu_j]\ne[\bs\tau_j]$. Therefore we impose $[\bs\nu_j]=[\bs\tau_j]$ in the summation. Equality \circled{4} uses Lemma \ref{<s,H>}.
\end{proof}

\subsection{Some notations}
\subsubsection{}
Recall the Hook polynomial $H_{\lambda}(z)$ defined for any partition $\lambda$:
\begin{equation}\label{eq-hook}
H_{\lambda}(z):=\prod_{x\in\lambda}(1-z^{h(x)}),
\end{equation}
where $\lambda$ is regarded as a Young diagram and $x$ runs over the boxes in the diagram, and $h(x)$ is the hook length. If we denote by $\lambda^{\ast}$ the dual partition of $\lambda$, then 
\begin{equation}\label{eq-hook-lgth}
\sum_{x\in\lambda}h(x)=|\lambda|+n(\lambda)+n(\lambda^{\ast}).
\end{equation}
For any $\bs\lambda=\lambda_+\lambda_-(\lambda_i)\in\mathfrak{T}^{\circ}$, put $H_{\bs\lambda}(z):=H_{\lambda_+}(z)H_{\lambda_-}(z)\prod_iH_{\lambda_i}(z)$. Regarding $\mathfrak{T}_{\ast}$ as a subset of $\mathfrak{T}^{\circ}$, $H_{\bs\lambda}(z)$ also makes sense for $\bs\lambda\in\mathfrak{T}_{\ast}$. Then for $\chi\in\Irr_{\bs\lambda}^{\sigma}$, we have (\cite[Chapter IV, \S 6 (6.7)]{Mac}):
\begin{equation}\label{eq-Mac6.7}
\frac{|\GL_n(q)|}{\chi_{\bs\lambda}(1)}=(-1)^nq^{\frac{1}{2}n(n-1)-n(\bs\lambda)}H_{\bs\lambda}(q).
\end{equation}
\begin{Rem}
The equation (\cite[Chapter IV, \S 6 (6.7)]{Mac}) literally reads $$\frac{|\GL_n(q)|}{\chi_{\bs\lambda}(1)}=(-1)^nq^{\frac{1}{2}n(n-1)-n(\bs\lambda^{\ast})}H_{\bs\lambda}(q),$$with $\bs\lambda^{\ast}=\lambda_+^{\ast}\lambda_-^{\ast}(\lambda_i^{\ast})$. The difference from (\ref{eq-Mac6.7}) is due to a different parametrisation of the unipotent characters of $\GL_n(q)$ by partitions.
\end{Rem}

\subsubsection{}\label{dim-Ch}
According to \cite[Theorem 4.6]{Shu1}, the dimension of the character variety is given by:
\begingroup
\allowdisplaybreaks
\begin{align*}
d:=&(2g-2)\dim G+\sum^{2k}_{j=1}\dim C_j\\
=&(2g-2)\dim G+\sum^{2k}_{j=1}(\dim G-\dim C_G(s_j\sigma))\\
\stackrel{\circled{1}}{=}&(2g-2)n^2+2kn^2-\sum^{2k}_{j=1}(2a(\bs\beta_j)+N)\\
=&n^2(2g+2k-2)-2kN-\sum^{2k}_{j=1}2a(\bs\beta_j).
\end{align*}
\endgroup
In \circled{1}, we have used (\ref{a=dimBu}) and the fact that if $H$ is a connected reductive group over $\mathbb{C}$ and $\mathcal{B}$ is the flag variety of $H$, then $\dim H=2\dim\mathcal{B}+\rk H$.

\subsubsection{}\label{Omega(q)}
Define the infinite series:
\begingroup
\allowdisplaybreaks
\begin{align*}
\Omega_{1}(z):=&\sum_{\bs{\alpha}\in\mathcal{P}^2}z^{k|\bs\alpha|}z^{(1-g-k)|\{\bs{\alpha}\}_1|}(H_{\{\bs{\alpha}\}_1}(z)z^{-n(\{\bs{\alpha}\}_1)})^{2g+2k-2}\prod^{2k}_{j=1}s_{\bs{\alpha}}[\left(
\begin{array}{cc}
1 & -z\\
-z & 1
\end{array}
\right)^{-1}
\begin{bmatrix}
\mathbf{x}^{(0)}_j\\
\mathbf{x}^{(1)}_j
\end{bmatrix}],\\
%%%
\Omega_{0}(z):=&\sum_{\bs{\alpha}\in\mathcal{P}^2}z^{k|\bs\alpha|}z^{(1-g-k)|\{\bs{\alpha}\}_0|}(H_{\{\bs{\alpha}\}_0}(z)z^{-n(\{\bs{\alpha}\}_0)})^{2g+2k-2}\prod^{2k}_{j=1}s_{\bs{\alpha}}[\left(
\begin{array}{cc}
1 & -z\\
-z & 1
\end{array}
\right)^{-1}
\begin{bmatrix}
\mathbf{x}^{(0)}_j\\
\mathbf{x}^{(1)}_j
\end{bmatrix}],\\
%%%
\Omega_{\ast}(z):
=&\sum_{\alpha\in\mathcal{P}}z^{(2-2g-k)|\alpha|}(H_{\alpha}(z)^2z^{-2n(\alpha)})^{2g+2k-2}\prod^{2k}_{j=1}s_{\alpha}[\frac{\mathbf{x}^{(0)}_j+\mathbf{x}^{(1)}_j}{1-z}],
\end{align*}
\endgroup
where the bracket means plethystic substitution.

\subsection{The Formulas for E-Polynomials}
\subsubsection{}
Recall that for each $1\le j\le 2k$, we let $\bs\beta_j$ denote the type of the semi-simple conjugacy class $C_j$.
\begin{Thm}\label{Main-Thm}
Suppose that $\mathcal{C}$ is strongly generic and satisfies \ref{CCL}. Then with the notations above, we have
\begingroup
\allowdisplaybreaks
\begin{align}
|\Ch_{\mathcal{C}}(\mathbb{F}_q)|=&~q^{\frac{1}{2}d}
\left\langle\frac{\Omega(q)_1\Omega(q)_{0}}{\Omega_{\ast}(q)},\prod^{2k}_{j=1}h_{\bs\beta^{\ast}_j}(\X_j)\right\rangle,\text{ if $n$ is odd;}\\
%%%
|\Ch_{\mathcal{C}}(\mathbb{F}_q)|=&~q^{\frac{1}{2}d}
\left\langle\frac{\Omega(q)_{0}^2}{\Omega_{\ast}(q)},\prod^{2k}_{j=1}h_{\bs\beta^{\ast}_j}(\X_j)\right\rangle,\text{ if $n$ is even.}
\end{align}
\endgroup
\end{Thm}
\begin{proof}
Proposition \ref{M=U/G}, equation (\ref{RepCFq}) and Proposition \ref{Frob-Form} now give (\textit{cf.} Notation \ref{A(C)} and Notation \ref{Cj+-})
\begingroup
\allowdisplaybreaks
\begin{align}
\nonumber
|\Ch_{\mathcal{C}}(\mathbb{F}_q)|=&\sum_{\mathbf{e}=(e_j)\in\mathbf{A}(\mathcal{C})}\sum_{\chi\in\Irr(G^F)^{\sigma}}\left(\frac{|G^F|}{\chi(1)}\right)^{2g-2}\prod_{j=1}^{2k}\frac{|C_{j,e_j}|\tilde{\chi}(C_{j,e_j})}{\chi(1)}\\%%%
=&\sum_{\mathbf{e}=(e_j)\in\mathbf{A}(\mathcal{C})}\sum_{|\epsilon\bs\omega|=N}\frac{|G^F|^{2g-2}\prod^{2k}_{j=1}|C_{j,e_j}|}{\chi(1)^{2g+2k-2}}\sum_{\chi\in\Irr^{\sigma}_{\epsilon\bs\omega}}\prod^{2k}_{j=1}\tilde{\chi}(C_{j,e_j}).
\label{eq-brute-E-poly}
\end{align}
\endgroup
In the second equality, we have used the fact that irreducible characters of the same type  have the same degree (see (\ref{eq-Mac6.7}) and (\ref{cd-brace})). We have also used Lemma \ref{restrictTypes} to restrict the summation to a small subset of $\Irr(G^F)^{\sigma}$. This formula holds for any generic tuple of conjugacy classes. 

We will only give the proof for $n$ odd. The even $n$ case is completely analogous. Since we assume $\mathcal{C}$ to satisfy \ref{CCL}, in the odd $n$ case, we have $\mathbf{A}(\mathcal{C})\cong\{\pm1\}^{2k}$, whereas in the even $n$ case, the group $\mathbf{A}(\mathcal{C})$ is trivial. We calculate
\begingroup
\allowdisplaybreaks
\begin{align}
\nonumber%%
&|\Ch_{\mathcal{C}}(\mathbb{F}_q)|\\
\nonumber%%
=&\sum_{\mathbf{e}=(e_j)\in\{\pm1\}^{2k}}\sum_{|\epsilon\bs\omega|=N}\frac{|G^F|^{2g-2}\prod^{2k}_{j=1}|C_{j,e_j}|}{\chi(1)^{2g+2k-2}}\sum_{\chi\in\Irr^{\sigma}_{\epsilon\bs\omega}}\prod^{2k}_{j=1}\tilde{\chi}(C_{j,e_j})\\%%
\nonumber
\stackrel{\circled{1}}{=}&2^{2k}\sum_{\bs\omega\in\mathfrak{T}}\frac{K(\bs\omega_{\ast})}{N(\bs\omega_{\ast})}\frac{|G^F|^{2g-2}\prod^{2k}_{j=1}|C_{j,+}|}{\chi(1)^{2g+2k-2}}\prod^{2k}_{j=1}\langle s_{\bs\omega}(\X_j),\tilde{H}_{\bs{\Lambda}(\bs{\beta}_j)}(\X_j,q)\rangle\\
%%%
\nonumber
=&2^{2k}\sum_{\bs\omega\in\mathfrak{T}}\frac{K(\bs\omega_{\ast})}{N(\bs\omega_{\ast})}\left(\frac{|G^F|}{\chi(1)}\right)^{2g+2k-2}\prod^{2k}_{j=1}\frac{|C_{j,+}|}{|G^F|}\prod^{2k}_{j=1}\langle s_{\bs\omega}(\X_j),\tilde{H}_{\bs{\Lambda}(\bs{\beta}_j)}(\X_j,q)\rangle\\
%%%
\nonumber
\stackrel{\circled{2}}{=}&2^{2k}\sum_{\bs\omega\in\mathfrak{T}}\frac{K(\bs\omega_{\ast})}{N(\bs\omega_{\ast})}(H_{\{\bs\omega\}}(q)q^{\frac{1}{2}n(n-1)-n(\{\bs\omega\})})^{2g+2k-2}\prod^{2k}_{j=1}\langle s_{\bs\omega}(\X_j),\frac{|C_{j,+}|}{|G^F|}\tilde{H}_{\bs{\Lambda}(\bs{\beta}_j)}(\X_j,q)\rangle\\%%
\nonumber
=&2^{2k}q^{\frac{1}{2}(n^2(2g+2k-2)-2kN)}\\
&\cdot\left\langle q^{-k(N+1)}\sum_{\bs\omega\in\mathfrak{T}}\!\frac{K(\bs\omega_{\ast})}{N(\bs\omega_{\ast})}q^{(1-g)|\{\bs\omega\}|}(H_{\{\bs\omega\}}(q)q^{-n(\{\bs\omega\})})^{2g+2k-2}\!\prod^{2k}_{j=1}\!s_{\bs\omega}(\X_j),\prod^{2k}_{j=1}\!\!\frac{|C_{j,+}|}{|G^F|}\tilde{H}_{\bs{\Lambda}(\bs{\beta}_j)}(\X_j,q)\rangle \right\rangle.\label{innerprod}
\end{align}
\endgroup
We have the factor $2^{2k}$ in \circled{1} because\begin{equation}\label{Delta}
\sum_{\mathbf{e}=(e_j)\in\{\pm\}^{2k}}\Delta_{\epsilon,\mathcal{C}_{\mathbf{e}}}=
\begin{cases}
2^{2k} & \text{if } \epsilon=+,\\
0 & \text{if }\epsilon=-,
\end{cases}
\end{equation}
and the sum can be taken over the entire $\mathfrak{T}$ because if $\bs\omega$ was not of size $N$ then the inner product of symmetric functions would vanish. Equality \circled{2} uses (\ref{eq-Mac6.7}).

Equations (\ref{mathbbG}) and (\ref{1/|GLn|}) show that
\begin{equation}
\frac{|C_{j,+}|}{|G^F|}=\frac{1}{2}q^{-2a(\bs\beta_j)-|\bs\beta_j|}b_{\bs{\Lambda}(\bs\beta_j)}(q^{-1})^{-1},
\end{equation}
where the factor $1/2$ comes from the two connected components of the orthogonal group. Combined with Corollary \ref{XoXlBCType}, this shows
\begin{equation}
\frac{1}{2}q^{-a(\bs\beta_j)}h_{\bs\beta^{\ast}_j}(\mathbf{P}\X_j)=
(-1)^{|\bs\beta_j|}\frac{|C_{j,+}|}{|G^F|}\tilde{H}_{\bs{\Lambda}(\bs\beta_j)}(\X_j;q),
\end{equation}
where we have specialised the variable $z$ in (\ref{P}) to $q$.

Note that for any symmetric functions $u(\Xo,\Xl)$ and $v(\Xo,\Xl)$, we have $$\left\langle u(\mathbf{P}\begin{bmatrix}
\Xo\\
\Xl
\end{bmatrix}),v(\begin{bmatrix}
\Xo\\
\Xl
\end{bmatrix})\right\rangle
=
\left\langle u(
\begin{bmatrix}
\Xo\\
\Xl
\end{bmatrix}),
v(\mathbf{P}\begin{bmatrix}
\Xo\\
\Xl
\end{bmatrix})
\right\rangle.$$ This can be checked on the basis of power sums. We will thus move $\mathbf{P}$ to the left hand side of the inner product.

Rewrite the left hand side of the inner product (\ref{innerprod}) as follows:
\begingroup
\allowdisplaybreaks
\begin{align}
\nonumber%%
&q^{-k(|\bs\omega_+|+|\bs\omega_-|+\sum_i|\omega_i|+1)}\sum_{\bs\omega\in\mathfrak{T}}\frac{K(\bs\omega_{\ast})}{N(\bs\omega_{\ast})}q^{(1-g)|\{\bs\omega\}|}(H_{\{\bs\omega\}}(q)q^{-n(\{\bs\omega\})})^{2g+2k-2}\prod^{2k}_{j=1}s_{\bs\omega}(\mathbf{P}\X_j)\\
\nonumber%%
=&~~\left(\sum_{\bs\omega_+\in\mathcal{P}^2}q^{k|\bs\omega_+|}q^{(1-g-k)|\{\bs\omega_+\}_1|}(H_{\{\bs\omega_+\}_1}(q)q^{-n(\{\bs\omega_+\}_1)})^{2g+2k-2}\prod^{2k}_{j=1}s_{\bs\omega_+}(\mathbf{P}\X_j)\right)\\
\nonumber%%
&\cdot\left(\sum_{\bs\omega_-\in\mathcal{P}^2}q^{k|\bs\omega_-|}q^{(1-g-k)|\{\bs\omega_-\}_0|}(H_{\{\bs\omega_-\}_0}(q)q^{-n(\{\bs\omega_-\}_0)})^{2g+2k-2}\prod^{2k}_{j=1}s_{\bs\omega_-}(\mathbf{P}\X_j)\right)\\
\nonumber%%
&\cdot\left(\sum_{\bs\omega_{\ast}}\frac{K(\bs\omega_{\ast})}{N(\bs\omega_{\ast})}q^{(2-2g-k)|\bs\omega_{\ast}|}(H_{\bs\omega_{\ast}}(q)^2q^{-2n(\bs\omega_{\ast})})^{2g+2k-2}\prod^{2k}_{j=1}s_{\bs\omega_{\ast}}(\mathbf{P}\X_j)\right).
\end{align}
\endgroup

By definition
\begingroup
\allowdisplaybreaks
\begin{align}
\nonumber%%
\frac{1}{\Omega_{\ast}(q)}=&\sum_{m\ge 0}(-1)^m\!\!\!\!\sum_{\substack{\bs\omega_{\ast}=(m_{\lambda})_{\lambda}\in\mathfrak{T}_{\ast}\\l(\bs\omega_{\ast})=m}}\frac{m!}{\prod_{\lambda}m_{\lambda}!}q^{(2-2g-k)|\bs\omega_{\ast}|}(H_{\bs\omega_{\ast}}(q)^2q^{-2n(\bs\omega_{\ast})})^{2g+2k-2}\prod^{2k}_{j=1}s_{\bs\omega_{\ast}}(\mathbf{P}\X_j)\\
%%%
=&\sum_{\bs\omega_{\ast}\in\mathfrak{T}_{\ast}}\frac{K(\bs\omega_{\ast})}{N(\bs\omega_{\ast})}q^{(2-2g-k)|\bs\omega_{\ast}|}(H_{\bs\omega_{\ast}}(q)^2q^{-2n(\bs\omega_{\ast})})^{2g+2k-2}\prod^{2k}_{j=1}s_{\bs\omega_{\ast}}(\mathbf{P}\X_j),
\end{align}
\endgroup
whence
\begingroup
\allowdisplaybreaks
\begin{align*}
&q^{-k(|\bs\omega_+|+|\bs\omega_-|+\sum_i|\omega_i|+1)}\sum_{\bs\omega\in\mathfrak{T}}\frac{K(\bs\omega_{\ast})}{N(\bs\omega_{\ast})}q^{(1-g)|\{\bs\omega\}|}(H_{\{\bs\omega\}}(q)q^{-n(\{\bs\omega\})})^{2g+2k-2}\prod^{2k}_{j=1}s_{\bs\omega}(\mathbf{P}\X_j)\\
=&\frac{\Omega_1(q)\Omega_0(q)}{\Omega_{\ast}(q)}
\end{align*}
\endgroup

Therefore,
\begin{equation}
|\Ch_{\mathcal{C}}(\mathbb{F}_q)|
=q^{\frac{1}{2}\left(n^2(2g+2k-2)-2kN\right)-\sum^{2k}_{j=1}a(\bs\beta_j)}\cdot\left\langle\frac{\Omega(q)_1\Omega(q)_{0}}{\Omega_{\ast}(q)},\prod^{2k}_{j=1}h_{\bs\beta^{\ast}_j}(\X_j)\right\rangle.
\end{equation}
Finally, we use \S \ref{dim-Ch}.
\end{proof}

\subsection{Connected Components}\label{Conn-Compo}\hfill

Here we give some explicit computations of E-polynomials in the case $n=2$ and $n=3$ using (\ref{eq-brute-E-poly}) and the character tables \cite[Table 1,2,3,4,5]{Shu2}. Unlike Theorem \ref{Main-Thm}, which assumes \ref{CCL}, we only assume genericity of conjugacy classes. This allows conjugacy classes with non-connected centralisers when $n$ is even. These computations lead to a conjecture on the number of connected components of $\Ch_{\mathcal{C}}$.
\subsubsection{$G=\GL_2$}

Suppose that $g=1$ and $2k=2$. Let $C_2$ be the conjugacy class of $\sigma$ and let $C_1$ be the conjugacy class of $\diag(a,a^{-1})\sigma$, with $a^q=a$ and $a^4\ne 1$, such that $\mathcal{C}=(C_1,C_2)$ is generic. The counting formula gives
\begingroup
\allowdisplaybreaks
\begin{align}
\nonumber%%
|\Ch_{\mathcal{C}}(\mathbb{F}_q)|=&\sum_{\chi}\frac{|C_1||C_2|\tilde{\chi}(C_1)\tilde{\chi}(C_2)}{\chi(1)^2}\\
\nonumber%%
=&2q^4-2q^3-2q+2.
\end{align}
\endgroup
Now let $C_{2,+}$ be the conjugacy class of $\diag(\mathfrak{i},-\mathfrak{i})\sigma$. Since its centraliser is not connected, there is another $G(q)$-conjugacy class $C_{2,-}$ that is $G$-conjugate to $\diag(\mathfrak{i},-\mathfrak{i})\sigma$. The counting formula gives
\begingroup
\allowdisplaybreaks
\begin{align*}
|\Ch_{\mathcal{C}}(\mathbb{F}_q)|=&\sum_{\chi}\left(\frac{|C_1||C_2|\tilde{\chi}(C_1)\tilde{\chi}(C_{2,+})}{\chi(1)^2}+\frac{|C_1||C_2|\tilde{\chi}(C_1)\tilde{\chi}(C_{2,-})}{\chi(1)^2}\right)\\
=&q^6-3q^4+4q^3-3q^2+1.
\end{align*}
\endgroup

Suppose that $g=1$ and $2k=4$. For $j\in\{1,2\}$, let $C_j$ be the conjugacy class of $\diag(a_j,a_j^{-1})\sigma$, with $a_j^q=a_j$ and $a_j^4\ne1$, and let $C_{3,+}$ and $C_{4,+}$ be the conjugacy class of $\diag(\mathfrak{i},-\mathfrak{i})\sigma$ such that $(C_1,C_2,C_{3,+},C_{4,+})$ is generic. The counting formula gives
\begingroup
\allowdisplaybreaks
\begin{align*}
|\Ch_{\mathcal{C}}(\mathbb{F}_q)|=&\sum_{\chi}\sum_{(e_3,e_4)\in\{\pm\}^2}\frac{|C_1||C_2||C_{3,e_3}||C_{4,e_4}|}{\chi(1)^4}\tilde{\chi}(C_1)\tilde{\chi}(C_2)\tilde{\chi}(C_{3,e_3})\tilde{\chi}(C_{4,e_4})\\
=&q^{12}-4q^{10}+3q^8+16q^7-32q^6+16q^5+3q^4-4q^2+1.
\end{align*}
\endgroup

%%%%%%%%%%%%%%%%
\begin{comment}%% a_1=i
\subsubsection{$g=1$, $2k=2$.}
Now we take $a_1=\sqrt{-1}$. This is an interesting case since then $\mathcal{C}$ is not strictly generic but is still generic. And the counting result should be compared with the previous example. Note the change in the cardinality of $C_1$. We have:
\begingroup
\allowdisplaybreaks
\begin{align*}
|\Ch_{\mathcal{C}}(\mathbb{F}_q)|=&\frac{q(q-1)^2(q+1)}{2(q+1)^2}\sum_{\alpha}(q+1)\alpha(-1)\\
%%
&+2\frac{q(q-1)^2(q+1)}{2}\\
%%
&+2\frac{q(q-1)^2(q+1)}{2q^2}\cdot q\\
%%
=&q^4-2q^3+2q^2-4q+1.
\end{align*}
\endgroup
\end{comment}
%%%%%%%%%%%%%%%%

\subsubsection{$G=\GL_3$}
Suppose that $g=0$ and $2k=4$. For any $j\in\{1,2,3,4\}$, let $C_{j,+}$ be the conjugacy class of $\diag(a_j,1,a_j^{-1})\sigma$, with $a_j^q=a_j$ and $a_j^4\ne 1$, such that $\mathcal{C}=(C_{1,+},C_{2,+},C_{3,+},C_{4,+})$ is generic. Let $c\in\mathbb{F}^{\ast}_q\setminus(\mathbb{F}_q^{\ast})^2$ and let $C_{j,-}$ be the conjugacy class of $\diag(a_j,c,a_j^{-1})\sigma$ for each $j$. Now the counting formula gives
\begingroup
\allowdisplaybreaks
\begin{align*}
|\Ch_{\mathcal{C}}(\mathbb{F}_q)|
=&\sum_{(e_j)\in\{\pm\}^4}\sum_{\chi}\left(\frac{|G(q)|}{\chi(1)}\right)^2\prod_j\frac{|C_{j,e_j}|}{|G(q)|}\prod_j\tilde{\chi}(C_{j,e_j})\\
=&q^{14}+2q^{13}+q^{12}-2q^{11}-3q^{10}-2q^9-15q^8+36q^7
\\%%
&-15q^6-2q^5-3q^4-2q^3+q^2+2q+1.
\end{align*}
\endgroup
Then we keep $C_{1,\pm}$, $C_{2,\pm}$ and $C_{3,\pm}$ as above but put $C_{4,\pm}$ to be the conjugacy classes of $\sigma$ and $\diag(1,c,1)\sigma$. The same counting formula gives
\begingroup
\allowdisplaybreaks
\begin{align*}
|\Ch_{\mathcal{C}}(\mathbb{F}_q)|
=&q^{12}+q^{11}-2q^9-q^8-7q^7+16q^6-7q^5-q^4-2q^3+q+1.
\end{align*}
\endgroup

\subsubsection{}
Observing the leading coefficients in the E-polynomials, we find that some of the character varieties have two connected components while others are connected. The isomorphism (\ref{Ch-rank-2}) and Remark \ref{conn-Ch-rank-2} in the appendix show that in the case $n=2$, the variety $\Ch_{\mathcal{C}}$ is connected exactly when a conjugacy class has non connected centraliser. Note that when $n$ is odd, a semi-simple class in $\GL_n\sigma$ always has non-connected centraliser.

\begin{Con}
Let $\mathcal{C}$ be a tuple of semi-simple conjugacy classes in $G\sigma$ that is not necessarily generic. Suppose that the character variety $\Ch_{\mathcal{C}}$ is non empty. Then it has two connected components if and only if every conjugacy class in $\mathcal{C}$ contains elements with connected centraliser, and is connected otherwise.
\end{Con}

%%%%%%%%%%%%%%%%%%%%
\section{Proof of the Main Theorem}\label{Section-Proof}
The goal of this section is to prove that our formula for the E-polynomial given by Theorem \ref{Main-Thm} agrees with the $t=-1$ specialisation of Conjecture \ref{The-Conj}. We also show that Conjecture \ref{The-Conj} implies Conjecture \ref{Cur-Poin}.

\subsection{Macdonald Polynomials}\label{subs-Mac}
\subsubsection{}
We begin by collecting some basic facts about the \textit{Macdonald polynomials}.

\begin{Defn}
Let $N\in\mathbb{Z}_{\ge0}$. The Macdonald polynomials are a family of symmetric functions$$\{H_{\lambda}(\Xo;z,w)\in\SymF_{z,w}[\Xo]\mid\lambda\in\mathcal{P}(N)\}$$ that are uniquely determined by the following properties:
\begin{itemize}
\item [(i)] $H_{\lambda}[(1-z)\Xo;z,w]\in\mathbb{Q}(z,w)\{ s_{\mu}(\Xo):\mu\ge\lambda\}$;

\item [(ii)] $H_{\lambda}[(1-w)\Xo;z,w]\in\mathbb{Q}(z,w)\{ s_{\mu}(\Xo):\mu\ge\lambda^{\ast}\}$;

\item [(iii)] $\langle H_{\lambda}(\Xo;z,w),s_{(n)}(\Xo)\rangle=1$,
\end{itemize}
where $\mathbb{Q}(z,w)\{-\}$ signifies the vector space over $\mathbb{Q}(z,w)$ spanned by the indicated elements.
\end{Defn}

The $z,w$-inner product on $\SymF_{z,w}[\Xo]$ is defined by: $$
\left\langle F[\Xo],G[\Xo]\right\rangle_{z,w}:=\left\langle F[\Xo],G[
(z-1)(1-w)\Xo]\right\rangle
$$
for any $F$, $G\in\SymF_{z,w}[\Xo]$. The Macdonald polynomials are orthogonal with respect to the $z,w$-inner product. Let $$N_{\lambda}(z,w)=\langle H_{\lambda}(\Xo;z,w),H_{\lambda}(\Xo;z,w)\rangle_{z,w}$$be the self-pairing of Macdonald polynomials. It has an explicit expression (See \cite[Corollary 5.4.8]{Hai}.)
\begin{equation}\label{Intro-N(q,t)}
N_{\lambda}(z,w)=\prod_{x\in\lambda}(z^{a(x)+1}-w^{l(x)})(z^{a(x)}-w^{l(x)+1}),
\end{equation}
where $x$ runs over the boxes in the Young diagram of $\lambda$, $a(x)$ and $l(x)$ denote the arm-length and leg-length respectively. We will need a deformation of $N_{\lambda}(z,w)$ defined by 
\begin{equation}\label{Intro-N(u,z,w)}
N_{\lambda}(u,z,w):=\prod_{x\in\lambda}(z^{a(x)+1}-uw^{l(x)})(z^{a(x)}-u^{-1}w^{l(x)+1}).\end{equation}
\begin{Rem}\label{Sym-Mac}
We have 
\begingroup
\allowdisplaybreaks
\begin{align*}
H_{\lambda}(\Xo;w,z)=&H_{\lambda^{\ast}}(\Xo;z,w),\\
N_{\lambda}(w,z)=&N_{\lambda^{\ast}}(z,w),\\
N_{\lambda}(zw,w^2,z^2)=&N_{\lambda^{\ast}}(zw,z^2,w^2).
\end{align*}
\endgroup
\end{Rem}

\begin{Lem}\label{Mac-t=q^-1}
For any $\lambda\in\mathcal{P}$, we have$$
H_{\lambda}(\Xo;z,z^{-1})=\frac{s_{\lambda}[\frac{\Xo}{1-z}]}{s_{\lambda}[\frac{1}{1-z}]}
$$
\end{Lem}

\begin{Lem}\label{N-t=q^-1}
For any $\lambda\in\mathcal{P}$, write $$f_{\lambda}(z)=\frac{1}{s_{\lambda}[\frac{1}{1-z}]}$$ then we have$$
N_{\lambda}(z,z^{-1})=z^{-|\lambda|}f_{\lambda}(z)^2.$$
\end{Lem}
The proofs of the above two lemmas are simple. We will state and prove the corresponding results for wreath Macdonald polynomials below. We will also need:
\begin{Prop}(\cite[Proposition 3.3.2]{Hai})\label{Hai3.3.2}
Let $H_{\lambda}(z)$ be the hook polynomial (\ref{eq-hook}). Then,
$$s_{\lambda}[\frac{1}{1-z}]=z^{n(\lambda)}H_{\lambda}(z)^{-1}.$$
\end{Prop}

\subsection{Wreath Macdonald Polynomials}\hfill

Wreath Macdonald polynomials can be defined for $(\mathbb{Z}/r\mathbb{Z})^N\rtimes\mathfrak{S}_N$ for any $r>1$. We will only need the case $r=2$, and will therefore work in the ring $\SymF:=\SymF_{z,w}[\Xo,\Xl]$.
\subsubsection{}
Recall that a partition is uniquely determined by its 2-core and 2-quotient. Given a 2-core, then the dominance partial order on $\mathcal{P}$ induces an order on the set of 2-partitions, regarded as the 2-quotient of those partitions with the given 2-core. Different 2-cores may give different orders. There are only finitely many different such induced orders.

Recall that $\X$ is regarded as a column of variables.
\begin{Defn}\label{wMac}
Let $N\in\mathbb{Z}_{\ge0}$. Fix a 2-core and thus an order on $\mathcal{P}^2(N)$. The \textit{wreath Macdonald polynomials} are a family of symmetric functions $$\{\tilde{H}_{\bs\alpha}(\X;z,w)\in\SymF\mid\bs\alpha\in\mathcal{P}^2(N)\}$$ that are uniquely determined by the following properties:
\begin{itemize}
\item [(i)] $\tilde{H}_{\bs\alpha}(\left(
\begin{array}{cc}
1 & -z\\
-z & 1
\end{array}
\right)\X;z,w)\in\mathbb{Q}(z,w)\{ s_{\bs\beta}(\X):\bs\beta\ge\bs\alpha\}$;

\item [(ii)] $\tilde{H}_{\bs\alpha}(\left(
\begin{array}{cc}
1 & -w\\
-w & 1
\end{array}
\right)\X;z,w)\in\mathbb{Q}(z,w)\{ s_{\bs\beta}(\X):\bs\beta\ge\bs\alpha^{\ast}\}$;

\item [(iii)] $\langle\tilde{H}_{\bs\alpha}(\X;z,w),s_{((n),\varnothing)}(\X)\rangle=1$.
\end{itemize}
\end{Defn}

The first two conditions in the definition are called triangularity conditions and the last one is called the normalisation condition.  

\begin{Thm}(\cite[Conjecture 7.2.19]{Hai},\cite[Corollary 1.2]{BF})
Wreath Macdonald polynomials exist and they are Schur-positive.
\end{Thm}
\begin{Rem}
The Schur positivity for wreath Macdonald polynomials means that the coefficients in the Schur basis are \textit{Laurent polynomials} instead of polynomials. We will see examples of this in Appendix \ref{App-WMac1-2}.
\end{Rem}

\subsubsection{}\label{w-q,t-inner}
We define the $z,w$-inner product on $\SymF$ by: $$
\left\langle F[\X],G[\X]\right\rangle_{z,w}:=\left\langle F[\X],G[\left(
\begin{array}{cc}
0 & -1\\
-1 & 0
\end{array}
\right)
\left(
\begin{array}{cc}
1 & -z\\
-z & 1
\end{array}
\right)
\left(
\begin{array}{cc}
1 & -w\\
-w & 1
\end{array}
\right)\X]\right\rangle
$$
for any $F$, $G\in\SymF$. This inner product is symmetric in $F$ and $G$. Observe that$$s_{\bs\alpha}[\left(
\begin{array}{cc}
0 & -1\\
-1 & 0
\end{array}
\right)\X]=(-1)^{|\bs\alpha|}s_{\bs\alpha^{\ast}}(\X).$$ The wreath Macdonald polynomials are orthogonal with respect to this inner product. Let $$\tilde{N}_{\bs\alpha}(z,w)=\langle \tilde{H}_{\bs\alpha}(\X;z,w),\tilde{H}_{\bs\alpha}(\X;z,w)\rangle_{z,w}$$be the self-pairing of wreath Macdonald polynomials.

According to Orr and Shimozono \cite{OS}, the self-pairing $\tilde{N}_{\bs\alpha}(z,w)$ can be written as $$\tilde{N}_{\bs\alpha}(z,w)=\tilde{N}^{\nabla}_{\bs\alpha}(z,w)\tilde{N}^{cot}_{\bs\alpha}(z,w),$$where for $e=0$ or $1$,
$$
\tilde{N}^{\cot}_{\bs\alpha}(z,w)=\!\!\!\prod_{\substack{x\in\{\bs\alpha\}_e\\h(x)\equiv0\!\!\!\mod 2}}\!\!\!(z^{a(x)+1}-w^{l(x)})(z^{a(x)}-w^{l(x)+1})
$$is the contribution from the cotangent bundle of the Hilbert scheme along the $\mathbb{Z}/2\mathbb{Z}$-fixed point locus, and $\tilde{N}^{\nabla}_{\bs\alpha}(z,w)$ is a complicated term involving a "wreath-$\nabla$" operator. We define the deformation
$$
\tilde{N}^{cot}_{\bs\alpha}(u,z,w):=\!\!\!\prod_{\substack{x\in\{\bs\alpha\}_e\\h(x)\equiv0\!\!\!\mod 2}}\!\!\!(z^{a(x)+1}-uw^{l(x)})(z^{a(x)}-u^{-1}w^{l(x)+1}),
$$in agreement with (\ref{Intro-N(u,z,w)}), and then $$\tilde{N}_{\bs\alpha}(u,z,w)=\tilde{N}^{\nabla}_{\bs\alpha}(z,w)\tilde{N}^{cot}_{\bs\alpha}(u,z,w),$$i.e. the "wreath-$\nabla$" term is not deformed.
\begin{Rem}\label{Sym-wMac}
We have 
\begingroup
\allowdisplaybreaks
\begin{align*}
\tilde{H}_{\bs\alpha}(\X;w,z)=&\tilde{H}_{\bs\alpha^{\ast}}(\X;z,w),\\
\tilde{N}_{\bs\alpha}(w,z)=&\tilde{N}_{\bs\alpha^{\ast}}(z,w),\\
\tilde{N}_{\bs\alpha}(zw,w^2,z^2)=&\tilde{N}_{\bs\alpha^{\ast}}(zw,z^2,w^2).
\end{align*}
\endgroup
\end{Rem}

\subsubsection{}
In this paragraph, we work with a fixed 2-core, which can be either $(0)$ or $(1)$.
\begin{Lem}\label{wMac-t=q^-1}
For any $\bs\alpha\in\mathcal{P}^2$, we have$$
\tilde{H}_{\bs\alpha}(\X;z,z^{-1})=\frac{s_{\bs\alpha}[\left(
\begin{array}{cc}
1 & -z\\
-z & 1
\end{array}
\right)^{-1}\X]}{s_{\alpha^{(0)}}[\frac{1}{1-z^2}]s_{\alpha^{(1)}}[\frac{z}{1-z^2}]}
$$
\end{Lem}
Note that this expression has nothing to do with the given 2-core.
\begin{proof}
The first two conditions in Definition \ref{wMac} imply that $$\tilde{H}_{\bs\alpha}(\X;z,z^{-1})=f(z)s_{\bs\alpha}[\left(
\begin{array}{cc}
1 & -z\\
-z & 1
\end{array}
\right)^{-1}\X]$$ for some $f(z)\in\mathbb{Q}(z)$. Taking inner product with $s_{((n),\varnothing)}(\X)=h_{((n),\varnothing)}(\X)$ is equivalent to setting $\Xo=1$ and $\Xl=0$. Therefore condition (iii) of Definition \ref{wMac} gives $$f(z)s_{\bs\alpha}[\left(
\begin{array}{cc}
1 & -z\\
-z & 1
\end{array}
\right)^{-1}
\begin{bmatrix}
1\\
0
\end{bmatrix}]=1,$$whence the lemma.
\end{proof}

\begin{Lem}\label{N'-t=q^-1}
For any $\bs\alpha\in\mathcal{P}^2$, write $$f_{\bs\alpha}(z)=\frac{1}{{s_{\alpha^{(0)}}[\frac{1}{1-z^2}]s_{\alpha^{(1)}}[\frac{z}{1-z^2}]}}$$ then we have$$
\tilde{N}_{\bs\alpha}(z,z^{-1})=z^{-|\bs\alpha|}f_{\bs\alpha}(z)^2.$$
\end{Lem}
\begin{proof}
Using the previous lemma, we calculate
\begingroup
\allowdisplaybreaks
\begin{align*}
&\left\langle \tilde{H}_{\bs\alpha}(\X;z,z^{-1}),\tilde{H}_{\bs\alpha}[\left(
\begin{array}{cc}
0 & -1\\
-1 & 0
\end{array}
\right)
\left(
\begin{array}{cc}
1 & -z\\
-z & 1
\end{array}
\right)
\left(
\begin{array}{cc}
1 & -z^{-1}\\
-z^{-1} & 1
\end{array}
\right)\X;z,w]\right\rangle
\\
=&\langle f(z)s_{\bs\alpha}(\X),f(z)s_{\bs\alpha}[\frac{\X}{z}]\rangle
\\
=&z^{-|\bs\alpha|}f_{\bs\alpha}(z)^2.
\end{align*}
\endgroup
\end{proof}

\subsection{Proof of the Main Theorem}
\subsubsection{}
Recall the symmetric functions defined in  \S \ref{subsec-Conjecture} and \S \ref{Omega(q)}.
\begin{Lem}\label{Euler-Omega}
We have
\begingroup
\allowdisplaybreaks
\begin{align*}
\Omega_1(\sqrt{z},\frac{1}{\sqrt{z}})&=\Omega_1(z),\\
\Omega_0(\sqrt{z},\frac{1}{\sqrt{z}})&=\Omega_0(z),\\
\Omega_{\ast}(\sqrt{z},\frac{1}{\sqrt{z}})&=\Omega_{\ast}(z).\\
\end{align*}
\endgroup
\end{Lem}
\begin{proof}
Let us compute $\Omega_1(\sqrt{z},\frac{1}{\sqrt{z}})$. By Lemma \ref{wMac-t=q^-1} and Lemma \ref{N'-t=q^-1}, we get \begin{equation}\label{Euler-Omega-eq1}
\frac{\prod_{j=1}^{2k}\tilde{H}_{\bs\alpha}(\X_j;z,z^{-1})}{\tilde{N}_{\bs\alpha}(z,z^{-1})\tilde{N}_{\bs\alpha}(1,z,z^{-1})^{k-1}}=z^{k|\bs\alpha|}\prod_{j=1}^{2k}s_{\bs\alpha}[\left(
\begin{array}{cc}
1 & -z\\
-z & 1
\end{array}
\right)^{-1}\X_j].
\end{equation}
By Lemma \ref{N-t=q^-1} and Proposition \ref{Hai3.3.2}, we have 
\begin{equation}\label{Euler-Omega-eq2}
N_{\{\bs\alpha\}_1}(1,z,z^{-1})^{g+k-1}=z^{-(g+k-1)|\{\bs\alpha\}_1|}(z^{-n(\{\bs\alpha\}_1)}H_{\lambda}(z))^{2(g+k-1)},
\end{equation}
whence the equality for $\Omega_1$. The proof for $\Omega_0$ is similar. For $\Omega_{\ast}$, we use Lemma \ref{Mac-t=q^-1} and Lemma \ref{N-t=q^-1}.
\end{proof}

The following proposition finishes the proof of Theorem \ref{Main-Thm-intro}.
\begin{Prop}
The following equality holds:
\begin{equation*}
H_c(\Ch_{\mathcal{C}};q,-1)=q^{\frac{1}{2}d}\mathbb{H}_{\mathbf{B}}(\sqrt{q},\frac{1}{\sqrt{q}}).
\end{equation*}
\end{Prop}
\begin{proof}
By Theorem \ref{Main-Thm}, we have
$$H_c(\Ch_{\mathcal{C}};q,-1)=q^{\frac{1}{2}d}
\left\langle\frac{\Omega(q)_1\Omega(q)_{0}}{\Omega_{\ast}(q)},\prod^{2k}_{j=1}h_{\bs\beta^{\ast}_j}(\X_j)\right\rangle$$if $n$ is odd. Then the proposition follows from Lemma \ref{Euler-Omega}. The case of even $n$ is similar.
\end{proof}

\begin{Prop}
The formula in Conjecture \ref{The-Conj} (iii) implies Conjecture \ref{Cur-Poin}.
\end{Prop}
\begin{proof}
By the definitions of $\Omega_e$, $e=0$, $1$, and $\Omega_{\ast}$ in \S \ref{subsec-series}, and Remark \ref{Sym-Mac} and Remark \ref{Sym-wMac}, we have
$$
\Omega_e(z,w)=\Omega_e(w,z),\quad\Omega_{\ast}(z,w)=\Omega_{\ast}(w,z).
$$
Then we see that
$$
\mathbb{H}_{\mathbf{B}}(z,w)=\mathbb{H}_{\mathbf{B}}(w,z),\quad\mathbb{H}_{\mathbf{B}}(z,w)=\mathbb{H}_{\mathbf{B}}(-z,-w).
$$
The proposition follows.
\end{proof}

\begin{Cor}
The E-polynomial $E(q)$ of $\Ch_{\mathcal{C}}$ satisfies
\begin{equation}
q^{d}E(q^{-1})=E(q).
\end{equation}
\end{Cor}
\begin{proof}
Use Theorem \ref{Main-Thm-intro} and the above proposition.
\end{proof}
This proves Theorem \ref{E-Cur-Poin}.

%%%%%%%%%%%%%%%%%%%%%%%%%%%%%%%%%%%%%%%%%%%%
\addtocontents{toc}{\protect\setcounter{tocdepth}{1}}
\appendix

%%%%%%%%%%%%%%%%%%%%%%
\section{Wreath Macdonald Polynomials of Degrees 1 and 2}\label{App-WMac1-2}
Here we give explicit expressions of Macdonald polynomials of degree 2 and wreath Macdonald polynomials of degrees 1 and 2. They will be used in the computation of the mixed Hodge polynomials of $\GL_n\lb\sigma\rb~$-character varieties in the appendices that follow. The reader will also observe some key features of wreath Macdonald polynomials.

\subsection{Young Diagrams}
\subsubsection{}
Young diagrams of the partitions of $2$, $3$, $4$ and $5$ are given below. The boxes defining their 2-quotients are indicated by "$\bullet$", and the second row of each table gives the corresponding 2-quotients. (See Remark \ref{conv-2-quot} for our convention on 2-quotient.) The partitions are ordered in such a way that a partition sits on the left of another if it is larger under the dominance order. We see that the 2-cores $(0)$ and $(1)$ induce the same order on the set of 2-partitions of size $1$, but different orders on the set of 2-partitions of size 2.
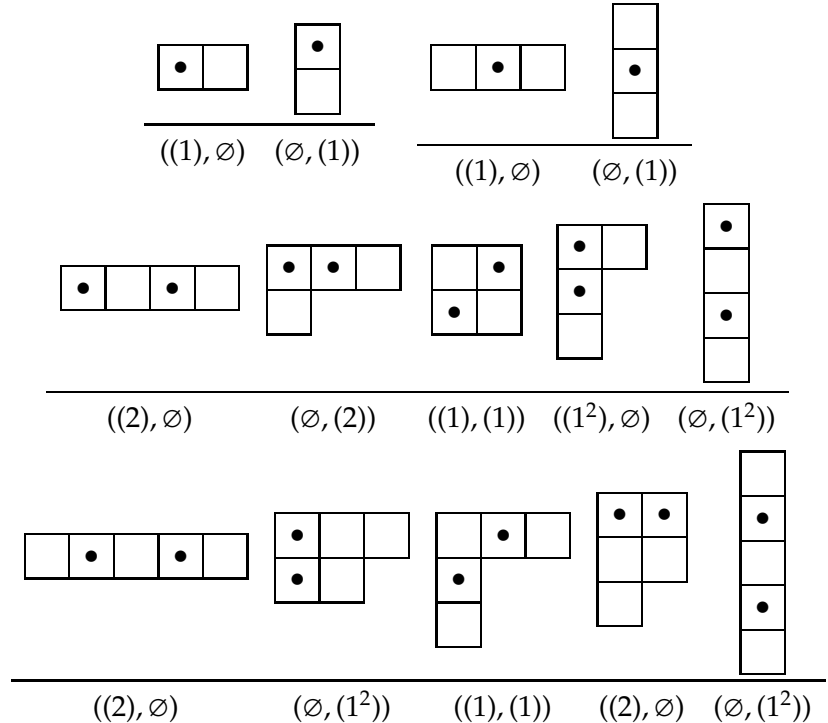
\begin{center}
{\renewcommand{\arraystretch}{1.2}
\begin{tabular}{cc}
\raisebox{-.5\normalbaselineskip}
{\begin{ytableau}
\bullet & {}
\end{ytableau}} 
& 
\begin{ytableau}
\bullet\\
{}
\end{ytableau}
\\[3ex] 
\hline
$((1),\varnothing)$ & $(\varnothing,(1))$\\
\end{tabular}}
\quad
{\renewcommand{\arraystretch}{1.2}
\begin{tabular}{cc}
\raisebox{-1\normalbaselineskip}
{\begin{ytableau}
{} & \bullet & {}
\end{ytableau}} 
& 
\begin{ytableau}
{} \\
\bullet\\
{}
\end{ytableau}
\\[6ex] 
\hline
$((1),\varnothing)$ & $(\varnothing,(1))$\\
\end{tabular}}
\end{center}

\begin{center}
{\renewcommand{\arraystretch}{1.2}
\begin{tabular}{ccccc}
\raisebox{-1.5\normalbaselineskip}
{\begin{ytableau}
\bullet & {} & \bullet & {}
\end{ytableau}} 
& 
\raisebox{-1\normalbaselineskip}
{\begin{ytableau}
\bullet & \bullet & {}\\
{}
\end{ytableau}}
&
\raisebox{-1\normalbaselineskip}
{\begin{ytableau}
{} & \bullet\\
\bullet & {}
\end{ytableau}}
&
\raisebox{-.5\normalbaselineskip}
{\begin{ytableau}
\bullet & {}\\
\bullet\\
{}
\end{ytableau}}
&
\begin{ytableau}
\bullet\\
{}\\
\bullet\\
{}
\end{ytableau}
\\[9.5ex] 
\hline
$((2),\varnothing)$ & $(\varnothing,(2))$ & $((1),(1))$ & $((1^2),\varnothing)$ & $(\varnothing,(1^2))$\\
\end{tabular}}
\end{center}

\begin{center}
{\renewcommand{\arraystretch}{1.2}
\begin{tabular}{ccccc}
\raisebox{-2\normalbaselineskip}
{\begin{ytableau}
{} & \bullet & {} & \bullet & {}
\end{ytableau}} 
& 
\raisebox{-1.5\normalbaselineskip}
{\begin{ytableau}
\bullet & {} & {}\\
\bullet & {}
\end{ytableau}}
&
\raisebox{-1.5\normalbaselineskip}
{\begin{ytableau}
{} & \bullet & {}\\
\bullet\\
{}
\end{ytableau}}
&
\raisebox{-1\normalbaselineskip}
{\begin{ytableau}
\bullet & \bullet\\
{} & {}\\
{}
\end{ytableau}}
&
\begin{ytableau}
{}\\
\bullet\\
{}\\
\bullet\\
{}
\end{ytableau}
\\[12.5ex] 
\hline
$((2),\varnothing)$ & $(\varnothing,(1^2))$ & $((1),(1))$ & $((2),\varnothing)$ & $(\varnothing,(1^2))$\\
\end{tabular}}
\end{center}

\subsection{Expressions}
\subsubsection{}\label{Express-Mac}
The wreath Macdonald polynomials of size $1$ can be directly computed using the definition. We find:
\begingroup
\allowdisplaybreaks
\begin{align}
\tilde{H}_{((1),\varnothing)}&=s_{((1),\varnothing)}+zs_{(\varnothing,(1))},\\
\tilde{H}_{(\varnothing,(1))}&=s_{((1),\varnothing)}+ws_{(\varnothing,(1))}.
\end{align}
\endgroup
Similarly, we can compute the Macdonald polynomials of degree 2:
\begingroup
\allowdisplaybreaks
\begin{align}
H_{(2)}&=s_{(2)}+zs_{(1^2)},\\
H_{(1^2)}&=s_{(2)}+ws_{(1^2)}.
\end{align}
\endgroup

The $z,w$-inner products of these wreath Macdonald polynomials are given by 
\begingroup
\allowdisplaybreaks
\begin{align}
\langle \tilde{H}_{((1),\varnothing)},\tilde{H}_{((1),\varnothing)}\rangle_{z,w}&=(z^2-1)(z-w),\\
\langle \tilde{H}_{(\varnothing,(1))},\tilde{H}_{(\varnothing,(1))}\rangle_{z,w}&=(w^2-1)(w-z).
\end{align}
\endgroup
And those of Macdonald polynomials are given by
\begingroup
\allowdisplaybreaks
\begin{align}
\langle H_{(2)},H_{(2)}\rangle_{z,w}&=(z^2-1)(z-w)(z-1)(1-w),\\
\langle H_{(1^2)},H_{(1^2)}\rangle_{z,w}&=(w^2-1)(w-z)(z-1)(1-w).
\end{align}
\endgroup
Note that these $z,w$-inner products differ from the corresponding products of wreath Macdonald polynomials by $(z-1)(1-w)$.

We will need to compute the Hall inner products of the (wreath) Macdonald polynomials against the complete symmetric functions. It suffices to notice
\begingroup
\allowdisplaybreaks
\begin{align}
h_{((1),\varnothing)}=s_{((1),\varnothing)},\quad &h_1(\Xo+\Xl)=s_{((1),\varnothing)}(\X)+s_{(\varnothing,(1))}(\X),\\
h_2=s_{(2)},\quad &h_1^2=s_{(2)}+s_{(1^2)}.
\end{align}
\endgroup

\subsubsection{}
We list below the wreath Macdonald polynomials of size 2 and their self-pairings. Observe that the coefficients can be a Laurent polynomial, and that in the self-pairings there are "wreath-$\nabla$" terms.

2-core $=(0)$:
\begingroup
\allowdisplaybreaks
\begin{align}
\tilde{H}_{((2),\varnothing)}=&s_{((2),\varnothing)}+z^2s_{(\varnothing,(2))}+z(z^2+1)s_{((1),(1))}+z^2s_{((1^2),\varnothing)}+z^4s_{(\varnothing,(1^2))}\\
\tilde{H}_{(\varnothing,(2))}=&s_{((2),\varnothing)}+wz^{-1}s_{(\varnothing,(2))}+(z+w)s_{((1),(1))}+z^2s_{((1^2),\varnothing)}+zws_{(\varnothing,(1^2))}\\
\tilde{H}_{((1),(1))}=&s_{((2),\varnothing)}+s_{(\varnothing,(2))}+(z+w)s_{((1),(1))}+zws_{((1^2),\varnothing)}+zws_{(\varnothing,(1^2))}\\
\tilde{H}_{((1^2),\varnothing)}=&s_{((2),\varnothing)}+zw^{-1}s_{(\varnothing,(2))}+(z+w)s_{((1),(1))}+w^2s_{((1^2),\varnothing)}+zws_{(\varnothing,(1^2))}\\
\tilde{H}_{(\varnothing,(1^2))}=&s_{((2),\varnothing)}+w^2s_{(\varnothing,(2))}+w(w^2+1)s_{((1),(1))}+w^2s_{((1^2),\varnothing)}+w^4s_{(\varnothing,(1^2))}
\end{align}
\endgroup

\begingroup
\allowdisplaybreaks
\begin{align}
\tilde{N}_{((2),\varnothing)}=&(z^4-1)(z^3-w)(z^2-1)(z-w)\\
\tilde{N}_{(\varnothing,(2))}=&z^{-2}(z^3-w)(z^2-w^2)(z^2-1)(z-w)\\
\tilde{N}_{((1),(1))}=&(z-w)(1-w^2)(z^2-1)(z-w)\\
\tilde{N}_{((1^2),\varnothing)}=&w^{-2}(z^2-w^2)(z-w^3)(z-w)(1-w^2)\\
\tilde{N}_{(\varnothing,(1^2))}=&(z-w^3)(1-w^4)(z-w)(1-w^2)
\end{align}
\endgroup

2-core $(1)$:
\begingroup
\allowdisplaybreaks
\begin{align}
\tilde{H}_{((2),\varnothing)}=&s_{((2),\varnothing)}+z^2s_{((1^2),\varnothing)}+z(z^2+1)s_{((1),(1))}+z^2s_{(\varnothing,(2))}+z^4s_{(\varnothing,(1^2))}\\
\tilde{H}_{((1^2),\varnothing)}=&s_{((2),\varnothing)}+wz^{-1}s_{((1^2),\varnothing)}+(z+w)s_{((1),(1))}+z^2s_{(\varnothing,(2))}+zws_{(\varnothing,(1^2))}\\
\tilde{H}_{((1),(1))}=&s_{((2),\varnothing)}+s_{((1^2),\varnothing)}+(z+w)s_{((1),(1))}+zws_{(\varnothing,(2))}+zws_{(\varnothing,(1^2))}\\
\tilde{H}_{(\varnothing,(2))}=&s_{((2),\varnothing)}+zw^{-1}s_{((1^2),\varnothing)}+(z+w)s_{((1),(1))}+w^2s_{(\varnothing,(2))}+zws_{(\varnothing,(1^2))}\\
\tilde{H}_{(\varnothing,(1^2))}=&s_{((2),\varnothing)}+w^2s_{((1^2),\varnothing)}+w(w^2+1)s_{((1),(1))}+w^2s_{(\varnothing,(2))}+w^4s_{(\varnothing,(1^2))}
\end{align}
\endgroup
The self-pairings are given by the same formulae as in the case of 2-core $(0)$, except that $\tilde{N}_{((2),\varnothing)}$ and $\tilde{N}_{((1^2),\varnothing)}$ should be exchanged.

%%%%%%%%%%%%%%%%%%%%%%
\section{Conjectural Mixed Hodge Polynomial; $n=1$ and $n=2$}\label{App-12}
Our formula for the E-polynomial suggests that a two-parameter series $\Omega_{e}(z,w)$, $e=0$ or $1$, should roughly be of this form (See (\ref{Euler-Omega-eq1}) and (\ref{Euler-Omega-eq2}))$$\sum_{\bs\alpha\in\mathcal{P}^2}\frac{N_{\{\bs\alpha\}_1}(z^2,w^2)^{g+k-1}}{\tilde{N}_{\bs\alpha}(z^2,w^2)^k}\prod_{j=1}^{2k}\tilde{H}_{\bs\alpha}(\mathbf{x}_j;z^2,w^2).$$But the E-polynomial is unable to distinguish a self-pairing from its deformed version. The computations in this appendix suggest that there should be $g+k-1$ deformed self-pairings in the numerator, and $k-1$ deformed self-pairings in the denominator, with the remaining one undeformed. This explains our definition in \S \ref{subsec-series}.

We work with arbitrary $g$ and $k>0$.

\subsection{Character Varieties with $n=1$ and $n=2$.}\hfill

We first give explicit descriptions of the character varieties when $n=1$ and $n=2$.
\subsubsection{}\label{Ch-rank-1}
When $n=1$, $\GL_1\lb\sigma\rb$ is just the multiplicative group with $\sigma$ acting as the inversion. In this case, there is only one conjugacy class contained in $\GL_1\sigma$: the connected component itself. By definition, the tuple of conjugacy classes $\mathcal{C}=(C_j)_{1\le j\le 2k}$ with each $C_j=\GL_1\sigma$ is generic, since there is no proper parabolic subgroup except the trivial one. We claim that $$\Ch_{\mathcal{C}}\cong\mathbb{G}_m^{2(g+k-1)}.$$The defining equation of the character variety reads:$$1=\prod_i^g[A_i,B_i]\prod_j^{2k}X_j\sigma=X_1X_2^{-1}X_3X_4^{-1}\cdots X_{2k-1}X_{2k}^{-1}$$with all $A_i$, $B_i$ and $X_j$ lying in $\mathbb{G}_m$. The conjugation action of $t\in\mathbb{G}_m$ on the coordinates is given by:$$t:(A_i,B_i)_i(X_j)_j\longmapsto(A_i,B_i)_i(t^2X_j)_j.$$The quotient is a $(2g+2k-2)$-dimensional torus.

\subsubsection{}
Now let $n=2$. We will give a relation between $\GL_2\lb\sigma\rb~$-character varieties and $\GL_2$-character varieties.

The centraliser of $\sigma$ in $\GL_2$ is $\SL_2$. Therefore the $\GL_2$-conjugacy class of $\sigma$ is the variety $\mathbb{G}_m\sigma$, where elements of $\mathbb{G}_m$ are regarded as scalar matrices. In fact, we can see that each conjugacy class $C\subset\GL_2\sigma$ is of the form $\mathbb{G}_mC'\sigma$, for some $\SL_2$-conjugacy class $C'\subset\SL_2$. The semi-simple conjugacy classes in $\GL_2\sigma$ have representatives of the form $$s\sigma=\diag(a,a^{-1})\sigma.$$ We require that $a^2\ne-1$. Then there are two possibilities. Either $s\sigma$ is conjugate to $\sigma$, or $a^4\ne1$, in which case $$C_{\GL_2}(s\sigma)=\{\diag(x,x^{-1})\mid x\in\mathbb{G}_m\}\cong\mathbb{G}_m.$$ These two kinds of conjugacy classes in $\GL_2\sigma$ correspond to the conjugacy classes of $1$ and of $\diag(a,a^{-1})$ in $\SL_2$ respectively.

The inverse image of a conjugacy class $C\subset\GL_2\sigma$ under the double covering $\mathbb{G}_m\times\SL_2\rightarrow\GL_2\cong\GL_2\sigma$ is $(\mathbb{G}_m\times C')\sqcup(\mathbb{G}_m\times -C')$ for some conjugacy class $C'\subset\SL_2$. Given a tuple of semi-simple conjugacy classes $\mathcal{C}=(C_j)_{1\le j\le 2k}$ in $\GL_2\sigma$, we denote by $\mathcal{C}'=(C'_j)_{1\le j\le 2k}$ a tuple of conjugacy classes of $\SL_2$ that gives $\mathcal{C}$ in the above fashion. Note that each $\SL_2$-conjugacy class is also a $\GL_2$-conjugacy class. Thus it makes sense to consider the varieties: $$\Rep_{\mathcal{C}'}^{\pm}:=\left\{(A_i,B_i)_i\times (X'_j)_j\in \GL_2^{2g}\times \prod_{j=1}^{2k} C'_j\mid\prod_{i=1}^g[A_i,B_i]\prod_{j=1}^{2k}X'_j=\pm1\right\},$$and $\Ch_{\mathcal{C}'}^{\pm}:=\Rep_{\mathcal{C}'}^{\pm}\ds\GL_2$. Define the $(2k-1)$-dimensional tori: $$T^{\pm}:=\{(t_1,\ldots,t_{2k})\in\mathbb{G}_m^{2k}\mid t_1t_2^{-1}t_3t_4^{-1}\cdots t_{2k-1}t_{2k}^{-1}=\pm 1\}.$$

We have an isomorphism:
\begingroup
\allowdisplaybreaks
\begin{align*}
(\Rep_{\mathcal{C}'}^+\times T^+)\sqcup(\Rep_{\mathcal{C}'}^-\times T^-)&\lisom\Rep_{\mathcal{C}}\\
(A_i,B_i)_i(X'_j)_j(t_j)_j&\longmapsto(A_i,B_i)_i(t_jX'_j\sigma)_j.
\end{align*}
\endgroup
Indeed, if $(A_i,B_i)_i(X_j)_j$ lies in $\Rep_{\mathcal{C}}$, then each $X_j$ can be written as $t_jX'_j\sigma$ for a unique $X'_j\in C'_j$ and a $t_j\in\mathbb{G}_m$, and we have $$\prod_i[A_i,B_i]\prod_jX'_j(t_1t_2^{-1}\cdots t_{2k-1}t_{2k}^{-1})=1.$$Since $\SL_2\cap \mathbb{G}_m=\{\pm 1\}$, the tuple $(A_i,B_i)_i(X'_j)_j(t_j)_j$ lies in $\Rep_{\mathcal{C}'}^+\times T^+\sqcup\Rep_{\mathcal{C}'}^-\times T^-$. This isomorphism is equivariant for the conjugation action of $\GL_2$ on $\Rep_{\mathcal{C}}$ and the following action on $\Rep_{\mathcal{C}'}^{\pm}\times T^{\pm}$:
$$g:(A_i,B_i)_i(X'_j)_j(t_j)_j\longmapsto(gA_ig^{-1},gB_ig^{-1})_i(gX'_jg^{-1})_j(\det(g)t_j)_j.$$Therefore the quotient by $\GL_2$ is isomorphic to $\Ch_{\mathcal{C}'}^{\pm}\times\mathbb{G}_m^{2k-2}$. We conclude that there is an isomorphism
\begin{equation}\label{Ch-rank-2}
(\Ch_{\mathcal{C}'}^+\sqcup\Ch_{\mathcal{C}'}^-)\times \mathbb{G}_m^{2k-2}\lisom\Ch_{\mathcal{C}}.
\end{equation}
If $\mathcal{C}$ is generic, then it follows from the definition of generic conjugacy classes in $\GL_2\sigma$ that both of $\Ch_{\mathcal{C}'}^{\pm}$ have generic conjugacy classes.
\begin{Rem}\label{conn-Ch-rank-2}
If $C_j$ is the conjugacy class of $\diag(a_j,a_j^{-1})\sigma$ for some $a_j^2=-1$, then $C'_j=-C'_j$ and (\ref{Ch-rank-2}) is no longer an isomorphism, but a double covering with the action of $\mathbb{Z}/2\mathbb{Z}$ permuting the two connected components. It follows that $\Ch_{\mathcal{C}}$ is connected.
\end{Rem}

\subsection{Mixed Hodge Polynomials with $n=1$ and $n=2$}
\subsubsection{}
By \S \ref{Ch-rank-1}, we have $$\Ch_{\mathcal{C}}\cong (\mathbb{C}^{\ast})^{2(g+k-1)}$$ and so $$H_c(\Ch_{\mathcal{C}};q,t)=(t+qt^2)^{2(g+k-1)},$$by \cite[Example 2.1.11]{HRV} and the K\"unneth isomorphism for mixed Hodge structures. On the other hand, according to Conjecture \ref{The-Conj}, we have
$$H_c(\Ch_{\mathcal{C}};q,t)=(t\sqrt{q})^{2(g+k-1)}\left\langle\frac{\Omega_{1}\Omega_{0}}{\Omega_{\ast}},\prod^{2k}_{j=1}h_{\bs\beta^{\ast}_j}\right\rangle.$$Now $h_{\bs\beta^{\ast}_j}$ has degree 0, therefore only the constant term on the left hand side of the inner product contributes. The constant terms of $\Omega_{\ast}$ and $\Omega_0$ are 1. As for $\Omega_1$, we first note that $\tilde{N}_{\bs\alpha}=1$ for $\bs\alpha=(\varnothing,\varnothing)$. Since $\{\bs\alpha\}_1=(1)$, we have $N_{\{\bs\alpha\}_1}(zw,z^2,w^2)=(z-w)^2$. We find$$(t\sqrt{q})^d\mathbb{H}_{\mathbf{B}}(-t\sqrt{q},\frac{1}{\sqrt{q}})=(t\sqrt{q})^{2(g+k-1)}(t\sqrt{q}+\frac{1}{\sqrt{q}})^{2(g+k-1)}=(t+qt^2)^{2(g+k-1)},$$which is the correct mixed Hodge polynomial.

\subsubsection{}
Now we consider the case $n=2$. We are unable to compute the mixed Hodge polynomial of arbitrary $\GL_2\lb\sigma\rb~$-character varieties. Instead, we use (\ref{Ch-rank-2}):$$(\Ch_{\mathcal{C}'}^+\sqcup\Ch_{\mathcal{C}'}^-)\times \mathbb{G}_m^{2k-2}\lisom\Ch_{\mathcal{C}}$$ and show that under this isomorphism, Conjecture \ref{The-Conj} and Conjecture \ref{Conj-HLRV} predict the same mixed Hodge polynomial. In fact, we will show that they give the same rational function in $z$ and $w$ before setting $z=-t\sqrt{q}$ and $w=\sqrt{q}^{-1}$.

By Conjecture \ref{Conj-HLRV}, the two character varieties $\Ch_{\mathcal{C}'}^+$ and $\Ch_{\mathcal{C}'}^-$ have identical mixed Hodge polynomials. Therefore, we only need to consider $\Ch_{\mathcal{C}'}^+\times \mathbb{G}_m^{2k-2}$ and then multiply the result by $2$. Now we compute
$$
(z^2-1)(1-w^2)\Log\Omega_{HLRV}(z,w)
$$ following the procedure of \cite[\S 2.3.3]{HLR}. We get a sum of four terms:
\begingroup
\allowdisplaybreaks
\begin{align*}
\mathbb{H}_2:=&\frac{[(z^3-w)(z-w)]^{2g}}{(z^4-1)(z^2-w^2)}\prod_{j=1}^{2k}H_{(2)}(\X_j;z^2,w^2),\\%%%
\mathbb{H}_{(1^2)}:=&\frac{[(z-w^3)(z-w)]^{2g}}{(z^2-w^2)(1-w^4)}\prod_{j=1}^{2k}H_{(1^2)}(\X_j;z^2,w^2),\\%%%
\mathbb{H}_1:=&-\frac{1}{2}\frac{(z-w)^{4g}}{(z^2-1)(1-w^2)}\prod_{j=1}^{2k}H_{(1)}(\X_j;z^2,w^2)^2,\\%%%
\mathbb{H}_1':=&-\frac{1}{2}\frac{(z^2-w^2)^{2g}}{(z^2+1)(1+w^2)}\prod_{j=1}^{2k}H_{(1)}(\X_j^2;z^4,w^4).
\end{align*}
\endgroup
In fact, the rational functions in front of the Macdonald polynomials were computed in \cite[\S 4.3]{HRV}. From $$H_{(1)}(\X_j^2;z^4,w^4)=p_2(\X_j)=s_{(2)}(\X_j)-s_{(1^2)}(\X_j),$$ we see that it is orthogonal to $h_1^2=s_{(2)}+s_{(1^2)}$, the symmetric function representing the regular semi-simple conjugacy class. According to the definition of the generic condition for $\mathcal{C}$ (not that for $\mathcal{C}'$), and the relation between $\mathcal{C}$ and $\mathcal{C}'$, there exists at least one regular semi-simple conjugacy class among $\mathcal{C}'$. Therefore, the term $\mathbb{H}_1'$ must vanish after taking the inner product.

Then we compute $$\frac{\Omega_{0}(z,w)^2}{\Omega_{\ast}(z,w)}.$$Write
\begingroup
\allowdisplaybreaks
\begin{align*}
\mathbf{H}_{((1),\varnothing)}:=&\frac{(z^3-w)^{2g}(z-w)^{2(g+k-1)}}{(z^4-1)(z^2-w^2)}\prod_{j=1}^{2k}\tilde{H}_{((1),\varnothing)}(\X_j;z^2,w^2),\\
\mathbf{H}_{(\varnothing,(1))}:=&\frac{(z-w^3)^{2g}(z-w)^{2(g+k-1)}}{(z^2-w^2)(1-w^4)}\prod_{j=1}^{2k}\tilde{H}_{(\varnothing,(1))}(\X_j;z^2,w^2),\\
\mathbf{H}_1:=&\frac{(z-w)^{2(2g+k-1)}}{(z^2-1)(1-w^2)}\prod_{j=1}^{2k}\tilde{H}_{(1)}(\mathbf{x}^{(0)}_j+\mathbf{x}^{(1)}_j;z^2,w^2),
\end{align*}
\endgroup
and 
$$
U=\mathbf{H}_{((1),\varnothing)}+\mathbf{H}_{(\varnothing,(1))},\quad U'=\mathbf{H}_1.
$$
Then 
\begingroup
\allowdisplaybreaks
\begin{align*}
\Omega_0^2&=1+2U+\text{higher degree terms},\\
\frac{1}{\Omega_{\ast}}&=1-U'+\text{higher degree terms},
\end{align*}
\endgroup

The contribution of the torus $\mathbb{G}_m^{2k-2}$ is $(z-w)^{2k-2}$. (Strictly speaking, there is an extra factor $(t\sqrt{q})^{2k-2}$, but this will be balanced by the difference between $\dim\Ch_{\mathcal{C}'}^+$ and $\dim\Ch_{\mathcal{C}}$.) Now we can readily see that $(z-w)^{2k-2}\mathbb{H}_2$ and $\mathbf{H}_{((1),\varnothing)}$ have the same rational function in front of the Macdonald polynomials. The same conclusion holds when matching the functions:
\begingroup
\allowdisplaybreaks
\begin{align*}
(z-w)^{2k-2}\mathbb{H}_{(1^2)}&\longleftrightarrow\mathbf{H}_{(\varnothing,(1))}\\
-2(z-w)^{2k-2}\mathbb{H}_1&\longleftrightarrow\mathbf{H}_1.
\end{align*}
\endgroup

According to \S \ref{Express-Mac}, under the correspondences:
\begingroup
\allowdisplaybreaks
\begin{align*}
H_{(2)}(\Z_j;z^2,w^2)&\longleftrightarrow\tilde{H}_{((1),\varnothing)}(\X_j;z^2,w^2),\\
H_{(1^2)}(\Z_j;z^2,w^2)&\longleftrightarrow\tilde{H}_{(\varnothing,(1))}(\X_j;z^2,w^2),\\
H_{(1)}(\Z_j;z^2,w^2)^2&\longleftrightarrow\tilde{H}_{(1)}(\mathbf{x}^{(0)}_j+\mathbf{x}^{(1)}_j;z^2,w^2),\\
h_2(\Z_j)&\longleftrightarrow h_{((1),\varnothing)}(\X_j),\\
h_{(1^2)}(\Z_j)&\longleftrightarrow h_1(\mathbf{x}^{(0)}_j+\mathbf{x}^{(1)}_j),
\end{align*}
\endgroup
the inner products are preserved. Therefore, we get the same functions in $z$ and $w$.

%%%%%%%%%%%%%%%%%%%%%%
\section{Conjectural Mixed Hodge Polynomial; $n=4$ and $n=5$}\label{App-45}
We give two sample computations of the rational function $\mathbb{H}_{\mathbf{B}}(z,w)$ in Conjecture \ref{The-Conj}. We will see that they indeed satisfy part (i) of Conjecture \ref{The-Conj}. These are the simplest cases where appear the deformed self-pairings of wreath Macdonald polynomials of degree 2 and 2-core $(0)$ or $(1)$. We will see the contributions of the "wreath-$\nabla$" terms, which are absent if $n<4$. This is our evidence that the "wreath-$\nabla$" terms should not be deformed (See \S \ref{w-q,t-inner}).

\subsection{$n=4$, $g=0$, $k=2$}\hfill

There are four conjugacy classes: $\mathcal{C}=(C_1,C_2,C_3,C_4)$. We assume that $C_1$ and $C_2$ are regular, $C_3$ and $C_4$ are the conjugacy class of $\sigma$.

\begingroup
\allowdisplaybreaks
\begin{align*}
&\mathbb{H}_{\mathbf{B}}(z,w)\\
=&-16\frac{(z^3-w)^2(z-w)^2(z^2+1)}{(z^2-1)^2(z^2-w^2)(1-w^2)}
-16\frac{(z-w^3)^2(z-w)^2(1+w^2)}{(1-w^2)^2(z^2-w^2)(z^2-1)}
+64\frac{(z-w)^4}{(z^2-1)^2(1-w^2)^2}\\
&-32\frac{(z-w)^4(1+z^2)}{(z^2-1)^2(z^2-w^2)(1-w^2)}
-32\frac{(z-w)^4(1+w^2)}{(1-w^2)^2(z^2-w^2)(z^2-1)}\\
&+4\frac{(z-w)^4(z^2+1)^2}{(z^2-1)^2(z^2-w^2)^2}
+4\frac{(z-w)^4(1+w^2)^2}{(1-w^2)^2(z^2-w^2)^2}
+8\frac{(z-w)^4(z^2+1)(1+w^2)}{(z^2-1)(1-w^2)(z^2-w^2)^2}\\
&+2\frac{(z^5-w)^2(z-w)^2(z^4+1)(z^2+1)^2}{(z^6-w^2)(z^2-w^2)(z^2-1)^2}+2\frac{(z-w^5)^2(z-w)^2(1+w^4)(1+w^2)^2}{(z^2-w^6)(z^2-w^2)(1-w^2)^2}\\
&+2z^4\frac{(z-w)^4(z^2+w^2)(z^2+1)^3}{(z^6-w^2)(z^2-w^2)^2(z^2-1)}
+2w^4\frac{(z-w)^4(z^2+w^2)(1+w^2)^3}{(z^2-w^6)(z^2-w^2)^2(1-w^2)}\\
&+8\frac{(z^3-w^3)^2(z-w)^2(z^2+1)(1+w^2)}{(z^2-w^2)^2(z^2-1)(1-w^2)}\\
=&~2\,w^8-4\,w^7\,z+4\,w^6\,z^2+8\,w^6-4\,w^5\,z^3-16\,w^5\,z\\
&+4\,w^4\,z^4+16\,w^4\,z^2+12\,w^4-4\,w^3\,z^5-16\,w^3\,z^3-32\,w^3\,z\\
&+4\,w^2\,z^6+16\,w^2\,z^4+40\,w^2\,z^2+8\,w^2-4\,w\,z^7-16\,w\,z^5-32\,w\,z^3-16\,w\,z\\
&+2\,z^8+8\,z^6+12\,z^4+8\,z^2.
\end{align*}
\endgroup

\subsection{$n=5$, $g=0$, $k=2$}\hfill

There are four conjugacy classes: $\mathcal{C}=(C_1,C_2,C_3,C_4)$. We assume that $C_1$ is regular, $C_3$ and $C_4$ are the conjugacy classes of $\sigma$, and $C_2$ is the conjugacy class of $\diag(a,1,1,a^{-1})\sigma$ with some $a^4\ne1$.

\begingroup
\allowdisplaybreaks
\begin{align*}
&\mathbb{H}_{\mathbf{B}}(z,w)\\
=&-4\frac{(z-w)^4\left((z^5-w)^2+(z-w)^2\right)(3+z^2)}{(z^2-1)^2(1-w^2)(z^2-w^2)}
-4\frac{(z-w)^4\left((z-w^5)^2+(z-w)^2\right)(3+w^2)}{(z^2-1)(1-w^2)^2(z^2-w^2)}\\
&+4\frac{(z-w)^4(z^5-w)^2(1+z^2)}{(z^2-1)^2(z^2-w^2)^2}
+4\frac{(z-w)^4(z-w^5)^2(1+w^2)}{(z^2-w^2)^2(1-w^2)^2}\\
&+2\frac{(z-w)^4\left((z-w^5)^2+(z^5-w)^2\right)(2+z^2+w^2)}{(z^2-1)(z^2-w^2)^2(1-w^2)}+32\frac{(z-w)^6}{(z^2-1)^2(1-w^2)^2}\\
&-8\frac{(z-w)^4(z^3-w)^2(1+z^2)}{(z^2-w^2)(z^2-1)^2(1-w^2)}
-8\frac{(z-w)^4(z-w^3)^2(1+w^2)}{(z^2-w^2)(z^2-1)(1-w^2)^2}\\
&+\frac{(z-w)^2(z^5-w)^2\left((z^9-w)^2+(z-w)^2\right)(1+z^2)(1+z^4)}{(z^6-w^2)(z^2-1)^2(z^2-w^2)}\\
&+\frac{(z-w)^2(z-w^5)^2\left((z-w^9)^2+(z-w)^2\right)(1+w^2)(1+w^4)}{(z^2-w^6)(z^2-w^2)(1-w^2)^2}\\
&+z^4\frac{(1+z^2)(z-w)^4\left((1+z^2)(z^2+w^2)(z^3-w^3)^2+z^2(1+z^2+w^2+z^4)(z-w)^2\right)}{(z^6-w^2)(z^2-1)(z^2-w^2)^2}\\
&+w^4\frac{(1+w^2)(z-w)^4\left((1+w^2)(z^2+w^2)(z^3-w^3)^2+w^2(1+z^2+w^2+w^4)(z-w)^2\right)}{(z^2-w^6)(z^2-w^2)^2(1-w^2)}\\
&+2\frac{(z-w)^4\left((2+z^2+w^2)(z^5-w^5)^2+(1+z^2)(1+w^2)(z^3-w^3)^2\right)}{(z^2-w^2)^2(1-w^2)(z^2-1)}.
\end{align*}
\endgroup

The result is a polynomial:

$\mathbb{H}_{\mathbf{B}}(z,w)=w^{24}-2\,w^{23}\,z+2\,w^{22}\,z^2+3\,w^{22}-2\,w^{21}\,z^3-6\,w^{21}\,z+2\,w^{20}\,z^4+6\,w^{20}\,z^2+6\,w^{20}-2\,w^{19}\,z^5-6\,w^{19}\,z^3-14\,w^{19}\,z+2\,w^{18}\,z^6+6\,w^{18}\,z^4+17\,w^{18}\,z^2+10\,w^{18}-2\,w^{17}\,z^7-6\,w^{17}\,z^5-18\,w^{17}\,z^3-26\,w^{17}\,z+2\,w^{16}\,z^8+6\,w^{16}\,z^6+18\,w^{16}\,z^4+35\,w^{16}\,z^2+14\,w^{16}-2\,w^{15}\,z^9-6\,w^{15}\,z^7-18\,w^{15}\,z^5-38\,w^{15}\,z^3-42\,w^{15}\,z+2\,w^{14}\,z^{10}+6\,w^{14}\,z^8+18\,w^{14}\,z^6+38\,w^{14}\,z^4+63\,w^{14}\,z^2+18\,w^{14}-2\,w^{13}\,z^{11}-6\,w^{13}\,z^9-18\,w^{13}\,z^7-38\,w^{13}\,z^5-72\,w^{13}\,z^3-62\,w^{13}\,z+2\,w^{12}\,z^{12}+6\,w^{12}\,z^{10}+18\,w^{12}\,z^8+38\,w^{12}\,z^6+75\,w^{12}\,z^4+101\,w^{12}\,z^2+22\,w^{12}-2\,w^{11}\,z^{13}-6\,w^{11}\,z^{11}-18\,w^{11}\,z^9-38\,w^{11}\,z^7-76\,w^{11}\,z^5-120\,w^{11}\,z^3-84\,w^{11}\,z+2\,w^{10}\,z^{14}+6\,w^{10}\,z^{12}+18\,w^{10}\,z^{10}+38\,w^{10}\,z^8+76\,w^{10}\,z^6+129\,w^{10}\,z^4+148\,w^{10}\,z^2+26\,w^{10}-2\,w^9\,z^{15}-6\,w^9\,z^{13}-18\,w^9\,z^{11}-38\,w^9\,z^9-76\,w^9\,z^7-132\,w^9\,z^5-186\,w^9\,z^3-108\,w^9\,z+2\,w^8\,z^{16}+6\,w^8\,z^{14}+18\,w^8\,z^{12}+38\,w^8\,z^{10}+76\,w^8\,z^8+132\,w^8\,z^6+207\,w^8\,z^4+204\,w^8\,z^2+28\,w^8-2\,w^7\,z^{17}-6\,w^7\,z^{15}-18\,w^7\,z^{13}-38\,w^7\,z^{11}-76\,w^7\,z^9-132\,w^7\,z^7-216\,w^7\,z^5-270\,w^7\,z^3-124\,w^7\,z+2\,w^6\,z^{18}+6\,w^6\,z^{16}+18\,w^6\,z^{14}+38\,w^6\,z^{12}+76\,w^6\,z^{10}+132\,w^6\,z^8+218\,w^6\,z^6+306\,w^6\,z^4+245\,w^6\,z^2+23\,w^6-2\,w^5\,z^{19}-6\,w^5\,z^{17}-18\,w^5\,z^{15}-38\,w^5\,z^{13}-76\,w^5\,z^{11}-132\,w^5\,z^9-216\,w^5\,z^7-316\,w^5\,z^5-326\,w^5\,z^3-110\,w^5\,z+2\,w^4\,z^{20}+6\,w^4\,z^{18}+18\,w^4\,z^{16}+38\,w^4\,z^{14}+75\,w^4\,z^{12}+129\,w^4\,z^{10}+207\,w^4\,z^8+306\,w^4\,z^6+354\,w^4\,z^4+228\,w^4\,z^2+11\,w^4-2\,w^3\,z^{21}-6\,w^3\,z^{19}-18\,w^3\,z^{17}-38\,w^3\,z^{15}-72\,w^3\,z^{13}-120\,w^3\,z^{11}-186\,w^3\,z^9-270\,w^3\,z^7-326\,w^3\,z^5-282\,w^3\,z^3-44\,w^3\,z+2\,w^2\,z^{22}+6\,w^2\,z^{20}+17\,w^2\,z^{18}+35\,w^2\,z^{16}+63\,w^2\,z^{14}+101\,w^2\,z^{12}+148\,w^2\,z^{10}+204\,w^2\,z^8+245\,w^2\,z^6+228\,w^2\,z^4+66\,w^2\,z^2+w^2-2\,w\,z^{23}-6\,w\,z^{21}-14\,w\,z^{19}-26\,w\,z^{17}-42\,w\,z^{15}-62\,w\,z^{13}-84\,w\,z^{11}-108\,w\,z^9-124\,w\,z^7-110\,w\,z^5-44\,w\,z^3-2\,w\,z+z^{24}+3\,z^{22}+6\,z^{20}+10\,z^{18}+14\,z^{16}+18\,z^{14}+22\,z^{12}+26\,z^{10}+28\,z^8+23\,z^6+11\,z^4+z^2.$

\subsection*{Conflicts of Interest} None.

\subsection*{Financial Support} None.

\addtocontents{toc}{\protect\setcounter{tocdepth}{-1}}
\bibliographystyle{alpha}
\bibliography{BIB}

\begin{thebibliography}{HMMS22}

\bibitem[BF14]{BF}
Roman Bezrukavnikov and Michael Finkelberg.
\newblock Wreath {M}acdonald polynomials and the categorical {M}c{K}ay
  correspondence.
\newblock {\em Camb. J. Math.}, 2(2):163--190, 2014.
\newblock With an appendix by Vadim Vologodsky.

\bibitem[BY15]{BY}
P.~P. Boalch and Daisuke Yamakawa.
\newblock Twisted wild character varieties.
\newblock {\em arXiv}, 2015.

\bibitem[Cam17]{C}
Vinvenzo Camb\`o.
\newblock On the {E}-polynomial of parabolic $\text{Sp}_{2n}~$-character
  varieties.
\newblock {\em arXiv}, 2017.

\bibitem[dCHM12]{dCHM}
Mark Andrea~A. de~Cataldo, Tam\'{a}s Hausel, and Luca Migliorini.
\newblock Topology of {H}itchin systems and {H}odge theory of character
  varieties: the case {$A_1$}.
\newblock {\em Ann. of Math. (2)}, 175(3):1329--1407, 2012.

\bibitem[Del71]{D1}
Pierre Deligne.
\newblock Th\'eorie de {Hodge} : {II}.
\newblock {\em Publications Math\'ematiques de l'IH\'ES}, 40:5--57, 1971.

\bibitem[Del74]{D2}
Pierre Deligne.
\newblock Th\'eorie de {Hodge} : {III}.
\newblock {\em Publications Math\'ematiques de l'IH\'ES}, 44:5--77, 1974.

\bibitem[DM94]{DM94}
Fran\c{c}ois Digne and Jean Michel.
\newblock Groupes r\'{e}ductifs non connexes.
\newblock {\em Ann. Sci. \'{E}cole Norm. Sup. (4)}, 27(3):345--406, 1994.

\bibitem[DM15]{DM15}
Fran\c{c}ois Digne and Jean Michel.
\newblock Complements on disconnected reductive groups.
\newblock {\em Pacific J. Math.}, 279(1-2):203--228, 2015.

\bibitem[DM18]{DM18}
Fran\c{c}ois Digne and Jean Michel.
\newblock Quasi-semisimple elements.
\newblock {\em Proc. Lond. Math. Soc. (3)}, 116(5):1301--1328, 2018.

\bibitem[DM20]{DM20}
Fran\c{c}ois Digne and Jean Michel.
\newblock {\em Representations of finite groups of {L}ie type}, volume~95 of
  {\em London Mathematical Society Student Texts}.
\newblock Cambridge University Press, Cambridge, 2020.

\bibitem[Hai03]{Hai}
Mark Haiman.
\newblock Combinatorics, symmetric functions, and {H}ilbert schemes.
\newblock In {\em Current developments in mathematics, 2002}, pages 39--111.
  Int. Press, Somerville, MA, 2003.

\bibitem[HLRV11]{HLR}
Tam\'{a}s Hausel, Emmanuel Letellier, and Fernando Rodriguez-Villegas.
\newblock Arithmetic harmonic analysis on character and quiver varieties.
\newblock {\em Duke Math. J.}, 160(2):323--400, 2011.

\bibitem[HMMS22]{HMMS}
Tamas Hausel, Anton Mellit, Alexandre Minets, and Olivier Schiffmann.
\newblock ${P}={W}$ via ${H}_2$.
\newblock {\em arXiv}, 2022.

\bibitem[HRV08]{HRV}
Tam\'{a}s Hausel and Fernando Rodriguez-Villegas.
\newblock Mixed {H}odge polynomials of character varieties.
\newblock {\em Invent. Math.}, 174(3):555--624, 2008.
\newblock With an appendix by Nicholas M. Katz.

\bibitem[HT03]{HT}
Tam\'{a}s Hausel and Michael Thaddeus.
\newblock Mirror symmetry, {L}anglands duality, and the {H}itchin system.
\newblock {\em Invent. Math.}, 153(1):197--229, 2003.

\bibitem[LN08]{LN}
G\'{e}rard Laumon and Bao~Ch\^{a}u Ng\^{o}.
\newblock Le lemme fondamental pour les groupes unitaires.
\newblock {\em Ann. of Math. (2)}, 168(2):477--573, 2008.

\bibitem[LS77]{LS}
George Lusztig and Bhama Srinivasan.
\newblock The characters of the finite unitary groups.
\newblock {\em J. Algebra}, 49(1):167--171, 1977.

\bibitem[LS12]{LiSe}
Martin~W. Liebeck and Gary~M. Seitz.
\newblock {\em Unipotent and nilpotent classes in simple algebraic groups and
  {L}ie algebras}, volume 180 of {\em Mathematical Surveys and Monographs}.
\newblock American Mathematical Society, Providence, RI, 2012.

\bibitem[Mac95]{Mac}
I.~G. Macdonald.
\newblock {\em Symmetric functions and {H}all polynomials}.
\newblock Oxford Mathematical Monographs. The Clarendon Press, Oxford
  University Press, New York, second edition, 1995.
\newblock With contributions by A. Zelevinsky, Oxford Science Publications.

\bibitem[Mal93]{Mal}
Gunter Malle.
\newblock Generalized {D}eligne-{L}usztig characters.
\newblock {\em J. Algebra}, 159(1):64--97, 1993.

\bibitem[Mel18]{Me1}
Anton Mellit.
\newblock Integrality of {H}ausel-{L}etellier-{V}illegas kernels.
\newblock {\em Duke Math. J.}, 167(17):3171--3205, 2018.

\bibitem[Mel19]{Me2}
Anton Mellit.
\newblock Cell decompositions of character varieties, 2019.

\bibitem[Mel20a]{Me4}
Anton Mellit.
\newblock Poincar\'{e} polynomials of character varieties, {M}acdonald
  polynomials and affine {S}pringer fibers.
\newblock {\em Ann. of Math. (2)}, 192(1):165--228, 2020.

\bibitem[Mel20b]{Me3}
Anton Mellit.
\newblock Poincar\'{e} polynomials of moduli spaces of {H}iggs bundles and
  character varieties (no punctures).
\newblock {\em Invent. Math.}, 221(1):301--327, 2020.

\bibitem[MS22]{MS}
Davesh Maulik and Junliang Shen.
\newblock The ${P}={W}$ conjecture for ${\GL}_n$.
\newblock {\em arXiv}, 2022.

\bibitem[OS20]{OS}
Daniel Orr and Mark Shimozono.
\newblock Private communications.
\newblock 2020.

\bibitem[Sch16]{Sch}
Olivier Schiffmann.
\newblock Indecomposable vector bundles and stable {H}iggs bundles over smooth
  projective curves.
\newblock {\em Ann. of Math. (2)}, 183(1):297--362, 2016.

\bibitem[Ses77]{Ses77}
C.~S. Seshadri.
\newblock Geometric reductivity over arbitrary base.
\newblock {\em Advances in Math.}, 26(3):225--274, 1977.

\bibitem[Sho87]{S87}
Toshiaki Shoji.
\newblock Green functions of reductive groups over a finite field.
\newblock In {\em The {A}rcata {C}onference on {R}epresentations of {F}inite
  {G}roups ({A}rcata, {C}alif., 1986)}, volume~47 of {\em Proc. Sympos. Pure
  Math.}, pages 289--301. Amer. Math. Soc., Providence, RI, 1987.

\bibitem[Sho01]{S01}
Toshiaki Shoji.
\newblock Green functions associated to complex reflection groups.
\newblock {\em J. Algebra}, 245(2):650--694, 2001.

\bibitem[Shu22]{Shu2}
Cheng Shu.
\newblock The character table of $\text{GL}_n(q)\rtimes\lb\sigma\rb$.
\newblock {\em Advances in Math.}, 403, 2022.

\bibitem[Shu23]{Shu1}
Cheng Shu.
\newblock On character varieties with non-connected structure groups.
\newblock {\em J. Algebra}, 631:484--516, 2023.

\bibitem[Spr98]{Spr}
T.~A. Springer.
\newblock {\em Linear algebraic groups}, volume~9 of {\em Progress in
  Mathematics}.
\newblock Birkh\"{a}user Boston, Inc., Boston, MA, second edition, 1998.

\bibitem[Ste68]{St}
Robert Steinberg.
\newblock {\em Endomorphisms of linear algebraic groups}.
\newblock Memoirs of the American Mathematical Society, No. 80. American
  Mathematical Society, Providence, R.I., 1968.

\bibitem[Wal06]{W1}
J.-L. Waldspurger.
\newblock Le groupe {${\rm GL}_N$} tordu, sur un corps fini.
\newblock {\em Nagoya Math. J.}, 182:313--379, 2006.

\end{thebibliography}
\end{document}